\documentclass[10pt,a4paper]{article}
\usepackage[utf8]{inputenc}
\usepackage{amsmath}
\usepackage{amsfonts}
\usepackage{color}
\usepackage{amssymb}
\usepackage{mathrsfs}
\usepackage[cal=boondoxo]{mathalfa}
\usepackage{esint}
\usepackage[colorinlistoftodos]{todonotes}
\usepackage[colorlinks=true]{hyperref}
\hypersetup{
	colorlinks=true,%
	citecolor=red,%
	filecolor=black,%
	linkcolor=blue,%
	urlcolor=blue
}
\usepackage[margin=1.3in]{geometry}

\usepackage[amsmath,thmmarks,hyperref]{ntheorem}
{
	\theoremstyle{nonumberplain}
	\theoremheaderfont{\bfseries}
	\theorembodyfont{\normalfont}
	\theoremsymbol{\mbox{$\Box$}}
	\newtheorem{pf}{Proof.}
}

\numberwithin{equation}{section}
\def\R{\mathbb{R}}
\def\B{\mathbb{B}}

\def\S{\mathbb{S}}
\def\loc{\mathrm{loc}}
\def\supp{\mathrm{supp}}
\def\N{\mathbb{N}}

\def\e{\epsilon}
\newcommand{\ud}{\mathrm{d}}

\newtheorem{thm}{Theorem}[section]

\newtheorem{lem}{Lemma}[section]
\newtheorem{rem}{Remark}[section]

\newtheorem{pro}{\indent Proposition}[section]
\newtheorem{question}{Question.}[section]

\newtheorem*{thm A}{Theorem A}
\newtheorem*{thm B}{Theorem B}
\newtheorem*{thm C}{Theorem C}
\newtheorem{cor}{Corollary}[section]
\usepackage{upgreek}
\usepackage{graphicx}

\makeatletter

\newdimen\bibspace
\renewenvironment{thebibliography}[1]{%
	\section*{\refname 
		\@mkboth{\MakeUppercase\refname}{\MakeUppercase\refname}}%
	\list{\@biblabel{\@arabic\c@enumiv}}%
	{\settowidth\labelwidth{\@biblabel{#1}}%
		\leftmargin\labelwidth
		\advance\leftmargin\labelsep
		\itemsep\bibspace
		\parsep\z@skip     %
		\@openbib@code
		\usecounter{enumiv}%
		\let\p@enumiv\@empty
		\renewcommand\theenumiv{\@arabic\c@enumiv}}%
	\sloppy\clubpenalty4000\widowpenalty4000%
	\sfcode`\.\@m}
{\def\@noitemerr
	{\@latex@warning{Empty `thebibliography' environment}}%
	\endlist}

\makeatother

\makeatletter

\allowdisplaybreaks

\allowdisplaybreaks

\allowdisplaybreaks
\title{article}
\begin{document}
	\title{\bf 
Sharp weighted Carleman and Huber isoperimetric inequalities on the unit ball in higher dimensions} 
	\date{\today}
	\author{\medskip Zhijie Chen, Changfeng Gui, Shihong Zhang \\
		}
	
	\renewcommand{\thefootnote}{\fnsymbol{footnote}}
	\maketitle	
	
	{\noindent\small{\bf Abstract: In this paper, using a limiting approach, we establish a new type of weighted  Carleman inequality in all dimensions $n\geq 2$ and classify all extremal functions. In particular, when $n=2$, we prove that our inequality is equivalent to a sharp norm inequality in the Bergman space. In even dimensions, we further establish a sharp weighted Huber isoperimetric inequality on the unit ball, which generalizes Huber's original result \cite[Ann. Math., 1954]{Huber} and may be regarded as a sharp counterpart of Y. Wang's isoperimetric inequality in the unit ball \cite[Adv. Math., 2015]{Wang}.
    }

		\medskip 
		
		{{\bf $\mathbf{2020}$ MSC:} 35A23, 35B06, 53C21, 31B10}
		
		
		
		\section{Introduction}		
			Let $\B^2$ be the unit ball in $\R^2$. The classical  Carleman inequality \cite{Carleman} states that, for any smooth subharmonic function $F$ on $\B^2$ with boundary value $F|_{\S^1}=f$, one has the isoperimetric inequality
		\begin{align}\label{Intro Caleman}
			\frac{1}{\pi}\int_{\B^2} e^{2F}
			\leq
			\left(\int_{\S^1} e^f \,\ud\mu\right)^2.
		\end{align}
		Here $\ud\mu$ denotes the normalized measure on $\S^1$, that is,
		$
		\int_{\S^1}\ud\mu=1.
		$
		Moreover, the equality holds if and only if $F$ is harmonic and either $f=c$ or
		$
		f(\eta)=\log\frac{1-|a|^2}{|\eta-a|^2}+c,
		$
		where $c\in\R$ and $a\in\B^2$. By the Riemann mapping theorem and the conformal invariance of \eqref{Intro Caleman}, the inequality \eqref{Intro Caleman} is in fact valid for any simply connected domain in $\R^2$.
		
		A natural question is whether a similar isoperimetric inequality still holds for functions that are not subharmonic. This leads to another more delicate isoperimetric inequality, the Huber inequality \cite{Huber}, which states that if
		$
		F\in C^2(\B^2)\cap C(\overline{\B^2})
		$
		and
		$
		\int_{\B^2}(-\Delta F(\xi))^{+}\,\ud \xi<2\pi,
		$
		then one has the following sharp isoperimetric inequality:
		\begin{align}\label{Intro Huber}    \int_{\B^2}e^{2F(\xi)}\ud \xi\leq \frac{2\pi^2}{2\pi-\int_{\B^2}(-\Delta F(\xi))^{+}\ud \xi}\left(\int_{\S^1}e^f\ud \mu\right)^2.
		\end{align}
		Moreover, if $F$ can be written as the difference of two subharmonic functions, then the equality in \eqref{Intro Huber} can be attained, and it holds if and only if $F$ takes the form \eqref{Intro Equ-Hu}. The importance of Huber's inequality may be viewed from two perspectives: analysis and geometry.
		
		The first is its significance in analysis and PDE. One of the main reasons why Huber's inequality is so useful is that it yields the (possibly singular) Bol inequality; for recent progress in the whole-space setting, we refer to \cite{Li&Wei}. The Bol inequality plays a fundamental role in the study of mean field type equations, since it provides a powerful comparison principle between the interior mass and the boundary length of level sets. In particular, it has become an important tool in establishing uniqueness, symmetry, and a priori estimates for  mean field equations; see, for instance, the works of Lin et al.~\cite{Lin-2,Lin-3,Lin-1}. Especially in recent years, the second author and Moradifam \cite{Gui Moradifam} developed the sphere covering method, which has led to substantial progress in the study of mean field equations. Their method relies in an essential way on Bol-type inequalities and related isoperimetric arguments. For the singular case, we also refer to \cite{Gui Moradifam-2}. These developments further illustrate that Huber's inequality is not only a delicate refinement of Carleman's inequality, but also a basic tool in the modern analysis of nonlinear elliptic equations with exponential nonlinearities.
		
		The second motivation comes from differential geometry. Fiala \cite{Fiala} and Huber \cite{Huber-2} generalized inequalities of weighted Huber from planar domains to simply connected surfaces with variable Gaussian curvature. A natural question in conformal geometry is whether there exists a higher-dimensional analogue. Recently, Y. Wang \cite{Wang} introduced Branson's $Q$-curvature as a higher-dimensional substitute for the Gaussian curvature and established a higher-order Fiala--weighted Huber isoperimetric inequality under integral curvature assumptions, although without a sharp isoperimetric constant. More precisely, she proved that 
		
		\begin{thm A*}[\protect{Y. Wang \cite[Theorem 1.1]{Wang}}]    Let $(\R^n, e^{2u}|\ud x|^2)$ be an even-dimensional complete manifold with the $Q$-curvature $Q_g\in L^1(\R^n, \ud V_g)$. If $(\R^n, e^{2u}|\ud x|^2)$ is  normal, i.e.,
			\begin{align*}
				u(x)=\frac{2}{(n-1)!|\S^n|}\int_{\R^n}\log\frac{|y|}{|x-y|}Q_g(y)e^{nu(y)}\ud y+C,
			\end{align*}
			and if 
			\begin{align*}
				\int_{\R^n}Q_g^{+}\ud V_g<k_n:=\frac{(n-1)!|\S^n|}{2},
			\end{align*}
			where $Q_g^{+}$ denotes the positive part of $Q_g$. Then there exists $C(n, \|Q_g\|_{\R^n,L^1(\ud V_g)})$ such that for any bounded smooth domain $\Omega$ in $\R^n$,
			\begin{align}\label{Wang Isoperi}
				|\Omega|_g\leq C(n, \|Q_g\|_{\R^n,L^1(\ud V_g)})|\partial \Omega|_g^{\frac{n}{n-1}}.
			\end{align}
		\end{thm A*}
		
		Under the same assumptions, Y.~Sire and Y.~Wang \cite{Wang&Yannick,Wang&Yannick-2} established several fractional Poincar\'e inequalities on conformally flat manifolds. Very recently, the third author proposed in \cite{S.Zhang 1} a sharp assumption involving only the $Q$-curvature, which replaces the normality condition.
		
		However, in higher dimensions, results of this type still appear to be extremely scarce in the existing literature. While the classical  Carleman inequality and its conformally invariant higher-dimensional extensions have been extensively studied, there seem to be very few works addressing genuinely weighted Huber inequalities beyond the two-dimensional setting, especially in a sharp form and with a precise characterization of the extremals. Very recently, the second author posed the following question: whether there exist isoperimetric inequalities in higher dimensions under suitable curvature conditions.
		
		In this paper, we aim to give a partial answer to this question in the conformally invariant setting of the unit ball. More precisely, starting from the sharp polynomial inequalities of type \eqref{Intro HLS}, we derive a new family of sharp exponential inequalities by a limiting procedure, and then show that these inequalities still enjoy a nontrivial conformal invariance despite the presence of the complicated weight $\ud \nu$. Compared with the existing literature, one distinctive feature of our result is that the limiting inequality carries a highly nontrivial weight, which disappears only in special cases and seems not to have appeared previously when $n$ is odd or $\sigma\neq 1$. Another important novelty is that we are able to classify all extremal functions in $L^{\infty}(\S^{n-1})$, which appears to be unavailable even in some earlier special cases. In particular, our results provide a rather rare class of sharp higher-dimensional weighted Huber/Carleman-type inequalities together with conformal invariance, explicit extremals, and new connections to complex analysis in dimension two and adapted metrics in conformal geometry for even dimensions.
		
		\subsection{High dimensional weighted  Carleman inequality}
		In higher dimensions, Hang--Wang--Yan \cite{Hang&Wang&Yan} generalized Carleman's inequality by using the harmonic extension operator
		\begin{align*}
			P_0(f)(\xi)
			=
			\int_{\S^{n-1}}
			\frac{1-|\xi|^2}{|\xi-\eta|^{n}}f(\eta)
			\,\ud\mu(\eta),
		\end{align*}
		where $\ud\mu$ denotes the normalized measure on $\S^{n-1}$, that is,
		$
		\int_{\S^{n-1}}\ud\mu=1.
		$
		Given a positive continuous function $f$, the function $P_0(f)$ is the harmonic extension of $f$, satisfying
		$
		P_0(f)\big|_{\partial\B^n}=f.
		$
		Consider the conformal metric
		$
		g=(P_0(f))^{\frac{4}{n-2}}g_{\R^n}.
		$
		Then
		\begin{align*}
			\mathrm{Vol}_g(\B^n)
			=
			\int_{\B^n}(P_0(f))^{\frac{2n}{n-2}}\,\ud \xi,
			\qquad
			\mathrm{Area}_g(\S^{n-1})
			=
			\int_{\S^{n-1}}f^{\frac{2(n-1)}{n-2}}\,\ud\mu.
		\end{align*}
		 Hang--Wang--Yan \cite{Hang&Wang&Yan} proved the following beautiful isoperimetric inequality: for any
		$f\in L^{\frac{2(n-1)}{n-2}}(\S^{n-1})$,
		\begin{align*}
			\|P_0(f)\|_{L^{\frac{2n}{n-2}}(\B^n)}
			\le
			n^{-\frac{n-2}{2(n-1)}}
			|\B^n|^{-\frac{n-2}{2n(n-1)}}
			\|f\|_{L^{\frac{2(n-1)}{n-2}}(\S^{n-1})}.
		\end{align*}
		Afterwards, they also studied isoperimetric inequalities on general Riemannian manifolds with boundary  in \cite{Hang&Wang&Yan-2}.
		
		In fact, there is a more general class of conformally invariant geometric inequalities that extends the Hang--Wang--Yan inequality. To introduce it, we set
		\begin{align*}
			p=p(\mathbf{a},\mathbf{b})=\frac{2(n-1)}{n+\mathbf{a}-2},
			\qquad
			q=q(\mathbf{a},\mathbf{b})=\frac{2n}{n-\mathbf{a}-2\mathbf{b}},
		\end{align*}
		where $\mathbf{a},\mathbf{b}$ satisfy the constraints
		\begin{align}\label{Index Condi}
			\mathbf{b} \geq 0,
			\qquad
			0<\mathbf{a}+\mathbf{b}<n-\mathbf{b},
			\qquad\text{and}\qquad
			\frac{n-\mathbf{a}-2\mathbf{b}}{2n}+\frac{n-\mathbf{a}}{2(n-1)}<1.
		\end{align}
		Given $f\in L^p(\S^{n-1})$, we consider the integral operator $Q_{\mathbf{a},\mathbf{b}}$ defined by
		\begin{align}\label{Def Q}
		    Q_{\mathbf{a},\mathbf{b}}(f)(\xi)
		=
		\int_{\S^{n-1}} H(\xi,\eta)\,f(\eta)\,\ud\mu(\eta),
		\end{align}
		where the kernel $H(\xi,\eta)$ is given by
		\begin{align*}
			H(\xi,\eta)
			=
			\left(\frac{1-|\xi|^2}{2}\right)^{\mathbf{b}}
			|\xi-\eta|^{\mathbf{a}-n}.
		\end{align*}

		\begin{thm B*}\protect{\cite[Theorem 1.1 and Theorem 1.3]{Gluck}}\label{thm B}
			Let $n\geq 2$ and  suppose $(\mathbf{a},\mathbf{b})$ satisfy \eqref{Index Condi}. Then 
			\begin{align}\label{Intro HLS}
				\|Q_{\mathbf{a},\mathbf{b}}(f)\|_{L^{q}(\B^{n})}\leq c_{\mathbf{a},\mathbf{b}}\|f\|_{L^{p}(\S^{n-1},\ud \mu) },
			\end{align}
			where the sharp constant $c_{\mathbf{a},\mathbf{b}}$  takes the form 
			\begin{align}\label{Intro const}
				c_{\mathbf{a},\mathbf{b}}
				=
				\left\|
				\int_{\S^{n-1}}H(\cdot,\eta)\,
				\ud \mu(\eta)
				\right\|_{L^q(\B^n)}.
			\end{align}
			Moreover, the equality holds if and only if, up to a multiplicative constant,
			$f=\bigl(\det \ud\Psi\big|_{\partial \B^{n}} \bigr)^{\frac{1}{p}}$,
			where $\Psi$ is a conformal transformation from $\B^n$ to $\B^n$.

		\end{thm B*}
		
		A number of important special cases of \eqref{Intro HLS} had been established before the full theorem became available. When $(\mathbf{a},\mathbf{b})=(0,1)$, one recovers the isoperimetric inequality of Hang--Wang--Yan \cite{Hang&Wang&Yan}. The critical line $\mathbf{a}+\mathbf{b}=1$, with $\mathbf{a}\in(2-n,1)$, was treated by S.~Chen \cite{Chen Shibing}. Later, Dou and Zhu \cite{Dou&Zhu} handled the case $\mathbf{b}=0$ with $\mathbf{a}\in(1,n)$. For the range $\mathbf{b}=1$ and $\mathbf{a}\in[2,n)$, Dou, Guo, and Zhu \cite{Dou&Guo&Zhu} introduced a subcritical method and proved the corresponding sharp inequality. The unresolved parameter range was subsequently completed by Gluck \cite{Gluck}, who extended the subcritical strategy to cover the full generality of \eqref{Intro HLS}. More recently, the third author \cite{Yang&Zhang} obtained sharp quantitative versions of \eqref{Intro HLS} through a delicate analysis of the underlying hypergeometric functions.
		
		An interesting observation of S.~Chen \cite{Chen Shibing} is that, by letting $\mathbf{a}\to 2-n$, one obtains an exponential isoperimetric-type inequality on the unit ball $\B^n$, see also \cite{Tian}: for any $f\in L^{\infty}(\S^{n-1})$, there holds
		\begin{align*}
			\|e^{\tilde{I}_n+P(f)}\|_{L^n(\B^n)}
			\leq
			\|e^{\tilde{I}_n}\|_{L^n(\B^n)}
			\|e^{f}\|_{L^{n-1}(\S^{n-1},\ud \mu)},
		\end{align*}
		where
		\begin{align*}
			\tilde{I}_n(|\xi|)
			=
			\frac{\ud }{\ud \mathbf{a}}\Big|_{\mathbf{a}=2-n}
			\frac{4\Gamma\left(\frac{n-\mathbf{a}}{2}\right)\sqrt{\pi}}
			{\Gamma\left(\frac{1-\mathbf{a}}{2}\right)\Gamma\left(\frac{n}{2}\right)}
			Q_{\mathbf{a},1-\mathbf{a}}(1)(\xi),
		\end{align*}
		where $Q_{\mathbf{a},1-\mathbf{a}}(1)(\xi)$ is given by \eqref{Intro equ-1} below and $\tilde{I}_n(1)=0$,
		and $P(f)$ denotes the hyperbolic harmonic extension
		\begin{align*}
		P(f)(\xi)
		=
		\int_{\S^{n-1}}
		\left(\frac{1-|\xi|^2}{|\xi-\eta|^2}\right)^{n-1}
		f(\eta)\,
		\ud \mu(\eta).
		\end{align*}
		More precisely, if we consider the standard hyperbolic metric
		$	g_{\B}
		=
		\left(\frac{2}{1-|\xi|^2}\right)^2|\ud \xi|^2,
		$
		then $P(f)$ satisfies
		$$
		\Delta_{\B}P(f)=0\quad\mathrm{in}\quad \B^n \qquad\mathrm{and}\qquad P(f)=f\quad\mathrm{on}\quad \S^{n-1}.
		$$
		\textit{
More importantly, the hyperbolic harmonic extension $P(f)$ can also be viewed as a polyharmonic extension with higher-order conformally invariant boundary data vanishing; see Subsection \ref{Sec 5.1}.
}

		S.~Chen's approach relies on the constraint $\mathbf{a}+\mathbf{b}=1$, which appears rather artificial in the general setting \eqref{Intro HLS}. A very natural question is therefore whether one can remove this constraint and take suitable limits in $\mathbf{a}$ and $\mathbf{b}$, here we choose a two-parameter path for which both exponents $p(\mathbf{a},\mathbf{b})$ and $q(\mathbf{a},\mathbf{b})$ tend to infinity, while the ratio $q/p$ converges to the prescribed exponent. We also emphasize that the resulting inequality is conformally invariant under the full conformal group of $\S^{n-1}$, although this invariance is far from obvious from its formulation.

		\begin{thm}\label{Thm 1}
		For any $f\in L^{1}(\S^{n-1})$ with $n\geq 2$ and $\sigma\in\left(\frac{1}{2},\frac{2n-1}{2(n-1)}\right)$, we have 
		\begin{align}\label{Intro Inequ}
			\|e^{P(f)}\|_{L^{\frac{n}{2\sigma-1}}(\B^n,\ud \nu)}	\leq \|e^{f}\|_{L^{n-1}(\S^{n-1},\ud \mu)},
		\end{align}
		where 
		\begin{align*}
			\ud \nu=\frac{(1-|\xi|^2)^{\frac{2n(1-\sigma)}{2\sigma-1}}e^{\frac{2n}{2\sigma-1}I(|\xi|^2)}\ud \xi}{\int_{\B^n}(1-|\xi|^2)^{\frac{2n(1-\sigma)}{2\sigma-1}}e^{\frac{2n}{2\sigma-1}I(|\xi|^2)}\ud \xi },\qquad  I(x)= \sum_{k=1}^{+\infty}\frac{\left(1-\frac{n}{2}\right)_k}{\left(\frac{n}{2}\right)_k}\frac{x^{k}}{2k}.
		\end{align*}
		Moreover, if $f\in L^{\infty}(\S^{n-1})$, then the equality holds if and only if $f=c$ or $f(\eta)=\log\frac{1-|a|^2}{|\eta-a|^2 }+c$ for some $a\in \B^n$ and $c\in \R$.
		\end{thm}
		\begin{rem}
		Due to the presence of the $L^{\frac{n}{2\sigma-1}}$ norm on the left-hand side, the condition $\sigma>\frac{1}{2}$ is completely natural. Moreover, the range of $\sigma$ is sharp in the following sense:
		\begin{align*}
			(1-|\xi|^2)^{\frac{2n(1-\sigma)}{2\sigma-1}}\in L^1(\B^n)
			\qquad\Longleftrightarrow\qquad
			\sigma<\frac{2n-1}{2(n-1)}.
		\end{align*}
		When $\sigma=1$, our result recovers the inequality of S.~Chen \cite{Chen Shibing}. Moreover, it seems that the extremal functions were not classified in \cite{Chen Shibing}; here we classify all extremal functions in $L^{\infty}(\S^{n-1})$.
		\end{rem}
		\begin{rem}
		In addition, one can verify that
		$
		\tilde{I}_n(|\xi|)=2I(|\xi|^2)-2I(1).
		$
		One advantage of this representation, compared with the formulation in S. Chen \cite{Chen Shibing}, is that it yields an explicit series expansion for $I$. More precisely, when $n$ is even, the series terminates and $I$ becomes a finite sum:
		\begin{align*}
			I(|\xi|^2)
			=
			\sum_{k=1}^{\frac{n}{2}-1}
			\frac{\left(1-\frac{n}{2}\right)_k}{\left(\frac{n}{2}\right)_k}
			\frac{|\xi|^{2k}}{2k}.
		\end{align*}
		In the special case $\sigma=1$ and $n=4$, we have
		$
		I(|\xi|^2)=-\frac{|\xi|^2}{4},
		$
		and hence
		$
		\Tilde{I}_4(\xi)=\frac{1-|\xi|^2}{2}.
		$
		This leads to a natural geometric interpretation through the so-called adapted metric
		\[
		g^{*}
		=
		e^{2\Tilde{I}_4(\xi)}|\ud \xi|^2
		=
		e^{1-|\xi|^2}|\ud \xi|^2.
		\]
		For more geometric information, see the elegant work of Ache and Chang \cite{Ache&Chang}. Roughly speaking, the adapted metric $g^*$ endows the compact manifold with boundary $(\B^4,\S^3)$ with several remarkable geometric properties:
		\begin{itemize}
			\item $\S^3$ is totally geodesic with respect to $g^*$;
			\item the $Q$-curvature of $g^*$ vanishes;
			\item the scalar curvature satisfies $R_{g^*}>0$.
		\end{itemize}
		Thus, in dimension four, the function $I$ is not merely an analytic quantity arising from a limiting procedure; it also admits a clear geometric interpretation through the adapted metric. In fact, for general even $n$ and $\sigma=1$, this explicit formula for the adapted metric already appeared in the work of Q.~Yang \cite{yang}. The weight in our inequality is exactly the conformal density that restores the conformal invariance of the left-hand side.
		
		However, when $n$ is odd or $\sigma\neq 1$, this more general type of weight does not seem to appear in the current literature, and, to the best of our knowledge, no satisfactory geometric interpretation is currently available. On the other hand, in the special case $n=2$, the extra term $I$ disappears, and the remaining weight
		\[
		\ud \nu
		=
		\frac{(1-|\xi|^2)^{\frac{4(1-\sigma)}{2\sigma-1}}\ud \xi}
		{\int_{\B^2}(1-|\xi|^2)^{\frac{4(1-\sigma)}{2\sigma-1}}\ud \xi}
		\]
		is exactly the weight appearing in Bergman spaces in complex analysis; see the discussion below.
		\end{rem}
		
		The main difference between our inequality and S.~Chen's inequality is that our inequality contains the weight
		$
		(1-|\xi|^2)^{\frac{2n(1-\sigma)}{2\sigma-1}},
		$
		which disappears when $\sigma=1$. This weight appears naturally for the following reason. For $(\mathbf{a},\mathbf{b})$ satisfying \eqref{Index Condi},  the formula (see \cite[Theorem 4.1]{Gong&Yang&Zhang} or \cite[(2.21)]{Yang&Zhang}) gives
		\begin{align}\label{Intro equ-1}
		Q_{\mathbf{a},\mathbf{b}}(1)(\xi)
		=
		\frac{\pi^{\frac{n}{2}}2^{1-\mathbf{b}}}{\Gamma(\frac{n}{2})|\S^{n-1}|}
		(1-|\xi|^2)^{\mathbf{a}+\mathbf{b}-1}
		{}_2F_1\left(\frac{n+\mathbf{a}}{2}-1,\frac{\mathbf{a}}{2}; \frac{n}{2}; |\xi|^2\right).
		\end{align}
		In Chen's inequality, the constraint is $\mathbf{a}+\mathbf{b}=1$, and therefore
		$
		(1-|\xi|^2)^{\mathbf{a}+\mathbf{b}-1}\equiv 1,
		$
		so this factor makes no contribution when differentiating with respect to $\mathbf{a}$. In our setting, however, this term is no longer constant, and it produces the additional weight
		$
		(1-|\xi|^2)^{\frac{2n(1-\sigma)}{2\sigma-1}}.
		$
		See the proof of Theorem \ref{Sec 2 thm} for details. When $\sigma<1$, one has $\frac{n}{2\sigma-1}>n$, and hence the left-hand side is measured in a stronger Lebesgue space. Meanwhile, the exponent of the weight $(1-|\xi|^2)$ is positive, so the weighted norm enforces a certain decay near the boundary. In contrast, when $\sigma>1$, one has $\frac{n}{2\sigma-1}<n$, and hence the left-hand side is measured in a weaker Lebesgue space. Meanwhile, the exponent of the weight $(1-|\xi|^2)$ is negative, so the weighted norm allows certain boundary singularities.
		
		As mentioned above, when $n=2$, this weight is closely related to the Bergman space, which indicates that its appearance here is highly nontrivial. To better illustrate our inequality and its connection with the Bergman space, we first give a brief introduction to this space.

\vspace{0.2cm}
		
		\textbf{Bergman space}:	Let $\alpha>1$ and $0<p<\infty$. We define the Bergman space $A^p_{\alpha}(\B^2)$ to be the class of holomorphic functions $F$ on the unit disc
	$
	\B^2=\{z\in \R^2:\ |z|<1\},
	$
	equipped with the norm
	\begin{align*}
		\|F\|_{A^p_{\alpha}(\B^2)}
		=
		\left(
		\int_{\B^2}
		|F(z)|^p
		(\alpha-1)(1-|z|^2)^{\alpha-2}
		\,\ud A
		\right)^{\frac{1}{p}}.
	\end{align*}
	Here $\ud A$ is the normalized measure on $\B^2$ such that
	$
	\int_{\B^2}\ud A=1.
	$
For more information on Bergman spaces, we refer to \cite{Duren,Zhu}.

	The Hardy space $H^p(\B^2)$, for $0<p<\infty$, consists of all holomorphic functions $F$ satisfying
	\begin{align*}
		\sup_{0\le r<1}
		\left(
		\frac{1}{2\pi}
		\int_{0}^{2\pi}
		|F(re^{i\theta})|^p
		\,\ud\theta
		\right)^{1/p}
		<+\infty.
	\end{align*}
	If $p\ge 1$, then the above quantity defines a norm on $H^p(\B^2)$. For $0<p\le q<+\infty$, it is clear that
	$
	H^q(\B^2)\subset H^p(\B^2).
	$
	For $p\ge 1$, the space $H^p(\B^2)$ can be viewed as a subspace of $L^p(\S^1)$ in the following sense: for any $F\in H^p(\B^2)$, define its radial limit by
	\begin{align*}
		\tilde{F}(e^{i\theta})
		=
		\lim_{r\to 1}F(re^{i\theta})
		\qquad \text{for a.e. } \theta\in \S^1.
	\end{align*}
	Then
	$
	\|\tilde{F}\|_{L^p(\S^1)}=\|F\|_{H^p(\B^2)}.
	$ More importantly, the Hardy space may be regarded as the limiting case of the Bergman space as $\alpha\to 1^+$, namely,
	\begin{align*}
		\lim_{\alpha\to 1^{+}}
		\|F\|_{A^p_{\alpha}(\B^2)}
		=
		\|F\|_{H^p(\B^2)}.
	\end{align*}
	
	An interesting problem concerning Bergman spaces, proposed in \cite[p.~688]{Bayart}, is the following.

\begin{question}\label{Intro Que}
	For any $0<p\leq q$ and $\alpha,\beta\geq 1$ such that $q/\beta \leq p/\alpha$, does the contractive inequality
	\begin{align*}
		\|F\|_{A^q_{\beta}(\B^2)}\leq \|F\|_{A^p_{\alpha}(\B^2)}
	\end{align*}
	hold?
\end{question}

Concerning this question, Carleman's inequality \eqref{Intro Caleman} is in fact equivalent to
\begin{align*}
	\|F\|_{A^2_2(\B^2)}
	\leq
	\|F\|_{H^1(\B^2)}.
\end{align*}
Burbea \cite{Burbea,Burbea-1,Burbea-3} generalized Carleman's inequality by proving that, for every $0<p<+\infty$ and every non-negative integer $n$, one has
\begin{align}\label{Intro Burbea}
	\|F\|_{A^{p(n+1)}_{n+1}(\B^2)}
	\leq
	\|F\|_{H^p(\B^2)}.
\end{align}
Recently, Bayart et al.~\cite{Bayart} proved the following theorem by a PDE approach, which provides a partial answer to this question.

\begin{thm C} [\protect{\cite[Theorem 1]{Bayart}}]  \label{Thm A-ineq}
	For $\alpha\geq \alpha_0=\frac{1+\sqrt{17}}{4}$ and $0<p<+\infty$, every $F\in A^p_{\alpha}(\B^2)$ satisfies
	\begin{align*}
		\|F\|_{A^{p(\alpha+1)/\alpha}_{\alpha+1}(\B^2)}\leq \|F\|_{A^{p}_{\alpha}(\B^2)}.
	\end{align*}
	Moreover, if $\alpha>\alpha_0$, then the equality holds if and only if $F$ is of the form
	\begin{align*}
		F(z)=\frac{C}{(1-z\bar{w})^{2\alpha/p}} \qquad\mathrm{for}\qquad C\in \mathbb{C}, ~w\in \B^2.
	\end{align*}
\end{thm C}

The weight in our inequality is precisely the weight appearing in the norm defining the Bergman space $A^{p\alpha}_{\alpha}$, with
	$\alpha=2/(2\sigma-1)$.
In the case $n=2$, our inequality is essentially equivalent to the following theorem, which may be viewed as a generalization of Burbea's inequality \eqref{Intro Burbea}; see \cite{Burbea-1}.

\begin{thm}\label{Thm 2}
	For $\alpha > 1$ and $p\geq 1$, if $F\in H^{p}(\B^2)$, then
	\begin{align*}
		\|F\|_{A^{p\alpha}_{\alpha}(\B^2)}\leq \|F\|_{H^{p}(\B^2)}.
	\end{align*}
	Moreover, the equality holds if and only if
	\begin{align*}
		F(z)=\frac{C}{(1-z\bar{a})^{2/p}}\qquad\mathrm{for}\qquad C\in\mathbb C, ~ a\in \B^2.
	\end{align*}
\end{thm}
\subsection{High dimensional Huber inequality}

Next, our motivations are twofold. First, we aim to establish a sharp version of Wang's isoperimetric inequality \eqref{Wang Isoperi} in the unit ball. Second, we hope to generalize the original Huber inequality \eqref{Intro Huber} to higher dimensions.

Let $n=2m\ge 2$, and consider the compact manifold $(\B^n,e^{2F}|\ud \xi|^2)$. By the definition of the $Q$-curvature, we have
\begin{align*}
    Q_g=e^{-nF}(-\Delta)^mF.
\end{align*}
To establish a higher-dimensional weighted Huber isoperimetric inequality on $\B^n$, a basic observation is that one cannot directly control the ``higher-order derivatives'' of $F$ on the boundary $\S^{n-1}$. Therefore, one is naturally led to restrict $F$ to a function space in which these higher-order boundary terms vanish. The most suitable choice is to impose conditions in terms of the higher-order conformally invariant boundary operators $B^{2m-1}_j$, and accordingly define the space
\begin{align*}
     \mathcal H_f
	:=
	\left\{
	u\in C^n( \overline{\B^n}) :
	B^{2m-1}_0u=f,\ B^{2m-1}_1u=\cdots=B^{2m-1}_{m-1}u=0
	\right\}.
    \end{align*}
    For $u\in \mathcal{H}_{f}$, the conformally invariant boundary operators $B^{2m-1}_j$ may be chosen as
\begin{align*}
	B^{2m-1}_0 u:=u|_{\S^{n-1}},
	\qquad
	B^{2m-1}_j u:=\partial_\nu^j u-T_j(u|_{\S^{n-1}}),
	\qquad j=1,\dots,m-1,
	\end{align*}
where $T_j$ is an intrinsic differential operator on $\S^{n-1}$; see \eqref{Sec 5 T_j}. 

For $m\geq 1$ and $\xi, \zeta\in \B^n$, let $\Psi_{\zeta}:\xi\mapsto \Psi_{\zeta}(\xi)$ denote the conformal transformation from $\B^n$ onto itself;
\begin{align*}
    \Psi_{\zeta}(\xi)
	=
	\frac{(1-|\zeta|^2)(\xi-\zeta)-|\xi-\zeta|^2\zeta}{1-2\xi\cdot \zeta+|\xi|^2|\zeta|^2},
    \end{align*}
for more details on $\Psi_{\zeta}$, see Section \ref{Sec 3}. Then the Green function of $(-\Delta)^m$ with Dirichlet boundary conditions is given by the Boggio's formula  (see \cite{Boggio} or  \cite[P. 50]{GGS}), and it can be rewritten as
    \begin{align}
        G_n(\xi,\zeta)=
		k_n^{-1}\left(
		\log \frac1{|\Psi_\zeta(\xi)|}
		+
		J_m(|\Psi_\zeta(\xi)|^2)\right),
    \end{align}
    where $k_n=(n-1)!|\S^n|/2=2^{n-2}\Gamma^2(n/2)|\S^{n-1}|$ (see Theorem A) and
    \begin{align*}
      J_m(t):=\frac12\int_t^1 \frac{(1-s)^{m-1}-1}{s}\ud s.
    \end{align*}
For the precise derivation, see Section \ref{Sec 5}. Our next result establishes a sharp weighted Huber isoperimetric inequality in higher dimensions. Since all of our assumptions are conformally invariant, the theorem also remains valid on domains that are conformally equivalent to the unit ball.
    \begin{thm}\label{Thm 3}
	Let $n=2m\ge 2$ and consider the compact manifold $(\B^n, e^{2F}|\ud \xi|^2)$ with boundary $\S^{n-1}$. If $F\in \mathcal{H}_f$ satisfies
	\begin{align}\label{Intro Integral Cond}
	    \gamma:=\int_{\B^n}Q_g^{+}\ud V_g=\int_{\B^n}\left((-\Delta)^mF\right)^+d\xi<(2\sigma-1)k_n,
	\end{align}
where $\sigma\in\left(\frac{1}{2},\frac{2n-1}{2(n-1)}\right)$, then we have the isoperimetric inequality
	\begin{align}
\|e^{F}\|_{L^{\frac{n}{2\sigma-1}}(\B^n,\ud \nu)}	\leq M_{\sigma,\beta;I,J_m} \|e^{f}\|_{L^{n-1}(\S^{n-1},\ud \mu)},
		\label{Intro:Huber-final}
	\end{align}
	where $\ud \nu$ is given in Theorem \ref{Thm 1}, $
	\beta:=\frac{n\gamma}{(2\sigma-1)k_n}
	$ and
	\[
	M_{\sigma,\beta;I,J_m}
	=
	\left(\frac{
		\displaystyle\int_0^1
		r^{n-1-\beta}(1-r^2)^{\frac{2n(1-\sigma)}{2\sigma-1}}
		e^{\frac{2n}{2\sigma-1} I(r^2)+\beta J_m(r^2)}\ud r
	}{
		\displaystyle\int_0^1
		r^{n-1}(1-r^2)^{\frac{2n(1-\sigma)}{2\sigma-1}}
		e^{\frac{2n}{2\sigma-1}  I(r^2)}\ud r
	}\right)^{\frac{2\sigma-1}{n}}.
	\]	
\end{thm}
\begin{rem}
    The assumption \eqref{Intro Integral Cond} is equivalent to $\beta<n$, and this condition is sharp for ensuring the finiteness of the sharp constant $M_{\sigma,\beta;I,J_m}$. Moreover, the assumption $F\in \mathcal{H}_f$ is also essential and cannot be removed; see Lemma \ref{Sec 5 Lem-2}.
\end{rem}
\begin{rem}
    If $\gamma=0$, then $M_{\sigma,\beta;I,J_m}=1$ and \eqref{Intro:Huber-final} reduces to the sharp inequality \eqref{Intro Inequ}. In particular, in the smooth class, the equality in \eqref{Intro:Huber-final} can hold only if $\gamma=0$ and $F=P (f_{a,c})$, with $f_{a,c}=\log\frac{1-|a|^2}{|\eta-a|^2 }+c$ or $c$ being an extremal of \eqref{Intro Inequ}. For $\gamma>0$, the equality cannot occur in the smooth class.
	In the enlarged class of functions allowing one interior singularity, for example,
    \begin{align*}
        F={P}(f)+\int_{\B^n} G_n(\xi,\zeta)\ud \hat{\mu}(\zeta)\quad\mathrm{where}~ \ud \hat{\mu} \mathrm{~is~ a~ Radon~ measure},    \end{align*}
    the equality is attained by $\ud \hat{\mu}=\gamma\delta_{a}$,  and 
	\begin{align}\label{Intro Equ-Hu}
	    F=F_{a,c}(\xi)
	=
	P(f_{a,c})(\xi)
	+
	\gamma\,G_n(\xi,a).	    
    \end{align}
     Approximating $\gamma\delta_{a}$ by smooth nonnegative densities yields a sequence of admissible smooth functions for which the quotient in \eqref{Intro:Huber-final} converges to $M_{\sigma,\beta;I,J_m}$. Hence $M_{\sigma,\beta;I,J_m}$ is also the sharp constant for the case $\gamma>0$.
\end{rem}

For $n=2$, we also obtain a version of the inequality on simply connected domains, which generalizes the classical Huber inequality. To state it, we first recall the notion of conformal radius. Let \(\Omega\subset\mathbb C\) be simply connected and \(x\in\Omega\). The conformal radius of \(\Omega\) at \(x\) is defined by
\[
	r_\Omega(x):=\frac{1}{|\psi'(x)|},
\]
where \(\psi:\Omega\to\mathbb D\) is a conformal map satisfying \(\psi(x)=0\); for more details, see \cite[p.~3]{P. Nicolas and S. Nikos}. Let \(\Psi:\mathbb D\to\Omega\) be conformal and write \(x=\Psi(\xi)\). Set \(\phi=\Psi^{-1}\). Since \(\phi(x)=\xi\), we compose \(\phi\) with the disk automorphism
\[
	T_\xi(w)=\frac{w-\xi}{1-\overline{\xi}w},
\]
so that \(\psi:=T_\xi\circ\phi\) maps \(x\) to \(0\). Hence
\[
	r_\Omega(x)
	=
	\frac{1}{|\psi'(x)|}
	=
	\frac{1}{|T_\xi'(\xi)|\,|\phi'(x)|}.
\]
Since
\[
	|T_\xi'(\xi)|=\frac{1}{1-|\xi|^2}
	\quad\text{and}\quad
	|\phi'(x)|=\frac{1}{|\Psi'(\xi)|},
\]
hence for fixed  conformal map 
\(\Psi:\mathbb D\to\Omega\), one has
\[
    r_\Omega(\Psi(\xi))=(1-|\xi|^2)|\Psi'(\xi)|,
    \qquad \xi\in\mathbb D.
\]
This formula is independent of the choice of \(\Psi\), since two such maps differ by an automorphism of \(\mathbb D\).
\begin{cor}\label{cor:Huber_domain}
	Let $\Omega\subset\mathbb C$ be a bounded simply connected $C^1$ domain, and let $\sigma\in\left(\frac12,\frac32\right)$. For $F\in C^2(\Omega)\cap C(\overline\Omega)$, suppose
	\[
	\gamma:=\int_\Omega(-\Delta F(x))^+\ud x<2\pi(2\sigma-1).
	\]
	Then
	\[
	\left(
	\frac{3-2\sigma}{\pi(2\sigma-1)}
	\int_\Omega e^{\frac{2}{2\sigma-1}F(x)}r^{\frac{4(1-\sigma)}{2\sigma-1}}_\Omega(x)\ud x
	\right)^{\frac{2\sigma-1}{2}}
	\le
	\frac{c_\sigma}{2\pi}\int_{\partial\Omega}e^F\ud s,
	\]
	where
	\[
	c_\sigma=
	\left[
	\frac{3-2\sigma}{2\sigma-1}
	B\left(
	1-\frac{\gamma}{2\pi(2\sigma-1)},\,
	\frac{3-2\sigma}{2\sigma-1}
	\right)
	\right]^{\frac{2\sigma-1}{2}},
	\]
    here $B(\cdot,\cdot)$ denote the beta function.
\end{cor}

In particular, when \(\sigma=1\), the above inequality reduces to the classical Huber isoperimetric inequality on simply connected domains.

\begin{cor}\label{cor:Huber_domain-1}
	Let $\Omega\subset\mathbb C$ be a bounded simply connected $C^1$ domain. For $F\in C^2(\Omega)\cap C(\overline\Omega)$, suppose
	\[
	\gamma:=\int_\Omega(-\Delta F(x))^+\ud x<2\pi.
	\]
	Then
	\[
	\left(
	\frac{1}{\pi}
	\int_\Omega e^{2F(x)}\ud x
	\right)^{\frac{1}{2}}
	\le
	\frac{c_1}{2\pi}\int_{\partial\Omega}e^F\ud s,
	\]
	where
	\begin{align*}
	c_1=&
	\left[
	B\left(
	1-\frac{\gamma}{2\pi},1
	\right)
	\right]^{\frac{1}{2}}
    =\left[
	\frac{\Gamma(1)\Gamma(1-\frac{\gamma}{2\pi})}{\Gamma(2-\frac{\gamma}{2\pi})}
	\right]^{\frac{1}{2}}=\left(\frac{1}{1-\frac{\gamma}{2\pi}}\right)^{1/2}.
    \end{align*}
\end{cor}

The paper is organized as follows. In Section \ref{Sec 2}, we use a limiting approach to derive the weighted Carleman-type inequality. In Section \ref{Sec 3}, we establish the conformal invariance of the weighted Carleman inequality and derive its Euclidean version. In Section \ref{Sec 4}, we apply the moving sphere method to classify the extremal functions of the weighted Carleman inequality and, in the case \(n=2\), prove the equivalence between our weighted Carleman inequality and a sharp norm inequality in the Bergman space. In Section \ref{Sec 5}, we discuss conformally invariant boundary operators and derive the weighted Huber-type isoperimetric inequality. Finally, in Appendix \ref{Appendix}, we prove the regularity of extremal functions for the weighted Carleman inequality.

		\section{The limiting approach}\label{Sec 2}
			We first recall some basic notation concerning the  hypergeometric function; for a complete treatment, see \cite{Hypergeometry,Gradshteyn&Ryzhik}.
		Given  real numbers $a,b,c$,  define
		\begin{align}\label{Hyper}
			{}_2F_1\left(a,b;c; z \right)=\sum_{r=0}^{+\infty}\frac{(a)_{r}(b)_{r}}{(c)_{r}}\frac{z^r}{r!}\qquad\mathrm{for}\qquad |z|<1,
		\end{align}
		where $ c \neq 0, -1, -2, \dots, $ and $(a)_k $ denotes the rising Pochhammer symbol:	
		$$
		(a)_{0}=1,\quad (a)_{k}=a(a+1)\cdots(a+k-1), \quad k\geq1.
		$$	Clearly, we have ${}_2F_1(a,b;c;z)={}_2F_1(b,a;c;z)$.   If $\mathrm{Re}~ c>\mathrm{Re}~ b>0$, it has the integral representation (cf. \cite[p.59]{Hypergeometry} ):
		\begin{align}\label{Pre 1.8}
			{}_2F_1(a,b;c;z)=\frac{\Gamma(c)}{\Gamma(b)\Gamma(c-b)}\int_{0}^{1}t^{b-1}(1-t)^{c-b-1}(1-tz)^{-a}\ud t.
		\end{align}
		Below we list two classical identities  that will be used in the subsequent analysis:
		\begin{itemize}	
			\item Transformation (cf. \cite[p.~1018: 9.131-1]{Gradshteyn&Ryzhik}):
			\begin{equation}\label{Pre 1.4}
				\begin{split}
					{}_2F_1(a,b;c;z)=(1-z)^{c-a-b} {}_2F_1(c-a,c-b;c;z).
				\end{split}
			\end{equation}
			
			\item Transformation (cf. \cite[p.1017: 9.122-1]{Gradshteyn&Ryzhik}): if $\mathrm{Re}~c>\mathrm{Re}(a+b)$, then
			\begin{equation}\label{Pre 1.5}
				\begin{split}
					{}_2F_1(a,b;c;1)=\frac{\Gamma(c)\Gamma(c-a-b)}{\Gamma(c-a)\Gamma
						(c-b)}.
				\end{split}
			\end{equation}
		\end{itemize}
	
Recall the definition of $Q_{\mathbf{a},\mathbf{b}}$ in \eqref{Def Q}. Following \cite{Gong&Yang&Zhang,Yang&Zhang}, we view it as a nonlocal operator and recall the following formula.

\begin{lem}[\protect{\cite[Theorem 4.1]{Gong&Yang&Zhang}; see also \cite[(2.21)]{Yang&Zhang}}]\label{Lem 2.1}
Assume that $\mathbf{a},\mathbf{b}$ satisfy the constraint \eqref{Index Condi}, and that $Q_{\mathbf{a},\mathbf{b}}$ is defined by \eqref{Def Q}. Then
\begin{align}\label{Sec Def Q}
	Q_{\mathbf{a},\mathbf{b}}(1)(\xi)
	=
	\frac{\pi^{\frac{n}{2}}2^{1-\mathbf{b}}}{\Gamma(\frac{n}{2})|\S^{n-1}|}
	(1-|\xi|^2)^{\mathbf{a}+\mathbf{b}-1}
	{}_2F_1\left(\frac{n+\mathbf{a}}{2}-1,\frac{\mathbf{a}}{2}; \frac{n}{2}; |\xi|^2\right).
\end{align}
\end{lem}

\begin{lem}\label{Lem 2.2}
Let \(h\in L^1(\B^n)\). Suppose that
$
f_t=1+tg_t,|g_t|\leq M,
$
and that \(g_t\to g\) pointwise as \(t\to0^+\). Then, for any \(k>0\),
\begin{align*}
    hf_t^{\frac{k}{t}}\ud \xi\to he^{kg}\ud \xi \qquad\mathrm{in~ measure~ sense}.
\end{align*}
\end{lem}

		To obtain a sharp exponential-type inequality, an effective approach is to pass to the limit in a sharp polynomial-type inequality. This idea goes back to Beckner \cite{Beckner}; see also \cite{Chen Shibing,Dolbeault&Esteban&Figalli&Frank&Loss}. Here we adopt a similar strategy. A crucial point is to choose the parameters appropriately so that $p,q\to +\infty$. Therefore, we can directly  prove the proof of the inequality part in Theorem \ref{Thm 1}.

		\begin{thm}\label{Sec 2 thm}
			For any $f\in L^{1}(\S^{n-1})$ with $n\geq 2$ and $\sigma\in\left(\frac{1}{2},\frac{2n-1}{2(n-1)}\right)$, we have 
			\begin{align}\label{Sec 2 Ineq}
				\left(\int_{\B^n}e^{m_{\sigma}P(f)(\xi)}\ud \nu\right)^{\frac{1}{m_{\sigma}}}\leq \int_{\S^{n-1}}e^{f}\ud \mu,
			\end{align}
			where $m_{\sigma}=\frac{n}{(2\sigma-1)(n-1)}$ and 
			\begin{align*}
				\ud \nu=\frac{(1-|\xi|^2)^{\frac{2n(1-\sigma)}{2\sigma-1}}e^{\frac{2n}{2\sigma-1}I(|\xi|^2)}\ud \xi}{\int_{\B^n}(1-|\xi|^2)^{\frac{2n(1-\sigma)}{2\sigma-1}}e^{\frac{2n}{2\sigma-1}I(|\xi|^2)}\ud \xi },\qquad  I(x)= \sum_{k=1}^{+\infty}\frac{\left(1-\frac{n}{2}\right)_k}{\left(\frac{n}{2}\right)_k}\frac{x^{k}}{2k}.
			\end{align*}
		\end{thm}
		\begin{rem}\label{Sec 2 rem}
			By the definition of the Gamma function and rising Pochhammer symbol, for $n$ is odd, we may rewrite $I(1)$ as
\begin{align*}
	I(1)
	=
	\sum_{k=1}^{+\infty}
	\frac{\Gamma(\frac{n}{2})}{\Gamma(k+\frac{n}{2})}
	\frac{\Gamma(k+1-\frac{n}{2})}{\Gamma(1-\frac{n}{2})}
	\frac{1}{2k}.
\end{align*}
Using Stirling's formula, we have
\begin{align}\label{Gamma asy}
	\frac{\Gamma(k+\alpha)}{\Gamma(k+\beta)}
	=
	O\bigl(k^{\alpha-\beta}\bigr)
	\qquad \text{as }\qquad k\to+\infty.
\end{align}
It follows that $I(x)$ is a bounded function in $[0,1]$. For $n$ is even, the summation in $I$ is finite, therefore it must be bounded. 
		\end{rem}
		
		\begin{pf}
			
			For $t\to 0^{+}$, set $\mathbf{a}=2-n+t$ and $\mathbf{b}=n-1-\sigma t$ in Lemma \ref{Lem 2.1}, 
			then 
			\begin{align*}
				p_t=\frac{2(n-1)}{t}, \qquad q_t=\frac{2n}{(2\sigma-1)t}.
			\end{align*}
			To preserve the constraint \eqref{Index Condi} in Theorem B 
 \ref{thm B}, we need to restrict $\sigma$ in the range
			\begin{align*}
				\frac{1}{2}<\sigma<\frac{2n-1}{2(n-1)}.
			\end{align*} 
			We also consider the corresponding perturbation operator
			\begin{align*}
				Q_t(u)(\xi)=\int_{\S^{n-1}}H_t(\xi,\eta)\,u(\eta)\,
				\ud \mu(\eta),
			\end{align*}
			where
			\begin{align*}
				H_t(\xi,\eta)=\left(\frac{1-|\xi|^2}{2}\right)^{n-1-\sigma t}|\xi-\eta|^{-2(n-1)+t}.
			\end{align*}
			We introduce the normalized measure
			\begin{align*}
				\ud \nu_t=\frac{Q^{q_t}_{t}(1)(\xi)}{\int_{\B^n}Q^{q_t}_{t}(1)}\ud \xi
				,
			\end{align*}
			then we can rewrite Theorem B \ref{thm B} or \cite[Theorem 1.1]{Gluck} as follows
			\begin{align}\label{Sec 2 thm equ-d}
				\|Q_{t}(u)/Q_t(1)\|_{L^{q_t}(\B^{n},\ud \nu_t)}\leq \|u\|_{L^{p_t}(\S^{n-1},\ud \mu) }.
			\end{align}
				Firstly, we need to claim 
			\begin{align*}
				\ud \nu_t\to 	\ud \nu=\frac{e^{\frac{2n}{2\sigma-1}\phi}}{\int_{\B^n} e^{\frac{2n}{2\sigma-1}\phi}}\ud \xi,
			\end{align*}
            	where
			\begin{align}\label{Sec 2 phi def}
				\phi(\xi)=(1-\sigma)\log(1-|\xi|^2)+ \sum_{k=1}^{+\infty}\frac{\left(1-\frac{n}{2}\right)_k}{\left(\frac{n}{2}\right)_k}\frac{|\xi|^{2k}}{2k}.
			\end{align}
			From the  formula \eqref{Sec Def Q} in Lemma \ref{Lem 2.1} , it follows that 
			 \begin{align*}
				Q_t(1)(\xi)
				=&\frac{\pi^{\frac{n}{2}}2^{2-n+\sigma t}}{\Gamma\!\left(\frac{n}{2}\right)|\S^{n-1}|}
				(1-|\xi|^2)^{(1-\sigma)t}
				{}_2F_1\!\left(\frac{t}{2},\frac{2-n+t}{2}; \frac{n}{2}; |\xi|^2\right)
			\end{align*}
			Noticing the inequality
			\begin{align*}
				\frac{n}{2}>\frac{t}{2}+\frac{2-n+t}{2}
			\end{align*}
		using \eqref{Pre 1.5}, 	hence the 
			hypergeometric function is convergent absolutely. 
Let \(z=|\xi|^2\in[0,1]\). For \(t\neq0\), by the series expansion of the hypergeometric function, we have
\[
\begin{aligned}
\frac{{}_2F_1\!\left(\frac t2,\frac{2-n+t}{2};\frac n2;z\right)-1}{t}
 =
\frac12\sum_{k=1}^{\infty}
\frac{\left(1+\frac t2\right)_{k-1}
\left(1-\frac n2+\frac t2\right)_k}
{\left(\frac n2\right)_k}
\frac{z^k}{k!}.
\end{aligned}
\]
Indeed, this follows from
\[
\left(\frac t2\right)_k
=
\frac t2\left(1+\frac t2\right)_{k-1},
\qquad k\ge1.
\]
For each fixed \(k\ge1\), the \(k\)-th term converges, as \(t\to0\), to
\[
\frac12
\frac{(1)_{k-1}\left(1-\frac n2\right)_k}
{\left(\frac n2\right)_k}
\frac{z^k}{k!}
=
\frac12
\frac{\left(1-\frac n2\right)_k}
{\left(\frac n2\right)_k}
\frac{z^k}{k}.
\]
It remains to justify the passage of the limit through the series. Choose
\(\varepsilon>0\) sufficiently small so that \(\varepsilon<n-1\). By the standard estimate for ratios of Gamma functions, see \eqref{Gamma asy},  uniformly for \(|t|\le \varepsilon\),
\[
\left|
\frac{\left(1+\frac t2\right)_{k-1}
\left(1-\frac n2+\frac t2\right)_k}
{\left(\frac n2\right)_k k!}
\right|
\le C k^{-n+\varepsilon},
\qquad k\ge1.
\]
Here the constant \(C\) is independent of \(t,k\) and \(z\in[0,1]\). Since
\[
\sum_{k=1}^{\infty} k^{-n+\varepsilon}<\infty,
\]
the above series is dominated by a summable sequence. Hence, by the dominated convergence theorem for series,
\begin{align}\label{Hyper equ-a}
\left.\frac{\ud}{\ud t}\right|_{t=0}
{}_2F_1\!\left(\frac t2,\frac{2-n+t}{2};\frac n2;z\right)
=
\frac12\sum_{k=1}^{\infty}
\frac{\left(1-\frac n2\right)_k}
{\left(\frac n2\right)_k}
\frac{z^k}{k}.
\end{align}
In particular, the derivative exists for every \(z=|\xi|^2\in[0,1]\), that is, for every \(|\xi|\le1\).
 It means that
			\begin{align*}
				{}_2F_1\!\left(\frac{t}{2},\frac{2-n+t}{2}; \frac{n}{2}; |\xi|^2\right)=1+tg_t,\quad \mathrm{where}\quad g_t\to \sum_{k=1}^{+\infty}\frac{\left(1-\frac{n}{2}\right)_k}{\left(\frac{n}{2}\right)_k}\frac{|\xi|^{2k}}{2k},\quad |g_t|\leq M.
			\end{align*}
           Then, using Lemma \ref{Lem 2.2}, we can see
			\begin{align*}
				\ud \nu_t=\frac{Q^{q_t}_{t}(1)(\xi)}{\int_{\B^n}Q^{q_t}_{t}(1)}\ud \xi=&\frac{(1-|\xi|^2)^{\frac{2n(1-\sigma)}{2\sigma-1}}\left(1+tg_t \right)^{q_t}}{\int_{\B^n}(1-|\xi|^2)^{\frac{2n(1-\sigma)}{2\sigma-1}}\left(1+tg_t \right)^{q_t}\ud \xi}\ud \xi
                \to \frac{e^{\frac{2n}{2\sigma-1}\phi}}{\int_{\B^n} e^{\frac{2n}{2\sigma-1}\phi}}\ud \xi=\ud \nu,
			\end{align*}
			here we have used $\frac{2n(1-\sigma)}{2\sigma-1}>-1$ and dominated convergence theorem.

			\textbf{Case 1}: $f\in L^{\infty}(\S^{n-1})$.  Letting $u=e^{\frac{t}{2(n-1)}f}\in L^{p_t}(\S^{n-1})$, then we can compute the asymptotic behaviour as follows:
			\begin{align*}
				\|u\|_{L^{p_t}(\S^{n-1},\ud \mu) }=\left(\int_{\S^{n-1}}e^{f}\ud \mu\right)^{\frac{t}{2(n-1)}}=1+\frac{t}{2(n-1)}\log \left(\int_{\S^{n-1}}e^{f}\ud \mu\right)+o(t),
			\end{align*}
			and 
			\begin{align*}
				\frac{Q_t(u) }{Q_t(1)}=&\frac{Q_t(1+\frac{tf}{2(n-1)}+o(t))}{Q_t(1)}=1+\frac{t}{2(n-1)}\frac{Q_t(f)}{Q_t(1)}+o(t)
				=:1+t\tilde{g}_t+o(t),
			\end{align*}
			where $|\tilde{g}_t|\leq \frac{1}{2(n-1)}\|f\|_{L^{\infty}(\S^{n-1})}$ and $\tilde{g}_t\to \tilde{g}_0:= \frac{1}{2(n-1)}\frac{Q_0(f)}{Q_0(1)}$, again using Lemma \ref{Lem 2.2}, we can get 
            \begin{align*}
                (1+t\tilde{g}_t+o(t))^{q_t}\ud \nu_t\to e^{\frac{2n\tilde{g}_0}{2\sigma-1}}\ud \nu\end{align*}
           Therefore, we can get 
			\begin{align*}
				\lim_{t\to 0}\frac{\|\frac{Q_t(u) }{Q_t(1)}\|_{L^{q_t}(\B^{n},\ud \nu_t)}-1}{t}=&\lim_{t\to 0} \frac{\exp\left(\frac{(2\sigma-1)t}{2n}\log\int_{\B^n}(1+t\tilde{g}_t+o(t))^{q_t}\ud \nu_t\right)-1}{t}\\
                =&\frac{(2\sigma-1)}{2n}\log\int_{\B^n}e^{\frac{2n\tilde{g}_0}{2\sigma-1}}\ud \nu.
			\end{align*}
            For $t\geq 0$, we define the function
            \begin{align*}
                l(t)=\|Q_{t}(e^{\frac{t}{2(n-1)}f})/Q_t(1)\|_{L^{q_t}(\B^{n},\ud \nu_t)}-\|e^{\frac{t}{2(n-1)}f}\|_{L^{p_t}(\S^{n-1},\ud \mu) },        \end{align*}
		from \eqref{Sec 2 thm equ-d}, we obtain $l(t)\leq 0$ and $l(0)=0$. Therefore, it is not hard to see $l_{+}'(0)\leq 0$, then	it follows that 
			\begin{align*}
				\frac{(2\sigma-1)(n-1)}{n}\log \int_{\B^n}e^{\frac{nQ_0(f)}{(2\sigma-1)(n-1)Q_0(1)}}\ud \nu\leq \log \left(\int_{\S^{n-1}}e^{f}\ud \mu\right).
			\end{align*}
			Set $m_{\sigma}=\frac{n}{(2\sigma-1)(n-1)}$, then 
			\begin{align*}
				\left(\int_{\B^n}e^{\frac{m_{\sigma}Q_0(f)}{Q_0(1)}}\ud \nu\right)^{\frac{1}{m_{\sigma}}}\leq \int_{\S^{n-1}}e^{f}\ud \mu,
			\end{align*}
		and the desired inequality follows by 
		\begin{align*}
			\frac{Q_0(f)}{Q_0(1)}=\int_{\S^{n-1}}\left(\frac{1-|\xi|^2}{|\xi-\eta|^2}\right)^{n-1}f(\eta)\ud \mu(\eta).
		\end{align*}
				\textbf{Case 2}: $f\in L^{1}(\S^{n-1})$. Without loss of generality, we can also assume $e^{f}\in L^1(\S^{n-1})$. Then, we define the approximate sequence
				\begin{align*}
					f_k=\max\{-k,\min\{f,k\}\}.
				\end{align*}
				Therefore, it follows that $\|f_k-f\|_{L^1(\S^{n-1})}\to 0$ as $k\to+\infty$, moreover we also estimate
				\begin{align*}
					\int_{\S^{n-1}}e^{f_k}\ud \mu\leq \int_{\S^{n-1}}e^{f}\ud\mu+\mu(\{f<-k\})e^{-k}.
				\end{align*}
				Moreover, we also have $P(f_k)\to P(f)$ pointwise, applying Fatou's Lemma, 
				\begin{align*}
				\left(\int_{\B^n}e^{m_{\sigma}P(f)(\xi)}\ud \nu\right)^{\frac{1}{m_{\sigma}}}	\leq& \liminf_{k\to+\infty}\left(\int_{\B^n}e^{m_{\sigma}P(f_k)(\xi)}\ud \nu\right)^{\frac{1}{m_{\sigma}}}\leq \liminf_{k\to+\infty}\int_{\S^{n-1}}e^{f_k}\ud \mu \\
				\leq &\int_{\S^{n-1}}e^{f}\ud \mu +\liminf_{k\to+\infty}e^{-k}=\int_{\S^{n-1}}e^{f}\ud \mu.
				\end{align*}
				This completes the whole proof.
		\end{pf}

We emphasize that, although the inequality can be obtained through a limiting argument, its conformal invariance does not follow automatically from this approach. Owing to the complexity of the measure $\ud \nu$, the conformal invariance is itself a nontrivial issue. The remaining tasks are therefore to verify the conformal invariance of this exponential inequality and to classify all extremal functions in $L^{\infty}(\S^{n-1})$.

		\section{Conformally invariant properties}\label{Sec 3}
		\subsection{Conformal invariance in \texorpdfstring{$\B^n$}{Bn}}		
		In this subsection, we establish the conformal invariance of inequality \eqref{Sec 2 Ineq}. To achieve this, we begin by recalling the nontrivial conformal automorphisms of $\B^n$; for further background, see \cite{Schoen&Yau}. Fix $a\in \B^n$, and let $\Psi_{a}:\B^n\to \B^n$ be the conformal transformation defined by
\begin{align*}
	\Psi_{a}(\xi)
	=
	\frac{(1-|a|^2)(\xi-a)-|\xi-a|^2a}{1-2\xi\cdot a+|\xi|^2|a|^2},
\end{align*}
then $\Psi_a^{-1}=\Psi_{-a}$. This map satisfies the basic identities $\Psi_{a}(0)=-a$ and $\Psi_{a}(a)=0$. We write $\Psi_a|_{\S^{n-1}}$ for the boundary value of $\Psi_a$ on $\S^{n-1}$. The corresponding Jacobian determinants are given by
\begin{align*}
	\det \ud \Psi_{a}(\xi)
	=
	\left(\frac{1-|a|^2}{1-2a\cdot\xi+|a|^2|\xi|^2 }\right)^n,
	\qquad
	\det \ud \Psi_a|_{\S^{n-1}}(\eta)
	=
	\left(\frac{1-|a|^2}{1-2a\cdot\eta+|a|^2 }\right)^{n-1}.
\end{align*}
For convenience, given $\Psi_a$ and $\xi,\xi'\in \B^{n}$, we set
$
	[a,\xi]=1-2a\cdot\xi+|a|^2|\xi|^2.
$
Then the following standard identities hold:
\begin{align}\label{conformal equ-a}
	1-|\Psi_{a}(\xi)|^2
	=
	\frac{(1-|a|^2)(1-|\xi|^2)}{[a,\xi]},
\end{align}
and
\begin{align}\label{conformal equ-b}
	|\Psi_{a}(\xi)-\Psi_{a}(\xi')|^2
	=
	\frac{(1-|a|^2)^2|\xi-\xi'|^2}{[a,\xi][a,\xi']}.
\end{align}

In what follows, we only need to handle the case \(n\geq 3\), since the case \(n=2\) is covered by Remark \ref{Sec 3.1 rem}. We also recall the classical Funk--Hecke formula. Let
$
	K\in L^{1}\bigl((-1,1),(1-t^{2})^{(n-3)/2}\ud t\bigr).
$
Then the Funk--Hecke formula states (see \cite[Chapter 22]{Abramowitz&Stegun}) that
\begin{align}\label{Funk-Hecke}
	\int_{\S^{n-1}} K(\xi\cdot\eta) Y_l(\eta)\,\ud \mu(\eta)
	=
	\lambda_l|\S^{n-1}|^{-1}Y_l(\xi),
	\qquad
	Y_l \in \mathscr{H}_l,
\end{align}
where $\mathscr{H}_l$ denotes the space of spherical harmonics of degree $l$, and
\begin{align*}
	\lambda_l
	=
	(4\pi)^{\frac{n-2}{2}}
	\frac{\Gamma\left(\frac{n-2}{2}\right)l!}{\Gamma\left(l+n-2\right)}
	\int_{-1}^{1}
	K(t)C^{\frac{n-2}{2}}_l(t)(1-t^2)^{\frac{n-3}{2}}\,\ud t,
\end{align*}
with $C_l^{\frac{n-2}{2}}$ denoting the $l$-th Gegenbauer polynomial of order $\frac{n-2}{2}$. We begin with the following basic integral lemma.
		
	\begin{lem}\label{Sec 3 Lem1}
For any $a\in \B^n$, one has
\begin{align*}
	\int_{\S^{n-1}}\log(1-2a\cdot\eta+|a|^2)\,\ud \mu(\eta)
	=
	-2I(|a|^2).
\end{align*}
\end{lem}

\begin{pf}
	The proof is based on a standard differentiation argument applied to a family of spherical averages. For $\lambda$ in a neighborhood of $0$, define
	\[
		L_{\lambda}(a)
		:=
		\int_{\S^{n-1}}(1-2a\cdot\eta+|a|^2)^{-\lambda}\,\ud \mu(\eta),
		\qquad |a|<1.
	\]
	Since the integrand depends smoothly on $\lambda$, differentiating at $\lambda=0$ gives
	\[
		\left.\frac{\ud}{\ud \lambda}\right|_{\lambda=0}L_{\lambda}(a)
		=
		-\int_{\S^{n-1}}\log(1-2a\cdot\eta+|a|^2)\,\ud \mu(\eta).
	\]
	Therefore, it remains to compute $L_{\lambda}(a)$ explicitly. Let $r=|a|$. By the standard formula for zonal integrals on the sphere, equivalently by the Funk--Hecke formula \eqref{Funk-Hecke} with $l=0$, we obtain
	\[
		L_{\lambda}(a)
		=
		\frac{\Gamma(\frac n2)}{\sqrt{\pi}\Gamma(\frac{n-1}{2})}
		\int_{-1}^1
		(1-2rt+r^2)^{-\lambda}(1-t^2)^{\frac{n-3}{2}}\,\ud t.
	\]
	To transform this integral into a hypergeometric form, we make the change of variables $t=1-2s$. Then
	\begin{align*}
		L_{\lambda}(a)
		&=
		\frac{\Gamma(\frac n2)}{\sqrt{\pi}\Gamma(\frac{n-1}{2})}
		2^{n-2}(1-r)^{-2\lambda}
		\int_0^1
		\left(1+\frac{4r}{(1-r)^2}s\right)^{-\lambda}
		s^{\frac{n-3}{2}}(1-s)^{\frac{n-3}{2}}\,\ud s.
	\end{align*}
	Now Euler's integral representation of the Gauss hypergeometric function (see \eqref{Pre 1.8}) yields
	\begin{align*}
		\int_0^1
		\left(1+\frac{4r}{(1-r)^2}s\right)^{-\lambda}
		s^{\frac{n-3}{2}}(1-s)^{\frac{n-3}{2}}\,\ud s
		=
		B\!\left(\frac{n-1}{2},\frac{n-1}{2}\right)
		{}_2F_1\!\left(\lambda,\frac{n-1}{2};n-1;-\frac{4r}{(1-r)^2}\right).
	\end{align*}
	Combining this with the prefactor and using the duplication formula for the Gamma function,
    \[\Gamma(2z)=\frac{2^{2z-1}}{\sqrt{\pi}}\Gamma(z)\Gamma(z+\frac{1}{2}),\]
    we simplify the constant term and arrive at
	\[
		L_{\lambda}(a)
		=
		(1-r)^{-2\lambda}
		{}_2F_1\!\left(\lambda,\frac{n-1}{2};n-1;-\frac{4r}{(1-r)^2}\right).
	\]
The next step is to apply the quadratic transformation for ${}_2F_1$ (cf.~\cite[p.~1018, Eq.~9.134--3]{Gradshteyn&Ryzhik}),
	\[
		(1-r)^{-2\alpha}\,
		{}_2F_1\!\left(\alpha,\beta;2\beta;-\frac{4r}{(1-r)^2}\right)
		=
		{}_2F_1\!\left(\alpha,\alpha-\beta+\frac12;\beta+\frac12;r^2\right).
	\]
	Taking $\alpha=\lambda$ and $\beta=\frac{n-1}{2}$, we obtain the more convenient expression
	\[
		L_{\lambda}(a)
		=
		{}_2F_1\!\left(\lambda,\lambda-\frac n2+1;\frac n2;|a|^2\right).
	\]
	Finally, differentiating with respect to $\lambda$ at $\lambda=0$, using \eqref{Hyper equ-a} and recalling the definition of $I$, we conclude that
	\[
		-\int_{\S^{n-1}}\log(1-2a\cdot\eta+|a|^2)\,\ud \mu(\eta)
		=
		2I(|a|^2).
	\]
	This is exactly the desired identity.
\end{pf}

	Using the above integral lemma, we can derive an explicit formula for the action of the Poisson kernel on the standard bubble.
		
		\begin{lem}\label{Conformal lem 2}
			For any $a\in \B^n$, there holds
			\begin{align*}
				\int_{\S^{n-1}}&\left(\frac{1-|\xi|^2}{|\xi-\eta|^2}\right)^{n-1}\log \left(\frac{1-|a|^2}{1-2a\cdot\eta+|a|^2}\right)
				\ud \mu(\eta)\\
				=&\log \left(\frac{1-|a|^2}{1-2a\cdot\xi+|a|^2|\xi|^2 }\right)+2I \left(|\Psi_{a}(\xi)|^2\right)-2I(|\xi|^2).
			\end{align*}
		\end{lem}
		\begin{pf}
			
			We argue by making a suitable conformal change of variables on the sphere. Consider $\zeta=\Psi_{\xi}\big|_{\S^{n-1}}(\eta)$. Then 
			\begin{align*}
				\left(\frac{1-|\xi|^2}{|\xi-\eta|^2}\right)^{n-1}\ud \mu(\eta)=\Psi_{\xi}\big|_{\S^{n-1}}^{*}( \ud \mu)= \ud \mu(\zeta),
			\end{align*}
			so that
			\begin{align*}
				L:=&\int_{\S^{n-1}}\left(\frac{1-|\xi|^2}{|\xi-\eta|^2}\right)^{n-1}\log (1-2a\cdot\eta+|a|^2)
				\ud \mu(\eta)\\
				=&\int_{\S^{n-1}}\log |a-\eta(\zeta)|^2
				\ud \mu(\zeta).
			\end{align*}
			Set $\tilde{a}=\Psi_{\xi}(a)$. Using \eqref{conformal equ-b} and the identity $-\xi=\Psi_{\xi}(0)$, we obtain 
			\begin{align*}
				|\tilde{a}-\zeta|^2=\frac{(1-|\xi|^2)^2|a-\eta|^2}{[a,\xi]|\xi-\eta|^2} \qquad\mathrm{and}\qquad |\zeta+\xi|^2=\frac{(1-|\xi|^2)^2}{|\xi-\eta|^2},
			\end{align*}
			and hence
			\begin{align*}
				|a-\eta|^2=\frac{|\tilde{a}-\zeta|^2[a,\xi]}{|\zeta+\xi|^2}.
			\end{align*}
			Substituting this identity into the expression for $L$ and using Lemma \ref{Sec 3 Lem1}, we find
			\begin{align*}
				L=&\log [a,\xi]+\int_{\S^{n-1}}\log|\tilde{a}-\zeta|^2-\log |\zeta+\xi|^2\ud \mu(\zeta)\\
				=&\log [a,\xi]+2I(|\xi|^2)-2I(|\tilde{a}|^2).
			\end{align*}
			The final conclusion follows by $|\tilde{a}|=|\Psi_{\xi}(a)|=|\Psi_{a}(\xi)|$.
		\end{pf}
\begin{rem}\label{Sec 3.1 rem}
When $n=2$, we have $I=0$, and the above formula is trivial. Indeed, the hyperbolic harmonic extension reduces to the usual harmonic extension, while the right-hand side is harmonic. Hence the desired formula follows.
\end{rem}

For $f\in L^{\infty}(\S^{n-1})$, in order to preserve the conformal invariance of the integral $\int_{\S^{n-1}}e^{f}\ud \mu$, we introduce the conformal action of $\Psi_a$ on $f$ by
		\begin{align*}
			f_{\Psi_a|_{\S^{n-1}}}=f\circ \Psi_a|_{\S^{n-1}}+\log\det\ud \Psi_a|_{\S^{n-1}}.
		\end{align*}
When \(n=2\), we have \(I=0\). Using \eqref{conformal equ-a}, Remark \ref{Sec 3.1 rem}, and the formula for \(\det \ud \Psi_a\), we obtain \eqref{Sec 3.1 conformal}.
 Now, using Lemma \ref{Conformal lem 2}, we obtain the following conformal invariance property.
		\begin{thm}
			For any conformal transformation $\Psi_a$ on $\S^{n-1}$, one has
			\begin{align}\label{Sec 3.1 conformal}
				\int_{\B^n}e^{m_{\sigma}P(f_{\Psi_a|_{\S^{n-1}}})(\xi)}e^{\frac{2n}{2\sigma-1}\phi(\xi)}\ud \xi=\int_{\B^n}e^{m_{\sigma}P(f)(\xi)}e^{\frac{2n}{2\sigma-1}\phi(\xi)}\ud \xi,
			\end{align}
            where $\phi$ is given by \eqref{Sec 2 phi def}.
		\end{thm}
		\begin{pf}
			Note that
			\begin{align*}
				P(f)(\Psi_a(\xi))=&\int_{\S^{n-1}}\left(\frac{1-|\Psi_a(\xi)|^2}{|\Psi_a(\xi)-\eta|^2}\right)^{n-1}f(\eta)\, 
				\ud \mu(\eta)\\
				=&\int_{\S^{n-1}}\left(\frac{1-|\Psi_a(\xi)|^2}{|\Psi_a(\xi)-\Psi_a|_{\S^{n-1}}(\eta)|^2}\right)^{n-1}f\circ\Psi_a|_{\S^{n-1}}(\eta)\det\ud \Psi_a|_{\S^{n-1}}\ud \mu(\eta).
			\end{align*}
			Applying \eqref{conformal equ-a}, \eqref{conformal equ-b}, and using
			$ \det\ud \Psi_a|_{\S^{n-1}}=(1-|a|^2)^{n-1}[a,\eta]^{1-n}$, we see that
			\begin{align*}
				\left(\frac{1-|\Psi_a(\xi)|^2}{|\Psi_a(\xi)-\Psi_a|_{\S^{n-1}}(\eta)|^2}\right)^{n-1}\det\ud \Psi_a|_{\S^{n-1}}=\left(\frac{1-|\xi|^2}{|\xi-\eta|^2}\right)^{n-1}.
			\end{align*}
			Thus, we conclude that
			\begin{align*}
				P(f\circ\Psi_a|_{\S^{n-1}})(\xi)=P(f)(\Psi_a(\xi)).
			\end{align*}
			Hence, the left-hand side becomes
			\begin{align*}
				&\int_{\B^n}e^{m_{\sigma}P(f_{\Psi_a|_{\S^{n-1}}})(\xi)}e^{\frac{2n}{2\sigma-1}\phi(\xi)}\ud \xi\\
				=&\int_{\B^n} e^{m_{\sigma}P(f)\circ \Psi_a+ \frac{n}{2\sigma-1}P\left(\log\frac{1-|a|^2}{1-2a\cdot\eta+|a|^2}\right)(\xi)}e^{\frac{2n}{2\sigma-1}\phi(\xi)}\ud \xi.
			\end{align*}
			On the other hand, the right-hand side can be written as
			\begin{align*}
				\int_{\B^n}e^{m_{\sigma}P(f)(\xi)}e^{\frac{2n}{2\sigma-1}\phi(\xi)}\ud \xi=\int_{\B^n}e^{m_{\sigma}P(f)\circ\Psi_a}e^{\frac{2n}{2\sigma-1}\phi\circ\Psi_a}\det \ud\Psi_a\ud \xi.
			\end{align*}
			Therefore, it remains to verify that
			\begin{align}\label{Conformal thm equ-a}
				\frac{n}{2\sigma-1}P\left(\log\frac{1-|a|^2}{1-2a\cdot\eta+|a|^2}\right)
				=\frac{2n}{2\sigma-1}(\phi\circ\Psi_a-\phi)+\log \det \ud\Psi_a.
			\end{align}
			By definition, and using \eqref{conformal equ-a}, the right-hand side is equal to
			\begin{align*}
				&\frac{2n}{2\sigma-1}\left(\phi\circ\Psi_a-\phi\right)+n\log\frac{1-|a|^2}{1-2a\cdot\xi+|a|^2|\xi|^2 }\\
				=&\frac{2n}{2\sigma-1}\left((1-\sigma)\log\frac{1-|\Psi_a(\xi)|^2}{1-|\xi|^2}+I(|\Psi_a(\xi)|^2)-I(|\xi|^2)\right)+n\log\frac{1-|a|^2}{1-2a\cdot\xi+|a|^2|\xi|^2 }\\
				=&\frac{n}{2\sigma-1}\log\frac{1-|a|^2}{1-2a\cdot\xi+|a|^2|\xi|^2 }+\frac{2n}{2\sigma-1}(I(|\Psi_a(\xi)|^2)-I(|\xi|^2) )\\
                =&\frac{n}{2\sigma-1}\left[\log\frac{1-|a|^2}{1-2a\cdot\xi+|a|^2|\xi|^2 }+2(I(|\Psi_a(\xi)|^2)-I(|\xi|^2) )\right].
			\end{align*}
			By Lemma \ref{Conformal lem 2}, this is exactly the left-hand side of \eqref{Conformal thm equ-a}. The proof is complete.
		\end{pf}

		\subsection{The weighted  Carleman inequality in \texorpdfstring{$\R^n_{+}$}{Rn+}}
		
		To classify all extremal functions, we need to transform the inequality to the upper half-space $\R^{n}_{+}$ and then apply the method of moving spheres to the Euler--Lagrange equation in $\R^{n}_{+}$. To this end, we first introduce the conformal map from the upper half-space to the unit ball. Let $\mathcal{S}: (\R^{n}_{+}, |\ud x|^2) \to (\B^{n},|\ud \xi|^2) $ 
		\begin{equation}\label{conf-map:ball_half-space}
			\xi=\mathcal{S}(x)=-\mathbf{e}_{n}+\frac{2(x+\mathbf{e}_{n})}{|x+\mathbf{e}_{n}|^2}\qquad\mathrm{where}\qquad x=(x',x_n),
		\end{equation}
		be a conformal map with the property that
		$$\mathcal{S}^\ast(|\ud \xi|^2)=\left(\frac{2}{|x+\mathbf{e}_{n}|^2}\right)^2 |\ud x|^2.$$
		Moreover, given two distinct points $x\neq y \in \R^n_{+}$, set $\xi=\mathcal{S}(x)$ and $\eta=\mathcal{S}(y)$. Then we have the following two useful identities:
		\begin{align}\label{Sec 4 equ-a}
			|\xi-\eta|^2= \frac{4|x-y|^2}{|x+\mathbf{e}_{n}|^2|y+\mathbf{e}_{n}|^2 }
		\end{align}
		and 
		\begin{align}\label{Sec 4 equ-b}
			1-|\xi|^2=\frac{4x_n}{|x+\mathbf{e}_{n}|^2}.
		\end{align}
		Clearly, $\mathcal{S}$ is the inversion with respect to the sphere centered at $-\mathbf{e}_{n}$ with radius $\sqrt{2}$, and its inverse satisfies $\mathcal{S}^{-1}=\mathcal{S}$. Given $f\in L^{\infty}(\S^{n-1})$, in order to preserve conformal invariance, we introduce the corresponding function on $\partial \R^n_{+}$ by
		\begin{align}\label{Sec 4 Def-u}
			u(x')=f\circ \mathcal{S}(x',0)+(n-1)\log\frac{2}{1+|x'|^2}:= f\circ \mathcal{S}(x',0)+(n-1)f_0,
		\end{align}
		Then
		\begin{align*}
			\int_{\R^{n-1}}e^{u(x')}\ud x'=|\S^{n-1}|\int_{\S^{n-1}}e^{f}\ud \mu.
		\end{align*}
		By the conformal properties of $\mathcal{S}$, together with \eqref{Sec 4 equ-a} and \eqref{Sec 4 equ-b}, we obtain
		\begin{align*}
			P(f)\circ \mathcal{S}(x)=&\frac{1}{|\S^{n-1}|}\int_{\R^{n-1}}\left(\frac{1-|\xi|^2}{|\xi-\eta|^2}\right)^{n-1}f\circ \mathcal{S}(y',0)\left(
			\frac{2}{1+|y'|^2}\right)^{n-1}\, 
			\ud y'\\
			=&\frac{2^{n-1}}{|\S^{n-1}|}\int_{\R^{n-1}}\left(\frac{x_n}{|x'-y'|^2+x_n^2}\right)^{n-1}f\circ \mathcal{S}(y',0)\, 
			\ud y'.
		\end{align*}
		Therefore, we conclude that
		\begin{align}\label{P-P}
		    P(f)\circ \mathcal{S}=\mathcal{P}(f\circ \mathcal{S}),
				\end{align}
		where $\mathcal P$ is the Poisson operator on $\R^n_{+}$ defined by
		\begin{align}\label{Sec 4 Def Pu}
			\mathcal{P}(u)(x',x_n)=\frac{2^{n-1}}{|\S^{n-1}|}\int_{\R^{n-1}}\left(\frac{x_n}{|x'-y'|^2+x_n^2}\right)^{n-1}u(y')\ud y',
		\end{align}
		and it is not difficult to verify that $\mathcal{P}(u)$ satisfies the following degenerate elliptic equation in $\R^n_{+}$:
		\begin{align}\label{Sec 4 PDE}
			\begin{cases}
				\displaystyle	\Delta \mathcal{P}(u)+\frac{2-n}{x_n}\partial_{x_n}\mathcal{P}(u)=0;\qquad&\mathrm{in}\qquad \R^n_{+},\\
				\displaystyle \mathcal{P}(u)(x',0)=u(x')\qquad&\mathrm{on}\qquad \partial\R^n_{+}.
			\end{cases}
		\end{align}
We note that, although the PDE \eqref{Sec 4 PDE} is degenerate at \(x_n=0\), the strong maximum principle still implies that, if \(\mathcal{P}(u)\not\equiv \mathrm{const}\), then a local minimum of the solution cannot be attained in the interior of \(\R^n_{+}\).

		\begin{lem}\label{Sec 3 Equ Rn Lem}
		Let $u$ be defined by \eqref{Sec 4 Def-u}, and let $\mathcal{P}(u)$ be defined by \eqref{Sec 4 Def Pu}. Then Theorem \ref{Sec 2 thm} is equivalent to
			\begin{align}\label{Sec 3 Equ Rn Lem euq}
			\left(\int_{\R_{+}^n}e^{m_{\sigma}\mathcal{P}(u)(x)}x_n^{\frac{2n(1-\sigma)}{2\sigma-1}}\ud x\right)^{1/m_{\sigma}}\leq c_{\sigma}\int_{\R^{n-1}}e^{u(x')}\ud x',
        \end{align}
		where $m_{\sigma}=\frac{n}{(2\sigma-1)(n-1)}$, $\sigma\in\left(\frac{1}{2},\frac{2n-1}{2(n-1)}\right)$, and the sharp constant $c_{\sigma}$ is given by
		\begin{align*}
			c_{\sigma}=|\mathbb S^{n-1}|^{-1}e^{-2(n-1)\left\{(1-\sigma)\log 2+I(1)\right\}}\left(\int_{\B^n} e^{\frac{2n}{2\sigma-1}\phi}\right)^{1/m_{\sigma}} .
		\end{align*}
		
		\end{lem}
		\begin{pf}
			
		For $\eta\in \S^{n-1}$ with $\eta=\mathcal{S}(y',0)$, we have
		\begin{align*}
			1+\eta_n(y)=\frac{2}{1+|y'|^2}.
		\end{align*}
		Therefore, by Lemma \ref{Conformal lem 2}, taking $a\to-\mathbf{e}_{n}$, we obtain
		\begin{align}\label{Sec 4 Pf0}
			\mathcal{P}(f_0)=	P(\log (1+\eta_n))\circ \mathcal{S}=\left[\log |\xi +\mathbf{e}_{n}|^2 -2I\left(1\right)+2I(|\xi|^2)\right]\circ \mathcal{S}-\log 2.
		\end{align}
		Hence,
		\begin{align*}
			-\frac{1}{2}&\mathcal{P}(f_0)(x)+ [(1-\sigma)\log(1-|\xi|^2)+I(|\xi|^2)]\circ \mathcal{S}+\frac{2\sigma-1}{2}\log\frac{2}{|x+\mathbf{e}_{n}|^2 } \\
			=&\left(1-\sigma\right)\log 2+I(1)+(1-\sigma)\log x_n.
		\end{align*}
		By a change of variables, we get
		\begin{align*}
			&\int_{\B^n}e^{m_{\sigma}P(f)(\xi)}e^{\frac{2n}{2\sigma-1}\left[(1-\sigma)\log(1-|\xi|^2)+I(|\xi|^2)\right]}\ud \xi\\
			=&\int_{\R_{+}^n}e^{m_{\sigma}\mathcal{P}(f\circ \mathcal{S})(x)}e^{\frac{2n}{2\sigma-1}\left[(1-\sigma)\log(1-|\xi|^2)+I(|\xi|^2)\right]\circ \mathcal{S}}\left(\frac{2}{|x'|^2+(1+x_n)^2}\right)^n\ud x\\
			=&\int_{\R_{+}^n}e^{m_{\sigma}\mathcal{P}(u)(x)}e^{ \frac{2n}{2\sigma-1}\left[-\frac{1}{2}\mathcal{P}(f_0)(x)+ [(1-\sigma)\log(1-|\xi|^2)+I(|\xi|^2)]\circ \mathcal{S}+\frac{2\sigma-1}{2}\log\frac{2}{|x+\mathbf{e}_{n}|^2 }\right]}  \ud x\\
			=&e^{\frac{2n}{2\sigma-1}((1-\sigma)\log 2+I(1))}\int_{\R_{+}^n}e^{m_{\sigma}\mathcal{P}(u)(x)}x_n^{\frac{2n(1-\sigma)}{2\sigma-1}}\ud x.
		\end{align*}
		Thus, the original inequality can be transformed into an equivalent inequality on $\R^{n}_{+}$.
		\end{pf}

		\section{Classification of extremal functions and its application}\label{Sec 4}

		\subsection{Moving Sphere approach}
       In this subsection, we apply the moving sphere method to classify solutions of the Euler--Lagrange equation \eqref{Sec 3 Equ Rn Lem euq},
       \begin{align}\label{Sec 4 E-L equ-s}
		e^{u(x')}=c'_{\sigma}\int_{\R^{n}_{+}}e^{m_{\sigma}\mathcal{P}(u)(y)}y_n^{\frac{2n(1-\sigma)}{2\sigma-1}}\mathscr{P}(x'-y',y_n)\ud y,
            \end{align}
where
			\begin{align*}
				\mathscr{P}(x',x_n) =\left(\frac{x_n}{|x'|^2+x_n^2}\right)^{n-1}.
            \end{align*}
       
     In Appendix \ref{Appendix}, if $f\in L^{\infty}(\S^{n-1})$, we will prove that every extremal function \eqref{Sec 4 E-L equ-s} is of class $C^1$. Accordingly, throughout the discussion below, we assume that $u\in C^1(\R^{n-1})$. To establish the conformal property under the Kelvin transformation, we first need the following technical lemma.
		\begin{lem}\label{Sec 4 Bubble lem}
	For any \(x=(x',x_n)\in \R^n_+\) and \(x_0=(x_0',0)\in \partial\R^n_+\), if $\lambda\in \R_{+}$, then
	\begin{align*}
		\frac{2^{n-1}}{|\S^{n-1}|}\int_{\R^{n-1}}
		\left(\frac{x_n}{|x'-y'|^2+x_n^2}\right)^{n-1}
		\log \frac{\lambda^2}{|y'-x_0'|^2}\ud y'
		=
		\log\frac{\lambda^2}{|x'-x_0'|^2+x_n^2}.
	\end{align*}
\end{lem}

\begin{pf}
	By the change of variables
	\[
	y'=x'+x_n z',
	\]
	we obtain
	\begin{align*}
		&\frac{2^{n-1}}{|\S^{n-1}|}\int_{\R^{n-1}}
		\left(\frac{x_n}{|x'-y'|^2+x_n^2}\right)^{n-1}
		\log \frac{\lambda^2}{|y'-x_0'|^2}\ud y' \\
		&=\frac{2^{n-1}}{|\S^{n-1}|}\int_{\R^{n-1}}
		\frac{1}{(1+|z'|^2)^{n-1}}
		\log \frac{\lambda^2}{|x'-x_0'+x_n z'|^2}\,dz' \\
		&=\log\frac{\lambda^2}{x_n^2}
		-\frac{2^{n-1}}{|\S^{n-1}|}\int_{\R^{n-1}}
		\frac{\log \left|z'+\frac{x'-x_0'}{x_n}\right|^2}{(1+|z'|^2)^{n-1}}\,dz'.
	\end{align*}
	
	Now let \(\mathcal{S}|_{\S^{n-1}}:\R^{n-1}\to \S^{n-1}\) be the stereographic projection. Then
	\[
	\frac{1}{|\S^{n-1}|}\frac{2^{n-1}}{(1+|z'|^2)^{n-1}}\,dz'=d\mu,
	\]
	and
	\[
	\left|z'+\frac{x'-x_0'}{x_n}\right|^2
	=
	\frac{(1+|z'|^2)\left(1+\frac{|x'-x_0'|^2}{x_n^2}\right)}{4}
	\left|\mathcal{S}(z')-\mathcal{S}\!\left(-\frac{x'-x_0'}{x_n}\right)\right|^2.
	\]
	Set $\eta=\mathcal{S}(z')$ and $\eta_0=\mathcal{S}\left(-\frac{x'-x_0'}{x_n}\right)$. Then we may rewrite the integral as
	\begin{align*}
		\frac{2^{n-1}}{|\S^{n-1}|}\int_{\R^{n-1}}
		\frac{\log \left|z'+\frac{x'-x_0'}{x_n}\right|^2}{(1+|z'|^2)^{n-1}}\,dz'
		=&
		\log\left(1+\frac{|x'-x_0'|^2}{x_n^2}\right)+	\int_{\S^{n-1}}\log|\eta-\eta_0|^2\ud \mu(\eta)+C'\\
		=&	\log\left(1+\frac{|x'-x_0'|^2}{x_n^2}\right)+C,
	\end{align*}
	where \(C\) is independent of \(x\) and \(x_0\). Indeed, the remaining terms are
	either independent of \(\frac{x'-x_0'}{x_n}\), or depend on it only through
	\[
	\int_{\S^{n-1}}\log|\eta-\eta_0|^2\ud \mu(\eta),
	\]
	which is independent of \(\eta_0\) by rotational invariance. To determine \(C\), we take \(x'=x_0'\). Then
	\[
	C=
	\frac{2^{n-1}}{|\S^{n-1}|}\int_{\R^{n-1}}
	\frac{\log|z'|^2}{(1+|z'|^2)^{n-1}}\,dz'.
	\]
	Using the inversion \(z'\mapsto z'/|z'|^2\), the above integral changes sign, and hence it is zero. Therefore \(C=0\).
	
	Consequently,
	\begin{align*}
		&\frac{2^{n-1}}{|\S^{n-1}|}\int_{\R^{n-1}}
		\left(\frac{x_n}{|x'-y'|^2+x_n^2}\right)^{n-1}
		\log \frac{\lambda^2}{|y'-x_0'|^2}\ud y' \\
		&=\log\frac{\lambda^2}{x_n^2}
		-\log\left(1+\frac{|x'-x_0'|^2}{x_n^2}\right)
		=\log\frac{\lambda^2}{|x'-x_0'|^2+x_n^2}.
	\end{align*}
	This completes the proof.
\end{pf}

	Now, we collect some basic notation in order to begin the moving sphere procedure.	
		For any $x_0=(x'_0,0)\in \partial \R^n_{+}$ and $\lambda>0$, we define the upper ball $\B^{+}_{\lambda,x_0}=\{x\in \R^{n}_{+}:|x-x_0|<\lambda\}$ and the Kelvin transformation from $\R^n_{+}$ to $\R^n_{+}$ by
		\begin{align*}
			\mathcal{S}_{\lambda,x_0}(x)=x_0+\frac{\lambda^2}{|x-x_0|^2}(x-x_0),
		\end{align*}
		then we see $\mathcal{S}_{\lambda,x_0}:\B^{+}_{\lambda,x_0}\to \B^{+,c}_{\lambda,x_0}:=\R^{n}_{+}\backslash \B^{+}_{\lambda,x_0}$, $\mathcal{S}_{\lambda,x_0}\circ \mathcal{S}_{\lambda,x_0}=\mathrm{Id}$ and 
		\begin{align*}
			\mathcal{S}_{\lambda,x_0}^{*}|\ud x|^2=\left(\frac{\lambda}{|x-x_0|}\right)^4|\ud x|^2.
		\end{align*}
		For any $u\in C^{1}(\R^{n-1})$, we still use the notation $\mathcal{S}_{\lambda,x_0}(x')=\mathcal{S}_{\lambda,x_0}(x',0)$, and we define the conformal transform of $u$ with respect to $\mathcal{S}_{\lambda,x_0}$ by
		\begin{align}\label{Sec 4 Def-ulambda}
			u_{\lambda,x_0}=u\circ \mathcal{S}_{\lambda,x_0}+(n-1)\log \frac{\lambda^2}{|x'-x'_0|^2}.
		\end{align}
		Then, we have
		\begin{align}\label{Sec 4 Def-ulambda-1}			\mathcal{P}(u)(\mathcal{S}_{\lambda,x_0}(x))=&\frac{2^{n-1}}{|\S^{n-1}|}\int_{\R^{n-1}}\left(\frac{\frac{\lambda^2x_n}{|x-x_0|^2}}{|\mathcal{S}_{\lambda,x_0}(x)-\mathcal{S}_{\lambda,x_0}(y',0)|^2}\right)^{n-1}u\circ \mathcal{S}_{\lambda,x_0}(y')\left(\frac{\lambda}{|y'-x'_0|}\right)^{2(n-1)}\ud y'\nonumber\\
			=&\frac{2^{n-1}}{|\S^{n-1}|}\int_{\R^{n-1}}\left(\frac{x_n}{|x'-y'|^2+x_n^2}\right)^{n-1}u\circ \mathcal{S}_{\lambda,x_0}(y')\ud y'=\mathcal{P}(u\circ \mathcal{S}_{\lambda,x_0})(x),
		\end{align}
		where we have used the identity 
		\begin{align*}
			|\mathcal{S}_{\lambda,x_0}(x)-\mathcal{S}_{\lambda,x_0}(y',0)|^2=\frac{\lambda^4(|x'-y'|^2+x_n^2)}{|x-x_0|^2|y'-x'_0|^2}.
		\end{align*}
		Using Lemma \ref{Sec 4 Bubble lem}, we obtain
        \begin{align*}
        \mathcal{P}(u)\circ\mathcal{S}_{\lambda, x_0}(y)=\mathcal{P}(u_{\lambda,x_0})(y)-2(n-1)\log\frac{\lambda}{|y-x_0|},
        \end{align*}
      and we can therefore derive the integral equation satisfied by $u_{\lambda,x_0}$ through the following straightforward computation:
		
		\begin{align*}
			e^{u_{\lambda,x_0}(x')}=&c'_{\sigma}\int_{\R^{n}_{+}}e^{m_{\sigma}\mathcal{P}(u)(y)}y_n^{\frac{2n(1-\sigma)}{2\sigma-1}}\mathscr{P}(\mathcal{S}_{\lambda,x_0}(x')-y',y_n)\left(\frac{\lambda}{|x'-x'_0|}\right)^{2(n-1)}\ud y  \\
			=& c'_{\sigma}\int_{\R^{n}_{+}}e^{m_{\sigma}\mathcal{P}(u)\circ \mathcal{S}_{\lambda,x_0}(y)}\left(\frac{\lambda^2y_n}{|y-x_0|^2}\right)^{\frac{n-2\sigma+1}{2\sigma-1}}|\mathcal{S}_{\lambda,x_0}(x')-\mathcal{S}_{\lambda,x_0}(y)|^{2-2n}\\
            &\times\left(\frac{\lambda}{|x'-x'_0|}\right)^{2(n-1)} \left(\frac{\lambda}{|y-x_0|}\right)^{2n}\ud y \\
			=&c'_{\sigma}\int_{\R^{n}_{+}}e^{m_{\sigma}\mathcal{P}(u_{\lambda,x_0})(y)}\left(\frac{\lambda^2y_n}{|y-x_0|^2}\right)^{\frac{n-2\sigma+1}{2\sigma-1}}\left(\frac{|y-x_0|^2|x'-x'_0|^2}{\lambda^4\left[|y'-x'|^2+y_n^2\right] }\right)^{n-1}\\
            &\times\left(\frac{\lambda}{|x'-x'_0|}\right)^{2(n-1)} \left(\frac{\lambda}{|y-x_0|}\right)^{\frac{4n(\sigma-1)}{2\sigma-1}}\ud y \\
			=& c'_{\sigma} \int_{\R^{n}_{+}}e^{m_{\sigma}\mathcal{P}(u_{\lambda,x_0})(y)}y_n^{\frac{2n(1-\sigma)}{2\sigma-1}}\mathscr{P}(x'-y',y_n)\ud y.
		\end{align*}

		With the above formulas at hand, we are able to derive the key integral representation formula for $e^{u_{\lambda,x_0} }-e^{u}$. This formula will serve as the core tool for obtaining the contradiction in Lemma \ref{MS stpe 3 Lem}.
		
		\begin{lem}\label{Sec 4 lem1}
			For $x_0=(x'_0,0)\in \partial\R^{n}_+$ and $\lambda>0$, it follows that 
			\begin{align*}
				e^{u_{\lambda,x_0}(x') }-e^{u(x')}=&c'_{\sigma} \int_{\B^{+}_{\lambda,x_0}}y_n^{\frac{2n(1-\sigma)}{2\sigma-1}} \left[e^{m_{\sigma}\mathcal{P}(u_{\lambda,x_0})}-e^{m_{\sigma}\mathcal{P} (u)}\right]J_{\lambda}(x',y) \ud y				
			\end{align*}
			where
			\begin{align*}
				J_{\lambda}(x',y)=\mathscr{P}(x'-y',y_n) -\left(\frac{\lambda}{|x'-x'_0|}\right)^{2(n-1)}\mathscr{P}(\mathcal{S}_{\lambda,x_0}(x')-y',y_n) 
			\end{align*}
            satisfies $J_{\lambda}(x',y)>0$ for  $(x',0), y\in \B^{+}_{\lambda,x_0}$.
		\end{lem}
		\begin{pf}
			We first decompose $\R^n_{+}$ into $\B^{+}_{\lambda,x_0}$ and $\B^{+,c}_{\lambda,x_0}$. Then
			\begin{align*}
				&\int_{\B^{+,c}_{\lambda,x_0}}e^{m_{\sigma}\mathcal{P}(u)(y)}y_n^{\frac{2n(1-\sigma)}{2\sigma-1}}\mathscr{P}(x'-y',y_n)\ud y \\
				=&\int_{\B^{+}_{\lambda,x_0}}e^{m_{\sigma}\mathcal{P}(u)\circ \mathcal{S}_{\lambda,x_0}(y)}\left(\frac{\lambda^2y_n}{|y-x_0|^2}\right)^{\frac{n-2\sigma+1}{2\sigma-1}}|x'-\mathcal{S}_{\lambda,x_0}(y)|^{2-2n} \left(\frac{\lambda}{|y-x_0|}\right)^{2n}\ud y\\
				=&\int_{\B^{+}_{\lambda,x_0}}e^{m_{\sigma}\mathcal{P}(u_{\lambda,x_0} )(y)}\left(\frac{\lambda^2y_n}{|y-x_0|^2}\right)^{\frac{n-2\sigma+1}{2\sigma-1}}\left(\frac{|y-x_0|^2|\mathcal{S}_{\lambda,x_0}(x')-x'_0|^2}{\lambda^4\left[(y'-\mathcal{S}_{\lambda,x_0}(x'))^2+y_n^2\right] }\right)^{n-1} \left(\frac{\lambda}{|y-x_0|}\right)^{\frac{4n(\sigma-1)}{2\sigma-1}}\ud y\\
				=&\int_{\B^{+}_{\lambda,x_0}}e^{m_{\sigma}\mathcal{P}(u_{\lambda,x_0} )(y)}\left(\frac{|\mathcal{S}_{\lambda,x_0}(x')-x_0'|}{\lambda}\right)^{2(n-1)}y_n^{\frac{2n(1-\sigma)}{2\sigma-1}} \mathscr{P}(\mathcal{S}_{\lambda,x_0}(x')-y',y_n)\ud y
			\end{align*}
			and, similarly,
			\begin{align*}
				&\int_{\B^{+,c}_{\lambda,x_0}}e^{m_{\sigma}\mathcal{P}(u_{\lambda,x_0} )(y)}y_n^{\frac{2n(1-\sigma)}{2\sigma-1}}\mathscr{P}(x'-y',y_n)\ud y \\
				=&\int_{\B^{+}_{\lambda,x_0}}e^{m_{\sigma}\mathcal{P}(u )(y)}\left(\frac{|\mathcal{S}_{\lambda,x_0}(x')-x_0'|}{\lambda}\right)^{2(n-1)}y_n^{\frac{2n(1-\sigma)}{2\sigma-1}} \mathscr{P}(\mathcal{S}_{\lambda,x_0}(x')-y',y_n)\ud y.
			\end{align*}
			Therefore, using the identity $|\mathcal{S}_{\lambda,x_0}(x')-x_0'|^2=\frac{\lambda^4}{|x'-x'_0|^2}$, we obtain the desired formula. Furthermore, if $(x',0), y\in \B^{+}_{\lambda,x_0}$, then
			\begin{align*}
				&\mathscr{P}(x'-y',y_n) -\left(\frac{\lambda}{|x'-x'_0|}\right)^{2(n-1)}\mathscr{P}(\mathcal{S}_{\lambda,x_0}(x')-y',y_n)  \\
				=&y_n^{n-1}\left[\frac{1}{\left(|x'-y'|^2+y_n^2\right)^{n-1}}-\left(\frac{\lambda}{|x'-x'_0|}\right)^{2(n-1)} \frac{1}{\left(|\mathcal{S}_{\lambda,x_0}(x')-y'|^2+y_n^2\right)^{n-1}} \right]>0,
			\end{align*}
			which follows from the identity
			\begin{align*}
				&\frac{|x'-x_0'|^2}{\lambda^2}\left(|\mathcal{S}_{\lambda,x_0}(x')-y'|^2+y_n^2\right)-(|x'-y'|^2+y_n^2)\\
				=&\frac{(|x'-x_0'|^2-\lambda^2)(|y'-x_0'|^2+y_n^2-\lambda^2)}{\lambda^2}>0.
			\end{align*}
			This completes the proof.
		\end{pf}
		\textbf{Step 1}: Start the moving sphere procedure.
		
		In this step, we shall use the asymptotic behavior of $u_{\lambda,x_0}$ to show that $e^{u_{\lambda,x_0}}>e^{u}$ in $\partial \R^n_{+}\cap\B^{+}_{\lambda,x_0}$ for sufficiently small $\lambda$. Here and as follows, keeping in mind that $u$ is defined in $\mathbb{R}^{n-1}$, we simply write $x'\in \B^{+}_{\lambda,x_0}$ to mean $(x',0)\in \B^{+}_{\lambda,x_0}$.

By Theorem \ref{cor:bootstrap_C1} and the subsequent argument, $f\in C^{\alpha}(\S^{n-1})$ and $u\in C^1_{\loc}(\R^{n-1})$.  Hence the following argument is justified.
        
		\begin{lem}
			Let $u$ be defined by \eqref{Sec 4 Def-u} and let $u_{\lambda,x_0}$ be defined by \eqref{Sec 4 Def-ulambda}. Then, for any $x_0\in \partial\R^{n}_{+}$, there exists $\lambda_0>0$ such that for every $\lambda\in (0,\lambda_0)$ we have 
			\begin{align*}
				e^{u_{\lambda,x_0}}>e^{u} \qquad\mathrm{in}\qquad \B^{+}_{\lambda,x_0}.
			\end{align*}
		\end{lem}
		\begin{pf}
			First, it suffices to prove that 
			\begin{align}\label{Sec 4 eu equ}
				e^{u_{\lambda,x_0}}<e^{u} \qquad\mathrm{in}\qquad \B^{+,c}_{\lambda,x_0}.
			\end{align}
			Indeed, since $\mathcal{S}_{\lambda,x_0}\circ \mathcal{S}_{\lambda,x_0}=Id$, we have $(u_{\lambda,x_0})_{\lambda,x_0}=u$. Therefore, if \eqref{Sec 4 eu equ} holds, then for any $x'\in \B^{+}_{\lambda,x_0}$, we have $\mathcal{S}_{\lambda,x_0}(x')\in \B^{+,c}_{\lambda,x_0}$, and hence 
			\begin{align*}
				e^{u(x')}=\left(\frac{\lambda}{|x'-x'_0|}\right)^{2(n-1)}e^{u_{\lambda,x_0}(\mathcal{S}_{\lambda,x_0}(x'))}<\left(\frac{\lambda}{|x'-x'_0|}\right)^{2(n-1)}e^{u(\mathcal{S}_{\lambda,x_0}(x')) }=e^{u_{\lambda,x_0}(x')} \quad\mathrm{in}\quad \B^{+}_{\lambda,x_0}.
			\end{align*}
			
			Next, set $r=|x'-x'_0|$. Since $u$ is $C^1$, we can choose $r_0>0$ sufficiently small such that 
			\begin{align*}
				\frac{\ud }{\ud r}\left(r^{n-1}e^{u}\right)>0,\quad\text{for}\quad 0<r\leq r_0.
			\end{align*}
			It then follows that whenever $r_0\geq|x'-x_0'|>\lambda$,
			\begin{align*}
				|x'-x_0'|^{n-1}e^{u(x')}>\left(\frac{\lambda^2}{|x'-x'_0|}\right)^{n-1}e^{u(\frac{\lambda^2(x'-x'_0)}{|x'-x'_0|^2}+x'_0)}.
			\end{align*}
			On the other hand, if $|x'-x'_0|>r_0$, then there exists $\lambda_1>0$ sufficiently small such that $\frac{\lambda^2}{|x'-x'_0|}<\frac{\lambda^2}{r_0}<1$ for all $\lambda<\lambda_1$. Since $u\in C^{1}(\R^{n-1})$ and $f\in C^{\alpha'}(\mathbb S)$, by \eqref{Sec 4 Def-u} we have 
			\begin{align*}
				u(0)=f(\mathbf{e}_{n})+(n-1)\log 2 \quad\mathrm{and}\quad e^{u(x')}=\left(\frac{2}{1+|x'|^2}\right)^{n-1}e^{f\circ \mathcal{S}(x',0)}=O(|x'|^{-2(n-1)})
			\end{align*}
			as $|x'|\to \infty$. Set $m=\min_{\S^{n-1}}e^{f}$ and $M=\max_{|x'-x'_0|\leq 1}e^{u(x')}$.
			
			\textbf{Case 1}: For $|x'-x'_0|\geq2\max\{r_0+|x_0'|,1\}$, we have $1<|x'|<\frac32|x'-x_0'|$. Letting 
			\begin{align*}
				\lambda<\min\left\{\frac{2}{3}\left(\frac{m}{M}\right)^{\frac{1}{2(n-1)}},\lambda_1\right\},
			\end{align*}
			 we obtain 
			\begin{align*}
				e^{u(x')}>&\left(\frac{1}{|x'|^2}\right)^{n-1}m>\left(\frac{4}{9|x'-x'_0|^2}\right)^{n-1}m >\left(\frac{\lambda^2}{|x'-x_0'|^2}\right)^{n-1}M\\
				\geq&\left(\frac{\lambda^2}{|x'-x'_0|^2}\right)^{n-1}e^{u(\frac{\lambda^2(x'-x'_0)}{|x'-x'_0|^2}+x_0')}=e^{u_{\lambda,x_0}(x')}.
			\end{align*}
			
			\textbf{Case 2}: For $r_0<|x'-x'_0|<2\max\{r_0+|x_0'|,1\}:=R_0$, we have $|x'|<\frac32R_0$. Letting 
            \begin{align*}
				\lambda<\left(\frac{m}{M}\right)^{\frac{1}{2(n-1)}}\frac{r_0}{R_0},
			\end{align*} we estimate
			\begin{align*}
				e^{u(x')}>m\left(\frac{1}{R^2_0}\right)^{n-1}>\left(\frac{\lambda^2}{r_0^2}\right)^{n-1}M>\left(\frac{\lambda^2}{|x'-x'_0|^2}\right)^{n-1}e^{u(\frac{\lambda^2(x'-x'_0)}{|x'-x'_0|^2}+x'_0)}=e^{u_{\lambda,x_0}(x')}.
			\end{align*}

			Thus, we have completed the proof of \eqref{Sec 4 eu equ} by taking 
			\begin{align*}
				\lambda_0=\min\left\{\frac23\left(\frac{m}{M}\right)^{\frac{1}{2(n-1)}},\left(\frac{m}{M}\right)^{\frac{1}{2(n-1)}}\frac{r_0}{R_0}, \lambda_1, r_0\right\}.
			\end{align*}
			
		\end{pf}
		\begin{rem}
			From the definition of $u_{\lambda,x_0}$, and using \eqref{Sec 4 Def-u} and \eqref{Sec 4 Def-ulambda}, near $x_0'$ we have 
			\begin{align*}
				\lim_{x'\to x_0'}	u_{\lambda,x_0}(x')= & \lim_{x'\to x_0'}\left(u\circ \mathcal{S}_{\lambda,x_0}(x')+(n-1)\log \frac{\lambda^2}{|x'-x'_0|^2} \right)\\
				=&\lim_{x'\to x_0'}	f\circ \mathcal{S}( \mathcal{S}_{\lambda,x_0}(x'),0)+(n-1)\lim_{x'\to x_0'}\left(\log\frac{2}{1+| \mathcal{S}_{\lambda,x_0}(x')|^2}+\log \frac{\lambda^2}{|x'-x'_0|^2}\right)\\
				=&f(-\mathbf{e}_{n})+(n-1)\log\frac{2}{\lambda^2}.
			\end{align*}
			Hence, we may regard $u_{\lambda,x_0}$ as a continuous function on $\partial \R^n_{+}$. 
			Moreover, as $\lambda\to 0$, it is clear that $u_{\lambda,x_0}(x_0')>u(x_0')$.
		\end{rem}

		\begin{cor}\label{Sec 4 Cor}
			If $u_{\lambda,x_0}\geq u$ and $u_{\lambda,x_0}\not\equiv u$  in $\B^{+}_{\lambda,x_0} $, then $	\mathcal{P}(u_{\lambda,x_0})> 	\mathcal{P}(u) $ in $\B^{+}_{\lambda,x_0} $. 
		\end{cor}
		\begin{pf}
			This follows from the fact that $\mathcal{P}(u_{\lambda,x_0}) $ and $\mathcal{P}(u)$ both satisfy the same equation \eqref{Sec 4 PDE} in $\R^n_{+}$, with boundary values
			\begin{align*}
				\mathcal{P}(u_{\lambda,x_0})(x',0)=u_{\lambda,x_0}(x')\geq \mathcal{P}(u)(x',0)=u(x') \qquad\mathrm{on}\qquad \partial\R^n_{+}\cap \B^{+}_{\lambda,x_0}.
			\end{align*}
			Moreover, on $\partial  \B^{+}_{\lambda,x_0}\cap \R^n_{+}$, we have $ \mathcal{S}_{\lambda,x_0}(x)=x$. Using $\mathcal{P}(u\circ \mathcal{S}_{\lambda,x_0})=	\mathcal{P}(u)(\mathcal{S}_{\lambda,x_0}(x)) $ (see \eqref{Sec 4 Def-ulambda-1}), and Lemma \ref{Sec 4 Bubble lem} leads to 
            \begin{align*}
                \mathcal{P}\left(\log\frac{\lambda^2}{|\cdot-x_0|^2}\right)(x',x_n)=\log \frac{\lambda^2}{|x'-x_0'|^2+x_n^2}=0\quad\text{for}\quad |x-x_0|=\lambda,
            \end{align*}
            it follows that
			\begin{align}\label{Sec 4 cor equ-a}
				\mathcal{P}(u_{\lambda,x_0})=\mathcal{P}(u) \qquad\mathrm{on}\qquad \partial  \B^{+}_{\lambda,x_0}\cap \R^n_{+}.
			\end{align}
			The strong maximum principle therefore implies that $	\mathcal{P}(u_{\lambda,x_0})> 	\mathcal{P}(u) $ in $\B^{+}_{\lambda,x_0} $.
		\end{pf}
		At this point, we can start the moving sphere process for $\lambda \in (0,\lambda_0)$. Accordingly, we define the limiting position of $\lambda$ by
		\begin{align}\label{Sec 4 limit lambda}
			\hat{\lambda}_0=\sup\{\lambda: e^{u}<e^{u_{\mu,x_0}} ~\mathrm{in}~\partial \R^n_{+}\cap \B^{+}_{\mu,x_0}~\mathrm{for}~ \mu\in (0,\lambda) \}
		\end{align}
		
		\textbf{Step 2}: The limiting position must be finite, namely, $\hat{\lambda}_0<+\infty$.
		\begin{lem}
			Let $\hat{\lambda}_0$ be defined by \eqref{Sec 4 limit lambda}. Then $\hat{\lambda}_0<+\infty$.
		\end{lem}
		\begin{pf}
			Assume by contradiction that $\hat{\lambda}_0=+\infty$. Then for every $\lambda>0$, we would have
			\begin{align*}
				e^{u_{\lambda,x_0}}>e^{u} \quad\mathrm{in}\quad \partial \R^n_{+}\cap\B^{+}_{\lambda,x_0} \qquad\mathrm{and}\qquad \int_{\partial \R^n_{+}\cap\B^{+}_{\lambda,x_0}}e^{u_{\lambda,x_0}}=\int_{\partial \R^n_{+}\cap\B^{+,c}_{\lambda,x_0}}e^{u}.
			\end{align*}
			Hence,
			\begin{align*}
				\int_{\partial \R^n_{+}\cap\B^{+,c}_{\lambda,x_0}}e^{u}=	\int_{\partial \R^n_{+}\cap\B^{+}_{\lambda,x_0}}e^{u_{\lambda,x_0}}>\int_{\partial \R^n_{+}\cap\B^{+}_{\lambda,x_0}}e^{u},
			\end{align*}
			which implies
			\begin{align*}
				\int_{\partial \R^n_{+}\cap\B^{+,c}_{\lambda,x_0}}e^{u}>\frac{1}{2}\int_{\partial\R^{n}_{+}}e^{u} \qquad\mathrm{for \quad any}\qquad \lambda>0.
			\end{align*}
			Letting $\lambda\to +\infty$, we obtain a contradiction.
		\end{pf}
		
		\begin{lem}\label{MS stpe 3 Lem}
			Let $\hat{\lambda}_0$ be defined by \eqref{Sec 4 limit lambda}. Then 
			\begin{align*}
				u(x')\equiv u_{\hat{\lambda}_0,x_0}(x') \qquad\mathrm{for}\qquad x'\in \partial \R^n_{+}.
			\end{align*}
		\end{lem}
		\begin{pf}
			By the definition of $\hat{\lambda}_0$, we have
			\begin{align*}
				u(x')\leq  u_{\hat{\lambda}_0,x_0}(x') \qquad\mathrm{for}\qquad x'\in \partial \R^n_{+}\cap\B^{+}_{\hat{\lambda}_0,x_0}.
			\end{align*}
			We claim that equality must hold everywhere. Suppose not. Then there exists $x'_1\in \partial \R^{n}_{+}\cap\B^{+}_{\hat{\lambda}_0,x_0}$ such that $u(x_1')<u_{\hat{\lambda}_0,x_0}(x_1')$. By the strong maximum principle, together with Corollary \ref{Sec 4 Cor}, we obtain 
			\begin{align*}
				\mathcal{P}(u_{\hat{\lambda}_0,x_0})(x)> 	\mathcal{P}(u)(x) \qquad\mathrm{in}\qquad \B^{+}_{\hat{\lambda}_0,x_0}.
			\end{align*}
			Then, using Lemma \ref{Sec 4 lem1}, we further get $e^{u_{\hat{\lambda}_0,x_0}(x')}>e^{u(x')}$ for all $x'\in \partial \R^n_{+}\cap\B^{+}_{\hat{\lambda}_0,x_0}$, namely,
			\begin{align}\label{Sec 4 MS equ-a0}
			u_{\hat{\lambda}_0,x_0}(x')>u(x') \qquad\mathrm{for\ all}\qquad x'\in \partial \R^n_{+}\cap\B^{+}_{\hat{\lambda}_0,x_0}.
			\end{align}
			By continuity, we can choose $c_0>0$ such that 
			\begin{align}\label{Sec 4 MS equ-a}
				\mathcal{P}(u_{\hat{\lambda}_0,x_0})(x)-	\mathcal{P}(u)(x)>2c_0 \qquad\mathrm{in}\qquad \overline{\B}^{+}_{\hat{\lambda}_0/2,x_0}.
			\end{align}
			Since $	\mathcal{P}(u_{\lambda,x_0}) $ depends continuously on $\lambda$, there exists $\delta>0$ such that
			\begin{align}\label{Sec 4 MS equ-b0}
				\mathcal{P}(u_{\lambda,x_0})(x)-	\mathcal{P}(u)(x)>c_0 \quad\mathrm{in}\quad \overline{\B}^{+}_{\hat{\lambda}_0/2,x_0},\quad\forall \lambda\in (\hat{\lambda}_0, \hat{\lambda}_0+\delta).
			\end{align}
            Clearly the definition of $\hat{\lambda}_0$ implies the existence of a sequence $(\hat{\lambda}_0, \hat{\lambda}_0+\delta)\ni \lambda_i\to \hat{\lambda}_0$ such that for each fixed $i$, there exists $x_i \in \overline{\B}^{+}_{\lambda_i,x_0}$ satisfying
			\begin{align*}
				\mathcal{P}(u_{\lambda_i,x_0})(x_i)-	\mathcal{P}(u)(x_i)=\min_{x\in \overline{\B}^{+}_{\lambda_i,x_0}}	\left(\mathcal{P}(u_{\lambda_i,x_0})(x)-	\mathcal{P}(u)(x)\right)<0.
			\end{align*}
			Indeed, if this minimum were nonnegative for any $\lambda-\hat{\lambda}_0>0$ small, then Lemma \ref{Sec 4 lem1} and \eqref{Sec 4 MS equ-b0} would imply $u_{\lambda,x_0}>u$ in $\B^{+}_{\lambda,x_0}$ for any $\lambda-\hat{\lambda}_0>0$ small, which contradicts the definition of $\hat{\lambda}_0$. 
            
            Since $ \mathcal{P}(u_{\lambda_i,x_0})$ and $\mathcal{P}(u)$ satisfy the same PDE \eqref{Sec 4 PDE}, by the strong maximum principle, together with \eqref{Sec 4 cor equ-a} and \eqref{Sec 4 MS equ-b0}, we conclude that
			\begin{align*}
				x_i=(x_i',0)\in \partial\R^n_{+}\cap \left(\B^{+}_{\lambda_i,x_0}\backslash \overline{\B}^{+}_{\hat{\lambda}_0/2,x_0}\right).
			\end{align*}
			On the boundary $\partial \R^{n}_{+}$, we have $u(x')=\mathcal{P}(u)(x',0)$. Hence
			\begin{align*}
				u_{\lambda_i,x_0}(x'_i)-	u(x'_i)=\min_{x'\in \partial\R^n_{+}\cap\B^{+}_{\lambda_i,x_0}}	u_{\lambda_i,x_0}(x')-	u(x')<0
			\end{align*}
			and 
			\begin{align*}
				\nabla (u_{\lambda_i,x_0}-	u)(x'_i)=0.
			\end{align*}
			Without loss of generality, we may assume that $x_i\to \hat{x}_0=(\hat{x}'_0,0) \in \overline{\B}^{+}_{\hat{\lambda}_0,x_0}\backslash \overline{\B}^{+}_{\hat{\lambda}_0/2,x_0}$, and therefore
			\begin{align}\label{Sec 4 MS equ-b}
				u_{\hat{\lambda}_0,x_0}(\hat{x}'_0 )-	u(\hat{x}'_0 )\leq 0 \qquad\mathrm{and}\qquad 	\nabla (u_{\hat{\lambda}_0,x_0}-	u)(\hat{x}'_0) =0.
			\end{align}
			The inequality in \eqref{Sec 4 MS equ-b} and \eqref{Sec 4 MS equ-a0} force $|\hat{x}'_0 -x'_0 |=\hat{\lambda}_0$. Then, using Lemma \ref{Sec 4 lem1}, we obtain
			\begin{align*}
				0=&\langle\nabla (e^{u_{\hat{\lambda}_0,x_0} }-e^{u})( \hat{x}'_0), \hat{x}'_0 -x'_0\rangle\\
				=&c_{\sigma}^{'} \int_{\B^{+}_{\hat{\lambda}_0,x_0}}y_n^{\frac{2n(1-\sigma)}{2\sigma-1}} \left[e^{m_{\sigma}\mathcal{P}(u_{\hat{\lambda}_0,x_0})}-e^{m_{\sigma}\mathcal{P} (u)}\right]\langle \nabla (J_{\hat{\lambda}_0})(\hat{x}_0',y),\hat{x}'_0 -x'_0\rangle\  \ud y\\
				=&-2(n-1)c'_{\sigma}		\int_{\B^{+}_{\hat{\lambda}_0,x_0}}y_n^{\frac{2n(1-\sigma)}{2\sigma-1}} \left[e^{m_{\sigma}\mathcal{P}(u_{\hat{\lambda}_0,x_0})}-e^{m_{\sigma}\mathcal{P} (u)}\right]\frac{y_n^{n-1}(\hat{\lambda}_0^2-|x_0'-y'|^2-y_n^2)}{(|\hat{x}_0'-y'|^2+y_n^2)^{n}}\ud y<0,
			\end{align*}
			which yields a contradiction. The final equality follows from
			\begin{align*}
				\nabla|_{x'=\hat{x}'_0} \mathscr{P}(x'-y',y_n)=&-\frac{2(n-1)(\hat{x}_0'-y')y_n^{n-1}}{(|\hat{x}_0'-y'|^2+y_n^2)^n}
			\end{align*}
			and
			\begin{align*}
				&\nabla|_{x'=\hat{x}'_0} \left(\frac{\hat{\lambda}_0}{|x'-x'_0|}\right)^{2(n-1)}\mathscr{P}(\mathcal{S}_{\hat{\lambda}_0,x_0}(x')-y',y_n)\\
				=&\frac{2(n-1)(\hat{x}_0'-x_0')(\hat{\lambda}_0^2-|x_0'-y'|^2-y_n^2)y_n^{n-1}}{\hat{\lambda}_0^2(|\hat{x}_0'-y'|^2+y_n^2)^n}-\frac{2(n-1)(\hat{x}_0'-y')y_n^{n-1}}{(|\hat{x}_0'-y'|^2+y_n^2)^n}.
			\end{align*}
		\end{pf}
		\textbf{Step 3}: $u$ must be radially symmetric with respect to some point $y_0'$.
		
		Recall the following classification result due to Y. Y. Li \cite{Li}.
		\begin{lem}[\protect{\cite[Lemma 5.8]{Li}}]\label{Technique Lem 2}
			Let $n\geq 2$, $\nu\in \R$, and $g\in C(\R^{n-1})$. Assume that for any $x\in \R^{n-1}$, there exists $\lambda(x)>0$ such that 
			\begin{align*}
				\left(\frac{\lambda(x)}{|y-x|}\right)^{\nu}g\left(x+\frac{\lambda(x)^2(y-x)}{|y-x|^2}\right)=g(y) \qquad\qquad\mathrm{~~for~~} \qquad y\in\R^{n-1}\backslash\{x\}.
			\end{align*}
			Then there exist $a\geq 0$, $d>0$, and $x_0\in \R^{n-1}$ such that
			\begin{align*}
				g(x)=\pm a\left(\frac{1}{d^2+|x-x_0|^2}\right)^{\nu/2}.
			\end{align*}
		\end{lem}
		
		\begin{lem}
			There exist $\lambda>0$ and $y_0'\in \R^{n-1}$ such that 
			\begin{align*}
				u(x')=(n-1)\ln \frac{2\lambda}{\lambda^2+|x'-y_0'|^2}+C.
			\end{align*}
		\end{lem}
		\begin{pf}
			By Step 2, for any $x=(x',0) \in \partial\R^{n}_{+}$, there is $\lambda(x')>0$ such that 
			\begin{align*}
				u(y')\equiv u\left(x'+\frac{\lambda^2(x')}{|y'-x'|^2}(y'-x')\right)+ (n-1)\log \frac{\lambda^2(x')}{|y'-x'|^2} \qquad\mathrm{for}\qquad y'\in \partial \R^n_{+}\backslash\{x'\}.
			\end{align*}
			Set $g=e^{u}$. Then we have
			\begin{align*}
				\left(\frac{\lambda(x')}{|y'-x'|}\right)^{2(n-1)}g\left(x'+\frac{\lambda^2(x')}{|y'-x'|^2}(y'-x')\right)=g(y')\qquad\mathrm{for}\qquad y'\in \partial \R^n_{+}\backslash\{x'\}.
			\end{align*}
			It then follows from Lemma \ref{Technique Lem 2} that there exist $\lambda>0, C'>0$ and $y_0'\in \R^{n-1}$ such that 
			\begin{align*}
				e^{u(x')}=g(x')=C'\left(\frac{2\lambda}{\lambda^2+|x'-y_0'|^2}\right)^{n-1}.
			\end{align*}
		\end{pf}

This completes the proof of Theorem \ref{Thm 1} by using \eqref{Sec 4 Def-u}.

			\subsection{Two dimensional case and Bergman space}
            
		In this subsection, we aim to prove that our weighted  Carleman inequality for $n=2$ is equivalent to Question \ref{Intro Que} with $\alpha=1$ and $q/\beta=p$, in which case our radial singular weight
$
\left(1-|\xi|^2\right)^{\frac{4(1-\sigma)}{2\sigma-1}}
$
is nontrivial.
        
	\textbf{Proof of Theorem \ref{Thm 2}:}       
			Replacing $f$ by $pf$, using Theorem \ref{Thm 1}, for $\sigma\in\left(\frac{1}{2},\frac{3}{2}\right) $ and $f\in L^{1}(\S^1)$, there holds
			\begin{align*}
				\left(\int_{\B^2}e^{\frac{2p}{2\sigma-1}P(f)(\xi)}\left(1-|\xi|^2\right)^{\frac{4(1-\sigma)}{2\sigma-1}}\ud\xi \right)^{\frac{2\sigma-1}{2p}}\leq \left(\int_{\B^2}\left(1-|\xi|^2\right)^{\frac{4(1-\sigma)}{2\sigma-1}}\right)^{\frac{2\sigma-1}{2p}}\left(\int_{\S^{1}}e^{pf}\ud \mu\right)^{1/p},
			\end{align*}
			where   $P(f)$ is the classical Poisson kernel in $\B^2$, 
			\begin{align*}
				P(f)(\xi)=\int_{\S^1}\frac{1-|\xi|^2}{|\xi-\eta|^2}f(\eta)\, 
				\ud \mu(\eta).
			\end{align*}
			\textbf{Case 1}:$|F|=e^{P(f)}$ for some smooth function $f\in L^1(\S^{1})$.
			Set $u=P(f)$ to be the harmonic extension of $f$, and let $v$ be a harmonic conjugate of $u$. Then define
			\begin{align*}
				F=e^{u+iv} \qquad\mathrm{with}\qquad |F|=e^{u},
			\end{align*}
			so that $F$ is holomorphic in $\B^2$. Set $\alpha=\frac{2}{2\sigma-1}\in (1,+\infty)$. Then $\frac{4(1-\sigma)}{2\sigma-1}=\alpha-2$. Therefore,
			\begin{align*}
				\left(\int_{\B^2}\left(1-|\xi|^2\right)^{\frac{4(1-\sigma)}{2\sigma-1}}\ud \xi\right)^{\frac{2\sigma-1}{2}}=\left(\frac{\pi}{\alpha-1}\right)^{\frac{1}{\alpha}}.
			\end{align*}
			Hence, we may rewrite the inequality as
			\begin{align*}
				\left(\int_{\B^2}|F|^{p\alpha}\frac{(\alpha-1)}{\pi}\left(1-|\xi|^2\right)^{\alpha-2}\ud\xi \right)^{\frac{1}{p\alpha}}\leq \left(\int_{\S^{1}}e^{pf}\ud \mu\right)^{1/p}.
			\end{align*}
			
			Since $\lim_{r\to 1}|F|(re^{i\theta})=e^{f}$ a.e. on $\S^1$, Jensen's inequality gives
			\begin{align*}
				|F|^{p}	=e^{p u}=e^{P(p f)}\leq P(e^{p f}),
			\end{align*}
			which implies that for any $r\in [0,1)$,
			\begin{align*}
				\frac{1}{2\pi}\int_{0}^{2\pi}|F(re^{i\theta})|^{p}\ud \theta&\leq 	\frac{1}{2\pi}\int_{0}^{2\pi}P(e^{p f})(re^{i\theta})\ud \theta\\
               &=\int_{\S^{1}}\left(\frac1{2\pi}\int_0^{2\pi}\frac{1-r^2}{|re^{i\theta}-\eta|^2}\ud\theta\right)e^{pf(\eta)}\ud\mu(\eta) =\int_{\S^{1}}e^{p f}\ud \mu.
			\end{align*}
			Therefore $F\in H^{p}(\B^2)$. Equivalently, our inequality becomes
			\begin{align*}
				\|F\|_{A^{p\alpha}_{\alpha}(\B^2)}\leq \|F\|_{H^{p}(\B^2) }.
			\end{align*}
			If the equality holds, then by Theorem \ref{Thm 1} we obtain
			\begin{align*}
				f(\eta)=c+\frac{1}{p}\log\frac{1-|a|^2}{|a-\eta|^2} \qquad\mathrm{for}\qquad a\in \B^2,
			\end{align*}
			which implies that there exists $a\in \B^2$ such that
			\begin{align*}
				F(z)=\frac{C}{(1-z\bar{a})^{2/p}}.
			\end{align*}
			\textbf{Case 2}: For a general function $F\in H^p(\B^2)$. Since $\log |F|\in L^1(\S^1)$, we have the factorization
\[
F=M_{F}Q_{F}
\]
(see \cite[Theorem 17.17, p.~344]{Rudin}), where
			\begin{align*}
				|Q_{F}|=e^{P(\log |F|)}, \qquad |M_F|=1\quad\mathrm{on}\quad \S^1.
			\end{align*}
			Here $Q_F$ is the outer factor and $M_F$ is the inner factor. More precisely, one may choose
			\begin{align*}
				Q_{F}(z)=c\exp\left\{\frac{1}{2\pi}\int_{-\pi}^{\pi}\frac{e^{it}+z}{e^{it}-z}\log |F(e^{it})|\,\ud t\right\}, \qquad |c|=1,
			\end{align*}
			and then set $M_{F}=F/Q_{F}$. We note that $Q_{F}$ has no zeros in $\B^2$, and hence $M_{F}$ is holomorphic in $\B^2$ and satisfies $|M_{F}|=1$ on $\S^{1}$. In fact, we also have
			\begin{align*}
				|M_{F}|\leq 1 \quad\mathrm{in}\quad \B^2.
			\end{align*}
			This follows from the fact that $\log|M_{F}|$ is subharmonic in $\B^2$ and vanishes on $\S^1$. Consequently, we obtain
			\begin{align}\label{Sec 5 equ-1}
				\|F\|_{A^{p\alpha}_{\alpha}(\B^2)}\leq \|Q_{F}\|_{A^{p\alpha}_{\alpha}(\B^2)}   \qquad\mathrm{and}\qquad\|Q_{F}\|_{H^{p}(\B^2)}=\|F\|_{H^{p}(\B^2)}.
			\end{align}
			On the other hand, by Case 1, we already know that
			\begin{align}\label{Sec 5 equ-2}
				\|Q_{F}\|_{A^{p\alpha}_{\alpha}(\B^2)} \leq\|Q_{F}\|_{H^{p}(\B^2)}.
			\end{align}
			Combining \eqref{Sec 5 equ-1} and \eqref{Sec 5 equ-2}, we conclude that
			\begin{align*}
				\|F\|_{A^{p\alpha}_{\alpha}(\B^2)}\leq \|F\|_{H^{p}(\B^2)}.
			\end{align*}
			When equality holds, $M_F=e^{i\theta}$ for some $\theta \in \R$ and $\log |Q_F|$ take the form in Theorem \ref{Thm 1}. This completes the proof.

		\section{weighted Huber inequality}\label{Sec 5}

        \subsection{Conformally invariant boundary operators}\label{Sec 5.1}
In this subsection, we aim to show that the hyperbolic extension $P(f)$ can also be viewed as a solution to a polyharmonic equation with conformally invariant boundary data, see Theorem \ref{thm:Green}. Based on this observation, we will then derive the corresponding Green representation formula.

	To explain this more clearly, we first recall the standard conformally invariant boundary operators in the upper half-space. For $n=2m$, on $(\R^n_+,\R^{n-1})$, J.~Case \cite{Case2} introduced the general-order conformally invariant boundary operators through the following recursive formulas:
\begin{align*}
\sum_{l=1}^{j}\binom{j}{l}
\frac{\Gamma\left(\frac12+j\right)\Gamma\left(\frac32+2j-2l-m\right)}
{\Gamma\left(\frac12+j-l\right)\Gamma\left(\frac32+2j-l-m\right)}
\overline{\Delta}^l\tilde{B}^{2m-1}_{2j-2l}
=
(-1)^j \partial_t^{2j}-\tilde{B}^{2m-1}_{2j},
\end{align*}
and
\begin{align*}
\sum_{l=1}^{j}\binom{j}{l}
\frac{\Gamma\left(\frac32+j\right)\Gamma\left(\frac52+2j-2l-m\right)}
{\Gamma\left(\frac32+j-l\right)\Gamma\left(\frac52+2j-l-m\right)}
\overline{\Delta}^l\tilde{B}^{2m-1}_{2j+1-2l}
=
(-1)^{j+1} \partial_t^{2j+1}-\tilde{B}^{2m-1}_{2j+1}.
\end{align*}
Here $\overline{\Delta}$ is the induced Laplacian on $\mathbb R^{n-1}$. It is not hard to see that 
\begin{align*}
    \tilde{B}^{2m-1}_{j}=(-1)^{[\frac{j+1}{2}]}\partial_t^j+\sum_{k=1}^{ [\frac{j}{2}]}a_{j,k} \partial_t^{j-2k}\bar{\Delta}^{k}   \end{align*}
for some constant $a_{j,k}$ where $[x]$ denotes the integer part of $x$.

\begin{thm}[\protect{J. Case \cite[Theorem 5.1]{Case2}}]\label{Case thm}
Let $f^{(0)}\in C^{m-1}(\R^{n-1})\cap H^{m-\frac12}(\R^{n-1})$. Then there exists a unique solution $V$ to
\begin{align}\label{eq:half-space-bvp}
\begin{cases}
(-\Delta)^mV=0 & \mathrm{in}\ \R^n_+,\\[0.3em]
\tilde{B}^{2m-1}_0(V)=f^{(0)} & \mathrm{on}\ \partial\R^n_+,\\[0.3em]
\tilde{B}^{2m-1}_j(V)=0,\quad 1\le j\le m-1 & \mathrm{on}\ \partial\R^n_+.
\end{cases}
\end{align}
Moreover, $V$ is given by the hyperbolic Poisson extension of $f^{(0)}$, namely
\begin{align*}
V=\mathcal{P}(f^{(0)})(x',x_n)
=
\frac{2^{n-1}}{|\S^{n-1}|}
\int_{\R^{n-1}}
\left(
\frac{x_n}{|x'-y'|^2+x_n^2}
\right)^{n-1}
f^{(0)}(y')\ud y'.
\end{align*}
\end{thm}

Recall \eqref{conf-map:ball_half-space} that $\mathcal{S}^{-1}:\B^n\to \R^n_+$ is the conformal map from the unit ball onto the upper half-space, and write
\begin{align*}
(\mathcal{S}^{-1})^\ast(|\ud x|^2)
=
\left(\frac{2}{|\xi+\mathbf e_n|^2}\right)^2|\ud \xi|^2
=
e^{2\tau}|\ud \xi|^2.
\end{align*}
Then, for any smooth function $V$ on $\R^n_+$,
\begin{align*}
\bigl((-\Delta_x)^mV\bigr)\circ \mathcal{S}^{-1}
=
e^{-n\tau}(-\Delta_\xi)^m\bigl(V\circ \mathcal{S}^{-1}\bigr).
\end{align*}
Accordingly, for any $F\in C^n( \overline{\B^n})$, we define
\begin{align}\label{Sec 5 Def-B}
B^{2m-1}_j(F)
:=
(-1)^{[\frac{j}{2}]} e^{j\tau}\bigl(\tilde{B}^{2m-1}_j(F\circ \mathcal{S})\bigr)\circ \mathcal{S}^{-1},
\qquad 0\le j\le m-1.
\end{align}
We first note that the operators are defined locally by choosing stereographic charts. The conformal covariance of Case's boundary operators (see \cite[p.4]{Case2}) implies that these local definitions agree on overlaps. Hence \(B_j^{2m-1}\) is globally well-defined on \(\mathbb S^{n-1}\). By the formulas for \(\tilde{B}^{2m-1}_j\), we see that \(B^{2m-1}_j\) must take the form
\begin{align}\label{B formula}
    B^{2m-1}_j
    =
    \partial_{\nu}^j
    +
    \sum_{k=0}^{j-1}D_{j-k} (\partial^k_{\nu}|_{\S^n-1}),
\end{align}
where \(D_k\) is an intrinsic differential operator on \(\S^{n-1}\). Using the conformal covariance above together with \eqref{P-P}, we obtain the following consequence.

\begin{lem}\label{Sec 5 lem}
Let $f\in C^{n}(\S^{n-1})$. Then there exists a unique solution $F$ to
\begin{align}\label{eq:ball-bvp}
\begin{cases}
(-\Delta)^mF=0 & \mathrm{in}\ \B^n,\\[0.3em]
B^{2m-1}_0(F)=f & \mathrm{on}\ \S^{n-1},\\[0.3em]
B^{2m-1}_j(F)=0,\quad 1\le j\le m-1 & \mathrm{on}\ \S^{n-1}.
\end{cases}
\end{align}
Moreover, $F$ is given by the hyperbolic Poisson extension of $f$, namely
\begin{align*}
F=P(f)(\xi)
:=
\int_{\S^{n-1}}
\left(
\frac{1-|\xi|^2}{|\xi-\eta|^2}
\right)^{n-1}
f(\eta)\,\ud \mu(\eta),
\end{align*}
where $\ud\mu$ denotes the normalized surface measure on $\S^{n-1}$.
\end{lem}

\begin{pf}
First, we decompose $f$ as
\begin{align*}
	f=f\chi+f(1-\chi),
\end{align*}
where $\chi$ is a cut-off function near $-\mathbf e_n$ such that
\begin{align*}
	\chi\equiv 0 \quad\mathrm{in}\quad B_{1/2}(-\mathbf e_n),
	\qquad
	\chi\equiv 1 \quad\mathrm{in}\quad
	\S^{n-1}\backslash B_{1}(-\mathbf e_n).
\end{align*}
Then $(f\chi)\circ \mathcal{S}\big|_{\partial\R^n_+}\in C^n_c(\R^{n-1})$. Applying Theorem \ref{Case thm}, we obtain that $P(f\chi)$ solves \eqref{eq:ball-bvp} with
\[
B^{2m-1}_0(P(f\chi))=f\chi
\quad \text{on } \S^{n-1}.
\]

For the other part $f(1-\chi)$, we choose another conformal map $\mathcal{S}'$, given by the inversion with respect to the sphere centered at $\mathbf e_n$ with radius $\sqrt{2}$:
\begin{align*}
	\mathcal{S}'(x)=\mathbf e_n-\frac{2(x-\mathbf e_n)}{|x-\mathbf e_n|^2}.
\end{align*}
Then
\[
(f(1-\chi))\circ \mathcal{S}'\big|_{\partial\R^n_+}\in C^n_c(\R^{n-1}).
\]
Again by Theorem \ref{Case thm}, $P(f(1-\chi))$ solves \eqref{eq:ball-bvp} with boundary value $f(1-\chi)$. By linearity, $P(f)$ solves \eqref{eq:ball-bvp} with boundary value $f$.

It remains to prove uniqueness. If $f\equiv0$, then by the structure of $B^{2m-1}_j$, see \eqref{B formula}, we have
\begin{align*}
\begin{cases}
(-\Delta)^mF=0 & \mathrm{in}\ \B^n,\\[0.3em]
F=\partial_{\nu}F=\cdots=\partial_{\nu}^{m-1}F=0 & \mathrm{on}\ \S^{n-1}.
\end{cases}
\end{align*}
By the uniqueness of the Dirichlet problem for $(-\Delta)^m$, we get $F\equiv 0$. This proves uniqueness and completes the proof.
\end{pf}

Let $G_n(\xi,\zeta)$ be the Green function of $(-\Delta)^{\frac{n}{2}}$ with Dirichlet boundary conditions. Then
the critical Boggio formula for $(-\Delta)^{\frac{n}{2}}$ (see \cite{Boggio} or  \cite[P. 50]{GGS}) gives
	\[
	G_n(\xi,\zeta)
	=
	k_n^{-1}
	\int_1^{\sqrt{[\xi,\zeta]}/|\xi-\zeta|}
	\frac{(v^2-1)^{m-1}}{v^{2m-1}}\ud v.
	\]
	where $k_n=2^{n-2}\Gamma^2(n/2)|\S^{n-1}|$. From \eqref{conformal equ-a}, we have 
	\[
	\frac{\sqrt{[\xi,\zeta]}}{|\xi-\zeta|}=\frac1{|\Psi_\zeta(\xi)|},
	\]
	the change of variables $s=v^{-2}$ yields
	\begin{align}\label{Green}
		G_n(\xi,\zeta)
		&=
		k_n^{-1}\int_1^{1/|\Psi_\zeta(\xi)|}\frac{(1-v^{-2})^{m-1}}{v}\ud v \nonumber\\
		&=
		\frac{k_n^{-1}}{2}\int_{|\Psi_\zeta(\xi)|^2}^{1}\frac{(1-s)^{m-1}}{s}\ud s \nonumber\\
		&=
		k_n^{-1}\left(
		\log \frac1{|\Psi_\zeta(\xi)|}
		+
		J_m(|\Psi_\zeta(\xi)|^2)
		\right),
	\end{align}
	where
    \begin{align*}
      J_m(t):=\frac12\int_t^1 \frac{(1-s)^{m-1}-1}{s}\ud s.
    \end{align*}
	In particular $G_n\ge 0$, since the integrand in the Boggio representation is nonnegative, and $G_n(\cdot,\zeta)=0$ on $\S^{n-1}$ because $|\Psi_\zeta(\xi)|=1$ there.
\begin{thm}\label{thm:Green}
For every $F\in \mathcal H_f$, where
\begin{align*}
        \mathcal H_f
	:=
	\left\{
	u\in C^n( \overline{\B^n}) :
	B^{2m-1}_0u=f,\ B^{2m-1}_1u=\cdots=B^{2m-1}_{m-1}u=0
	\right\},
    \end{align*}
we have the following unique integral representation:
	\[
	F(\xi)
	=
	P(f)(\xi)
	+
	\int_{\B^n}
	G_n(\xi,\zeta)\,
(-\Delta)^m F(\zeta)\,\ud \zeta.
	\]
\end{thm}
\begin{pf}
    By the definition of the Green function, we know that $G_n(\xi,\cdot)\ast (-\Delta)^mF$ is the solution to
   \begin{align*}
\begin{cases}
(-\Delta)^mu=(-\Delta)^mF & \qquad\mathrm{in}~\ \B^n,\\[0.3em]
u=\partial_{\nu}u=\cdots=\partial_{\nu}^{m-1}u=0 & \qquad\mathrm{on}~\ \S^{n-1}.
\end{cases}
\end{align*}
Therefore, setting
\[
\tilde{F}:=F-G_n(\xi,\cdot)\ast (-\Delta)^mF,
\]
we see that $\tilde{F}$ must be $m$-harmonic and satisfy the corresponding boundary conditions. Moreover, by the definition \eqref{Sec 5 Def-B} of $B^{2m-1}_j$, each $B^{2m-1}_j$ is a linear differential operator whose highest-order term is $\partial_{\nu}^j$. It follows that
\begin{align*}
B^{2m-1}_0(\tilde{F})=f, \qquad B^{2m-1}_j(\tilde{F})=0 \qquad 1\leq j\leq m-1.
\end{align*}
The conclusion then follows from Lemma \ref{Sec 5 lem}.
\end{pf}

Although conformally invariant boundary operators are explicitly known in the upper half-space, their concrete formulas remain unknown on general compact manifolds with boundary, even in the case of the standard ball. In what follows, we derive explicit formulas for these conformal boundary operators in our setting, namely for functions $F\in \mathcal{H}_f$.
	 For $f\in C^n(\S^{n-1})$, we can expand $f$ as the spherical harmonics
	\begin{align}\label{Sec 5 expansion}
	    f=\sum_{\ell=0}^\infty \sum_{k=1}^{N_\ell} f_{\ell,k}Y_{\ell,k},
	\qquad
	-\Delta_{\S^{n-1}}Y_{\ell,k}=\lambda_\ell Y_{\ell,k},
	\qquad
	\lambda_\ell=\ell(\ell+n-2).
	\end{align}
    \begin{lem}\label{Sec 5.1 Lem-a}
        Let $f\in C^n(\S^{n-1})$ with the spherical harmonic expansions \eqref{Sec 5 expansion}, then we have
\[
	P(f)(\xi)
	=
	\sum_{\ell,k} f_{\ell,k}R_\ell(r)Y_{\ell,k}(\theta),
	\qquad
	\xi=r\theta,\quad r=|\xi|,\ \theta\in \S^{n-1}.
	\]
    where
	\begin{align}\label{Sec 5.1 R equ}
	R_\ell(r)
	=
	\frac{{}_2F_1(\ell,1-m;\ell+m;r^2)}
	{{}_2F_1(\ell,1-m;\ell+m;1)} r^\ell
	  \end{align}    and ${}_2F_1$ is hypergeometric function defined by \eqref{Hyper}.
    \end{lem}
    \begin{pf}
If $l=0$, the equality is trivial. So, we only focus on $l\geq 1$. 	Because
	\[
	\Delta_{\B}
	=
	\frac{(1-r^2)^2}{4}\Delta
	+\frac{n-2}{2}(1-r^2)\,\xi\cdot \nabla,
	\]
	the equation $\Delta_{\B} P(f)=0$ yields, on each spherical harmonic mode,
	\begin{align}\label{Sec 5.1 equ-a}
	R_\ell''(r)
	+
	\left(\frac{n-1}{r}+\frac{2(n-2)r}{1-r^2}\right)R_\ell'(r)
	-\frac{\lambda_\ell}{r^2}R_\ell(r)=0,
	\qquad
	R_\ell(1)=1.
	\end{align}
	Setting
	\[
	R_\ell(r)=r^\ell F(r^2),\qquad s=r^2,
	\]
	then we can find 
    \begin{align}\label{Sec 5 equ-d}
        F(s)=O(s^{-l/2}) \qquad\mathrm{as}\qquad s\to 0.   \end{align}
By the direct computation, we find that $F$ solves
	\[
	s(1-s)F''+\bigl(\ell+m-(\ell-m+2)s\bigr)F'-\ell(1-m)F=0.
	\]
According to the properties of the hypergeometric equation (see page 1016 in \cite{Gradshteyn&Ryzhik}),  there exist two linear independent solutions $F^{(1)}(s)$, $F^{(2)}(s)$, where
\[
	F^{(1)}(s)={}_2F_1(\ell,1-m;\ell+m;s),
	\]
	and \begin{align*}
	    F^{(2)}(s)\sim s^{1-l-m} \quad\mathrm{as}\quad s\to 0.
	\end{align*}
Then by \eqref{Sec 5 equ-d}, the second solution $F^{(2)}$ disappears and so $F(s)=cF^{(1)}(s)$. By the property of $P(f)$, we have $P(f)|_{\S^{n-1}}=f$,  
    thus we have  the normalization $R_\ell(1)=1$, which gives $F(s)=F^{(1)}(s)/F^{(1)}(1)$.
\end{pf}    
  For \(j=1,\dots,m-1\), we define the intrinsic differential operators on $\S^{n-1}$ as follows:
\begin{align}\label{Sec 5 T_j}
T_j:=p_{j}(-\Delta_{\mathbb S^{n-1}}),
\end{align}
where the polynomials \(p_{j}\) are defined as follows. Set
\[
p_{0}(x)=1,\qquad p_{1}(x)=0.
\]
For \(j\ge2\), set
\[
\begin{aligned}
2(n-j-1)p_{j}(x)
={}&
\Bigl[
10\binom{j-1}{2}
-(4n-10)(j-1)
\Bigr]p_{j-1}(x)
\\
&+
\Bigl[
24\binom{j-1}{3}
-6(n-3)\binom{j-1}{2}
\Bigr]p_{j-2}(x)
\\
&+
\Bigl[
24\binom{j-1}{4}
-6(n-3)\binom{j-1}{3}
\Bigr]p_{j-3}(x)
\\
&-
2x(j-1)p_{j-2}(x)
-
2x\binom{j-1}{2}p_{j-3}(x),
\end{aligned}
\]
with \(p_{q}\equiv0\) for \(q<0\) and $\binom{j}{k}=0$ if $j<k$.
	
\begin{lem}
	For any $u\in \mathcal{H}_{f}$, the conformal boundary operators can be written as
	\begin{align}\label{Boundary Operator}
		B^{2m-1}_0 u:=u|_{\S^{n-1}},
		\qquad
		B^{2m-1}_j u:=\partial_\nu^j u-T_j(u|_{\S^{n-1}}),
		\qquad j=1,\dots,m-1,
	\end{align}
	where $\nu$ is the outward Euclidean unit normal on $\S^{n-1}$.
\end{lem}
\begin{rem}
We should point out that, if \(u\notin\mathcal{H}_{f}\), the conformal boundary operators \(B^{2m-1}_j\) cannot in general be written in the above form. It is only after restricting \(u\) to the special space \(\mathcal{H}_{f}\) that these operators can be simplified as above.
\end{rem}
\begin{pf}
	By Theorem \ref{thm:Green}, it suffices to compute the higher-order normal derivatives of $P(f)$,
	\begin{align*}
		\partial_\nu^j P(f)
		=
		\sum_{\ell,k} f_{\ell,k}R_\ell^{(j)}(1)Y_{\ell,k}.
	\end{align*}
	By Lemma \ref{Sec 5.1 Lem-a}, the numbers $R_\ell^{(j)}(1)$ are uniquely determined. We next claim that
	\[
	R_\ell^{(j)}(1)=p_j(\lambda_\ell).
	\]
	Indeed, we rewrite \eqref{Sec 5.1 equ-a} as
	\[
	aR_\ell''+bR_\ell'-\lambda_\ell cR_\ell=0,
	\]
	where
	\[
	a(r)=r^2-r^4,\qquad
	b(r)=(n-1)r+(n-3)r^3,\qquad
	c(r)=1-r^2.
	\]
    then it follows that the initial data 
    \begin{align*}
        R_\ell^{(0)}(1)=1 \qquad\mathrm{and}\qquad R_\ell^{(1)}(1)=0.
    \end{align*} 
We also note that one can alternatively differentiate the formula \eqref{Sec 5.1 R equ} directly to obtain
$
R_\ell^{(1)}(1)=0.$
	Differentiating this equation \(j-1\) times and evaluating at \(r=1\), we obtain
	\begin{align}\label{eq:t-recursion-general}
		0={}&
		\sum_{i=0}^{j-1}
		\binom{j-1}{i}
		a^{(i)}(1)R_\ell^{(j+1-i)}(1)
		+
		\sum_{i=0}^{j-1}
		\binom{j-1}{i}
		b^{(i)}(1)R_\ell^{(j-i)}(1)
		\nonumber\\
		&-
		\lambda_\ell
		\sum_{i=0}^{j-1}
		\binom{j-1}{i}
		c^{(i)}(1)R_{\ell}^{(j-i-1)}(1).
	\end{align}
	The required derivatives at \(r=1\) are
	\[
	a(1)=0,\qquad a'(1)=-2,\qquad a''(1)=-10,
	\qquad a^{(3)}(1)=-24,\qquad a^{(4)}(1)=-24,
	\]
	and
	\[
	a^{(i)}(1)=0,\qquad i\ge5.
	\]
	Moreover,
	\[
	b(1)=2(n-2),\qquad b'(1)=4n-10,
	\]
	\[
	b''(1)=6(n-3),\qquad b^{(3)}(1)=6(n-3),
	\]
	and
	\[
	b^{(i)}(1)=0,\qquad i\ge4.
	\]
	Finally,
	\[
	c(1)=0,\qquad c'(1)=-2,\qquad c''(1)=-2,
	\]
	and
	\[
	c^{(i)}(1)=0,\qquad i\ge3.
	\]
	In \eqref{eq:t-recursion-general}, the coefficient of \(R_\ell^{(j)}(1)\) is
	\[
	(j-1)a'(1)+b(1)
	=
	-2(j-1)+2(n-2)
	=
	2(n-j-1).
	\]
	Therefore, for \(j\ge2\),
	\begin{align}\label{eq:t-recursion}
		2(n-j-1)R_\ell^{(j)}(1)
		={}&
		\Bigl[
		10\binom{j-1}{2}
		-(4n-10)(j-1)
		\Bigr]R_\ell^{(j-1)}(1)
		\nonumber\\
		&+
		\Bigl[
		24\binom{j-1}{3}
		-6(n-3)\binom{j-1}{2}
		\Bigr]R_\ell^{(j-2)}(1)
		\nonumber\\
		&+
		\Bigl[
		24\binom{j-1}{4}
		-6(n-3)\binom{j-1}{3}
		\Bigr]R_\ell^{(j-3)}(1)
		\nonumber\\
		&-
		2\lambda_\ell(j-1)R_\ell^{(j-2)}(1)
		-
		2\lambda_\ell\binom{j-1}{2}R_\ell^{(j-3)}(1),
	\end{align}
	with the convention that \(R_\ell^{(q)}(1)=0\) for \(q<0\).

	Since \(p_j(\lambda_\ell)\) satisfies the same recursive formula as \(R_\ell^{(j)}(1)\), with the same initial data, we conclude that
	\[
	R_\ell^{(j)}(1)=p_j(\lambda_\ell).
	\]
	Thus
	\[
	\partial_\nu^jP(f)
	=
	\sum_{\ell,k}f_{\ell,k}p_j(\lambda_\ell)Y_{\ell,k}
	=
	T_j(f).
	\]
	Consequently,
	\[
	B^{2m-1}_j(P(f))=\partial_\nu^jP(f)-T_j(f)=0,
	\qquad j=1,\dots,m-1,
	\]
	and clearly
	\[
	B^{2m-1}_0(P(f))=f.
	\]
	This proves the claim.
\end{pf}

\begin{rem}\label{rem:compare-Case-m2m3}
For $m=2,3$, the boundary operators $B_j^{2m-1}$ defined by \eqref{Boundary Operator}
agree with the corresponding lower-order boundary operators constructed by Case \cite{J.Case} and Case--Luo \cite{Case&Luo} on the round ball.
\begin{itemize}
    \item 

When $m=2$, then $n=4$,  one has $p_1(x)=0$,
therefore,
\[
B_1^3u=\partial_\nu u,
\]
which is exactly Case's boundary operator $B_1^3$ in the critical four-dimensional case, see \cite{J.Case}.

\item When $m=3$, then $n=6$,  we obtain $p_1(x)=0, p_2(x)=-\frac{x}{3}$, 
it follows that
\[
B_1^5u=\partial_\nu u,
\qquad
B_2^5u=\partial_\nu^2u-\frac13\Delta_{\S^5}u,
\]
which coincides with the corresponding formulas of Case--Luo \cite{Case&Luo} in the critical six-dimensional case.
\end{itemize}
Hence, for $m=2,3$, the boundary operators obtained from the hyperbolic Poisson extension coincide with the lower-order Case-type conformally covariant boundary operators on the round ball.
\end{rem}

        \subsection{Proof of Theorem \ref{Thm 3}}

 \begin{lem}\label{lem:reverse-Chebyshev}
	Let $I\subset \R$ be an interval, and let $\nu$ be a finite positive measure on $I$ with
	\[
	0<\nu(I)<\infty.
	\]
	Assume that $u,v,uv\in L^1(I,\nu)$, where $u$ is nondecreasing on $I$ and $v$ is nonincreasing on $I$.
	Then
	\begin{equation}\label{eq:reverse-Chebyshev}
		\int_I u(r)v(r)\,\ud\nu(r)
		\le
		\frac{1}{\nu(I)}
		\left(\int_I u(r)\,\ud\nu(r)\right)
		\left(\int_I v(r)\,\ud\nu(r)\right).
	\end{equation}
\end{lem}

\begin{pf}
	Since $u$ is nondecreasing and $v$ is nonincreasing, for every $r,s\in I$ we have
	\[
	(u(r)-u(s))(v(r)-v(s))\le 0.
	\]
	Integrating this inequality over $I\times I$ with respect to $\nu\otimes \nu$, we obtain
	\[
	\iint_{I\times I}(u(r)-u(s))(v(r)-v(s))\,\ud\nu(r)\ud\nu(s)\le 0.
	\]
	Expanding the integrand gives
	\begin{align*}
		0
		&\ge
		\iint_{I\times I}
		\bigl(u(r)v(r)+u(s)v(s)-u(r)v(s)-u(s)v(r)\bigr)\,\ud\nu(r)\ud\nu(s) \\
		&=
		2\nu(I)\int_I u(r)v(r)\,\ud\nu(r)
		-2\left(\int_I u(r)\,\ud\nu(r)\right)\left(\int_I v(r)\,\ud\nu(r)\right).
	\end{align*}
	Hence
	\[
	\nu(I)\int_I uv\,\ud\nu
	\le
	\left(\int_I u\,\ud\nu\right)\left(\int_I v\,\ud\nu\right),
	\]
	which is exactly \eqref{eq:reverse-Chebyshev}.
\end{pf}

Let $\alpha=\frac{n}{2\sigma-1}$. Then Theorem \ref{Thm 3} reduces to the following statement. Our approach is partly inspired by Huber's original proof \cite{Huber}, although several essential modifications and adjustments are needed in our setting.

\begin{thm}\label{thm:Huber}
	Let $n=2m\ge 2$, and let $F\in \mathcal H_f$. Set $\alpha=\frac{n}{2\sigma-1}>n-1$ and 
	\[
	\gamma:=\int_{\B^n}\bigl((-\Delta)^mF(\xi))^{+}\ud \xi\qquad \mathrm{and}\qquad \beta:=\alpha k_n^{-1} \gamma<n,	\]
	 then we have the following high dimensional weighted Huber isoperimetric inequality
	\begin{align}
		\left(\int_{\B^n}
		e^{\alpha F(\xi)}
		\ud \nu\right)^{1/\alpha}
		\le
		M_{\alpha,\beta;I,J_m}
		\left(
		\int_{\S^{n-1}}e^{(n-1)f}\ud \mu
		\right)^{1/(n-1)},
		\label{eq:Huber-final}
	\end{align}
	where 
    \begin{align*}
        \ud \nu=\frac{(1-|\xi|^2)^{\alpha-n}
		e^{2\alpha  I(|\xi|^2)}
		\,d\xi}{\int_{\B^n}(1-|\xi|^2)^{\alpha-n}
		e^{2\alpha  I(|\xi|^2)}
		\,d\xi}
    \end{align*}
    and
	\[
	M_{\alpha,\beta;I,J_m}
	=
	\left[\frac{
		\displaystyle\int_0^1
		r^{n-1-\beta}(1-r^2)^{\alpha-n}
		e^{2\alpha  I(r^2)+\beta J_m(r^2)}\ud r
	}{
		\displaystyle\int_0^1
		r^{n-1}(1-r^2)^{\alpha-n}
		e^{2\alpha I(r^2)}\ud r
	}\right]^{\frac{1}{\alpha}} .
	\]
\end{thm}

\begin{pf}
	By Theorem \ref{thm:Green},
	\[
	F(\xi)
	=
	P(f)(\xi)
	+
	\int_{\B^n}
	G_n(\xi,\zeta)\,(-\Delta)^mF(\zeta)\ud \zeta.
	\]
	Since $G_n\ge 0$ and
	\[
	(-\Delta)^mF=((-\Delta)^mF)^{+}-((-\Delta)^mF)^{-},
	\]
	for $\gamma>0$, set $\ud \hat{\nu}(\zeta):=\frac{((-\Delta)^mF)^{+}(\zeta)}{\gamma}\,\ud \zeta$ , we get
	\[
	F(\xi)
	\le
	P(f)(\xi)
	+
	\gamma\int_{\B^n}
	G_n(\xi,\zeta)\,\ud \hat{\nu}(\zeta).
	\]
	Therefore, using \eqref{Green}, we can see
	\[
	e^{\alpha F(\xi)}
	\le
	e^{\alpha P(f)(\xi)}
	\exp\left(
	\beta\int_{\B^n}
	\left(
	\log \frac1{|\Psi_\zeta(\xi)|}
	+
	J_m(|\Psi_\zeta(\xi)|^2)
	\right)\,\ud \hat{\nu}(\zeta)
	\right).
	\]
	Since the exponential is convex, Jensen's inequality yields
	\[
	e^{\alpha F(\xi)}
	\le
	e^{\alpha P(f)(\xi)}
	\int_{\B^n}
	|\Psi_\zeta(\xi)|^{-\beta}
	e^{\beta J_m(|\Psi_\zeta(\xi)|^2)}
	\,\ud \hat{\nu}(\zeta).
	\]
	Multiply both sides by $(1-|\xi|^2)^{\alpha-n}e^{2\alpha I(|\xi|^2)}$
	and integrate:
	\begin{align}
		\int_{\B^n} e^{\alpha F(\xi)}(1-|\xi|^2)^{\alpha-n}e^{2\alpha I(|\xi|^2)}\,d\xi
		\le
		\int_{\B^n}\mathcal I_\zeta\,\ud \hat{\nu}(\zeta),
		\label{eq:jensen-reduction}
	\end{align}
	where
	\[
	\mathcal I_\zeta
	:=
	\int_{\B^n}
	e^{\alpha P(f)(\xi)}(1-|\xi|^2)^{\alpha-n}e^{2\alpha I(|\xi|^2)}
	|\Psi_\zeta(\xi)|^{-\beta}
	e^{\beta J_m(|\Psi_\zeta(\xi)|^2)}\,d\xi.
	\]
	
	\medskip
	Next, we will focus on how to estimate the term $\mathcal I_\zeta$. The basic idea is using the conformal invariance to change $|\Psi_{\zeta}(\xi)|$ into $|\xi'|$. Fix $\zeta\in\B^n$ and make the change of variables $\xi=\Psi_a(\xi')$. By the formulas \eqref{conformal equ-a},
	\[
	1-|\xi|^2=\frac{1-|a|^2}{[a,\xi']}(1-|\xi'|^2),
	\qquad
	d\xi=\left(\frac{1-|a|^2}{[a,\xi']}\right)^n\,d\xi',
	\]
	and, if $\eta\in \S^{n-1}$, we can also define
	\[
	f_a(\eta):=f\circ\Psi_a(\eta)+\log \frac{1-|a|^2}{[a,\eta]},
	\]
	then using Lemma \ref{Conformal lem 2} and $P(f\circ\Psi_a)=P(f)\circ \Psi_a$, we can see
	\[
	 P(f)(\xi)= P(f_a)(\xi')-\log \frac{1-|a|^2}{[a,\xi']}+2I(|\xi'|^2)-2I(|\Psi_a(\xi')|^2).
	\]
	Using these identities, the factors involving 
    extra term cancel exactly, and we obtain
\begin{align*}
    &e^{\alpha P(f)(\xi)}(1-|\xi|^2)^{\alpha-n}e^{2\alpha I(|\xi|^2)}
	|\Psi_\zeta(\xi)|^{-\beta}
	e^{\beta J_m(|\Psi_\zeta(\xi)|^2)}\ud \xi\\
    =&e^{\alpha P(f_a)(\xi')}(1-|\xi'|^2)^{\alpha-n}e^{2\alpha I(|\xi'|^2)}
	|\Psi_\zeta\circ\Psi_a(\xi')|^{-\beta}
	e^{\beta J_m(|\Psi_\zeta\circ \Psi_a(\xi')|^2)}\ud \xi'\end{align*}
   Since $\Psi_{\zeta}^{-1}=\Psi_{-\zeta}$, choosing $a=-\zeta$, it follows that 
	\begin{align*}
	\mathcal I_\zeta
	=&
	\int_{\B^n}
	e^{\alpha  P (f_{-\zeta})(\xi')}
	(1-|\xi'|^2)^{\alpha-n}
	e^{2\alpha I(|\xi'|^2)}
	|\xi'|^{-\beta}e^{\beta J_m(|\xi'|^2)}
	\,d\xi'\\
    =&\int_{0}^{1}\left(\int_{\S^{n-1}} e^{\alpha P(f_{-\zeta})(r\omega)}\,d\omega\right)
	\left(r^{-\beta}e^{\beta J_m(r^2)} \right) r^{n-1}(1-r^2)^{\alpha-n}
	e^{2\alpha I(r^2)}\ud r \\
    :=&\int_{0}^{1}M_{\zeta}(r)
	e^{\beta g_m(r)} r^{n-1}(1-r^2)^{\alpha-n}
	e^{2\alpha I(r^2)}\ud r,     \end{align*}
	where
	\[
	M_{\zeta}(r):=\int_{\S^{n-1}} e^{\alpha P f_{-\zeta}(r\omega)} \,d\omega \qquad\mathrm{and}\qquad g_m(r):=\log\frac1r+J_m(r^2).
	\]
    Next, we claim that 
    \begin{align}\label{Sec 5 cliam}
        M'_{\zeta}(r)\geq 0\qquad \mathrm{and}\qquad g'_m(r)<0.    \end{align}
	For the first inequality in \eqref{Sec 5 cliam}, since $P(f_{-\zeta})$ is hyperbolic harmonic,
	\[
	\operatorname{div}\!\left((1-|\xi|^2)^{2-n}\nabla  (P(f_{-\zeta}))\right)=0.
	\]
	Therefore
	\begin{align}\label{Sec 5.2 euq-q}
	    \operatorname{div}\!\left((1-|\xi|^2)^{2-n}\nabla e^{\alpha P(f_{-\zeta})}\right)
	=
	\alpha^2(1-|\xi|^2)^{2-n}e^{\alpha P(f_{-\zeta})}|\nabla  P(f_{-\zeta})|^2\ge 0.
	\end{align}
	
	By the divergence theorem,
	\begin{align*}
	(1-r^2)^{2-n}r^{n-1}M_{\zeta}'(r)
	=&
	\int_{\partial B_r}
	(1-|\xi|^2)^{2-n}\partial_\nu e^{\alpha P(f_{-\zeta})}\,d\sigma\\
	=&
	\int_{B_r}
	\operatorname{div}\!\left((1-|\xi|^2)^{2-n}\nabla e^{\alpha P(f_{-\zeta})}\right)\,d\xi
	\ge 0.
	\end{align*}
	Hence
	$
	M_{\zeta}'(r)\ge 0$ for $0<r<1.
	$
	That is, $M_{\zeta}$ is nondecreasing.
	
	To achieve the second inequality in \eqref{Sec 5 cliam}, by the integral formula for $J_m$,
	\[
	g_m(r)=\frac12\int_{r^2}^{1}\frac{(1-s)^{m-1}}{s}\ud s,
	\]
	hence
	\[
	g_m'(r)=-\frac{(1-r^2)^{m-1}}{r}<0.
	\]
Since $M_{\zeta}$ is nondecreasing and $g_m$ is nonincreasing, then Lemma \ref{lem:reverse-Chebyshev} yields
	\begin{align}\label{Sec 5 equ-c}
	    \mathcal I_\zeta
	\le&
	M^{\alpha}_{\alpha,\beta;I,J_m}\int_{0}^{1}M_{\zeta}(r)
	 r^{n-1}(1-r^2)^{\alpha-n}
	e^{2\alpha I(r^2)}\ud r\nonumber \\
    =&M^{\alpha}_{\alpha,\beta;I,J_m}\int_{\B^n}
	e^{\alpha P (f_{-\zeta})(\xi')}
	(1-|\xi'|^2)^{\alpha-n}
	e^{2\alpha I(|\xi'|^2)}
	\,d\xi'	\end{align}
	with
	\[
	M^{\alpha}_{\alpha,\beta;I,J_m}
	=
	\frac{
		\displaystyle\int_0^1 r^{n-1-\beta}(1-r^2)^{\alpha-n}
		e^{2\alpha I(r^2)+\beta J_m(r^2)}\ud r
	}{
		\displaystyle\int_0^1 r^{n-1}(1-r^2)^{\alpha-n}
		e^{2\alpha I(r^2)}\ud r
	}.
	\]
	The condition $\beta<n$ is exactly the integrability condition near $r=0$. 
    Using \eqref{Sec 5 equ-c} and \eqref{Intro Inequ} in Theorem \ref{Thm 1}, we can estimate 
    \begin{align*}
    \frac{\mathcal I_\zeta}{\int_{\B^n} (1-|\xi'|^2)^{\alpha-n}e^{2\alpha I(|\xi'|^2)}\,d\xi'}
    &\le M^{\alpha}_{\alpha,\beta;I,J_m}\int_{\B^n}
		e^{\alpha  P(f_{-\zeta})}\ud\nu
        \le M^{\alpha}_{\alpha,\beta;I,J_m}\|e^{f_{-\zeta}}\|_{L^{n-1}(\S^{n-1},\ud \mu)}^{\alpha}\\
        &=	M^{\alpha}_{\alpha,\beta;I,J_m}\|e^{f}\|_{L^{n-1}(\S^{n-1},\ud \mu)}^{\alpha},\quad\forall \zeta\in\B^n,
    \end{align*}
    where we have used 
    \[
	\int_{\S^{n-1}} e^{(n-1)f_{-\zeta}}\ud \mu
	=
	\int_{\S^{n-1}} e^{(n-1)f}\ud \mu, \quad\forall \zeta \in\B^n
	\]
    to obtain the last equality. Inserting this estimate into \eqref{eq:jensen-reduction}, we get  \eqref{eq:Huber-final}.
	
	It remains to discuss the equality. If \(\gamma=0\), then \((-\Delta)^mF\le 0\) a.e. in \(\mathbb B^n\). Since \(G_n\ge0\), the Green representation gives
\[
F(\xi)\le P(f)(\xi).
\]
Therefore the desired inequality follows directly from Theorem \ref{Thm 1}, with \(M_{\alpha,0;I,J_m}=1\).
Moreover, equality forces \(F=P(f)\) and \(f\) to be an extremal of Theorem \ref{Thm 1}.

 If $\gamma>0$ and $F$ is smooth, the equality in the Jensen inequality would force the probability measure $\hat{\nu}$ to be a Dirac mass, which is impossible since $\hat{\nu}$ is absolutely continuous with respect to $\ud \zeta$. Hence the equality cannot occur in the smooth class.
	
	Finally, in the enlarged singular class allowing one interior atom, for example, 
    \begin{align*}
         F=P(f)+\int_{\B^n} G_n(\xi,\zeta)\ud \hat{\mu}(\zeta)\quad\mathrm{where}~ \ud \hat{\mu} \mathrm{~is~ a~ Radon~ measure}.   
    \end{align*}
 In the enlarged singular class, suppose equality holds. Then equality in the first estimate
\[
F\le P(f)+G_n*\hat\mu_+
\]
forces \(\hat\mu_-=0\). Equality in Jensen's inequality forces
$
\hat\mu_+=\gamma\delta_{\zeta_0}
$
for some \(\zeta_0\in\mathbb B^n\). For this fixed \(\zeta_0\), equality in Lemma \ref{lem:reverse-Chebyshev} implies that
\(M_{\zeta_0}\) is constant. Since
\[
(1-r^2)^{2-n}r^{n-1}M_{\zeta_0}'(r)
=
\alpha^2\int_{B_r}
(1-|\xi|^2)^{2-n}
e^{\alpha P(f_{-\zeta_0})}
|\nabla P(f_{-\zeta_0})|^2\,d\xi,
\]
we obtain
\[
\nabla P(f_{-\zeta_0})\equiv0.
\]

Hence \(f_{-\zeta_0}\equiv c\). On the other hand, by Theorem \ref{Thm 1}, the equality condition also forces \(f=f_{a,c}\). Therefore, necessarily \(a=\zeta_0\). Hence the singular extremals are precisely
\[
F_{a,c}(\xi)=P(f_{a,c})(\xi)+\gamma G_n(\xi,a).
\]
Conversely, these functions make every step above an equality.

To prove the sharpness in the original smooth class, one approximates \(\gamma\delta_a\) by a sequence of nonnegative smooth functions \(\rho_\varepsilon\in C^{\infty}_c(\B^n)\to \gamma\delta_a\) satisfying
$
\int_{\B^n}\rho_\varepsilon=\gamma,
$
and defines
\[
F_\varepsilon(\xi):=P(f_{a,c})(\xi)+\int_{\B^n}G_n(\xi,\zeta)\rho_\varepsilon(\zeta)\,\ud \zeta.
\]
Then each $F_\varepsilon$ belongs to the admissible smooth class and satisfies
$
\int_{\B^n}((-\Delta)^mF_\varepsilon)^{+}\,\ud \zeta=\gamma.
$
Moreover, $F_\varepsilon\to F_{a,c}$ pointwise, hence, by the Fatou's Lemma,
\[
\liminf_{\e\to 0}\frac{\|e^{F_\varepsilon}\|_{L^{\frac n{2\sigma-1}}(\B^n,\ud\nu)}}
{\|e^{f_{a,c}}\|_{L^{n-1}(\S^{n-1},\ud\mu)}}
\geq 
M_{\alpha,\beta;I,J_m}.
\]
 On the other hand, for smooth class, we also have 
 \begin{align*}
     \limsup_{\e\to 0}\frac{\|e^{F_\varepsilon}\|_{L^{\frac n{2\sigma-1}}(\B^n,\ud\nu)}}
{\|e^{f_{a,c}}\|_{L^{n-1}(\S^{n-1},\ud\mu)}}
\leq 
M_{\alpha,\beta;I,J_m}. \end{align*}
 Therefore, the constant $M_{\alpha,\beta;I,J_m}$ cannot be replaced by any smaller one.
\end{pf}
	\begin{rem}
	    When $\sigma=1$, equivalently $\alpha=n$, the quantity $\|e^{F}\|_{L^n(\B^n,\ud \nu)}$ is also conformally invariant. More precisely, if $F|_{\S^{n-1}}=f$, then under the conformal transformation $\Psi_a$ we define
\begin{align*}
    F_{\Psi_a}(\xi)=F\circ \Psi_{a}(\xi)+\log \frac{1-|a|^2}{[a,\xi]}+2I(|\Psi_a(\xi)|^2)-2I(|\xi|^2),
\end{align*}
which can also be characterized by $
\Psi_a^{*}(e^{nF}\ud V_{g^*})=e^{nF_{\Psi_a}}\ud V_{g^*},
$
where $g^*=e^{2\tilde{I}_n(|\xi|^2)}|\ud \xi|^2$ is the adapted metric. If $\xi\in \S^{n-1}$, then this reduces exactly to the conformal transformation of $f$ with respect to $L^{n-1}(\S^{n-1},\ud \mu)$, namely,
\begin{align*}
    F_{\Psi_a}\big|_{\S^{n-1}}(\eta)=f\circ \Psi_{a}(\eta)+\log \frac{1-|a|^2}{|a-\eta|^2}=f_{\Psi_a}(\eta).
\end{align*}
	\end{rem}

Lastly, in the case $m=2$ and $\gamma=0$, we point out that if $F\notin \mathcal{H}_f$, then the desired weighted Huber isoperimetric inequality may fail. Notice that $M_{\alpha,\beta;I,J_m}=1$ if $\gamma=0$ or equivalently $\beta=0$.
	
	\begin{lem}\label{Sec 5 Lem-2}
		For $n=4$ and $f=0$ on $\S^{3}$, there exists a function $F$ satisfying $\Delta^2F=0$, $F\big|_{\S^3}=f=0$ and $F\notin \mathcal{H}_{f}$ such that, for any $\sigma\in\left(\frac{1}{2},\frac{7}{6}\right)$, we have
		\begin{align}
			\|e^{F}\|_{L^{\frac{4}{2\sigma-1}}(\B^4,\ud \nu)}	> \|e^{f}\|_{L^{3}(\S^{3},\ud \mu)},
			\label{Intro:Huber-final-2}
		\end{align}
		where 
		\[
		\ud \nu=\frac{(1-|\xi|^2)^{\frac{8(1-\sigma)}{2\sigma-1}}e^{\frac{8}{2\sigma-1}I(|\xi|^2)}\ud \xi}{\int_{\B^4}(1-|\xi|^2)^{\frac{8(1-\sigma)}{2\sigma-1}}e^{\frac{8}{2\sigma-1}I(|\xi|^2)}\ud \xi }.
		\]
	\end{lem}
	\begin{pf}
		The idea is very simple: it suffices to choose
		$
		F(\xi)=1-|\xi|^2.
		$
		For $n=4$, we have $B^3_0(F)=f=0$ and $B^3_1(F)=\partial_{\nu}F$. Moreover, $F$ solves the boundary value problem
		\begin{align}\label{eq:ball-bvp-2}
			\begin{cases}
				(-\Delta)^2F=0 & \mathrm{in}~\ \B^4,\\[0.3em]
				F=0 & \mathrm{on}~\ \S^{3},\\[0.3em]
				\partial_{\nu}F=-2, & \mathrm{on}~\ \S^{3}.
			\end{cases}
		\end{align}
		Hence $F\notin \mathcal{H}_f$. On the other hand, since $F(\xi)=1-|\xi|^2>0$ in $\B^4$, it follows immediately that
		\begin{align*}
			\|e^{F}\|_{L^{\frac{4}{2\sigma-1}}(\B^4,\ud \nu)}	>1= \|e^{f}\|_{L^{3}(\S^{3},\ud \mu)}.
		\end{align*}
		This proves the claim.
	\end{pf}

	\textbf{Proof of Corollary \ref{cor:Huber_domain}}: 
Without loss of generality, we may assume that $\Omega$ is smooth; otherwise, we can approximate it by smooth domains.
		Let \(\Psi:\B^2\to\Omega\) be conformal. Since \(\partial\Omega\) is
		smooth, \(\Psi\) extends smoothly to \(\overline{\B^2}\) and
		\(\Psi'\neq0\) on \(\overline{\B^2}\). Define
		\[
		f(\xi):=F(\Psi(\xi))+\log|\Psi'(\xi)|,
		\qquad \xi\in\B^2 .
		\]
		Then \(f\in C^2(\B^2)\cap C(\overline{\B^2})\). Since
		\(\log|\Psi'|\) is harmonic, we have
		\[
		\Delta f(\xi)
		=
		|\Psi'(\xi)|^2(\Delta F)(\Psi(\xi)).
		\]
		Therefore, by the change of variables \(x=\Psi(\xi)\),
		\[
		\int_{\B^2}(-\Delta f)^+\,d\xi
		=
		\int_{\Omega}(-\Delta F)^+\,dx
		=
		\gamma
		<
		2\pi(2\sigma-1).
		\]
		Applying the disk version of the inequality to \(f\), we get
		\[
		\left(
		\frac{3-2\sigma}{\pi(2\sigma-1)}
		\int_{\B^2}
		e^{\frac{2}{2\sigma-1}f(\xi)}
		(1-|\xi|^2)^{\frac{4(1-\sigma)}{2\sigma-1}}
		\,d\xi
		\right)^{\frac{2\sigma-1}{2}}
		\le
		\frac{c_\sigma}{2\pi}
		\int_{\partial\B^2}e^f\,ds .
		\]
		The boundary term transforms as
		\[
		\int_{\partial\B^2}e^f\,ds
		=
		\int_{\partial\B^2}
		e^{F(\Psi(\xi))}|\Psi'(\xi)|\,ds_\xi
		=
		\int_{\partial\Omega}e^F\,ds .
		\]
	 Hence
		\[
		\begin{aligned}
			&\int_{\B^2}
			e^{\frac{2}{2\sigma-1}f(\xi)}(1-|\xi|^2)^{\frac{4(1-\sigma)}{2\sigma-1}}\,d\xi \\
			&=
			\int_{\B^2}
			e^{\frac{2}{2\sigma-1}F(\Psi(\xi))}
			|\Psi'(\xi)|^{\frac{2}{2\sigma-1}}
			(1-|\xi|^2)^{\frac{4(1-\sigma)}{2\sigma-1}}\,d\xi \\
			&=
			\int_{\B^2}
			e^{\frac{2}{2\sigma-1}F(\Psi(\xi))}
			\bigl((1-|\xi|^2)|\Psi'(\xi)|\bigr)^{\frac{4(1-\sigma)}{2\sigma-1}}
			|\Psi'(\xi)|^2\,d\xi .
		\end{aligned}
		\]
		By the conformal radius formula,
		\[
		r_\Omega(\Psi(\xi))=(1-|\xi|^2)|\Psi'(\xi)|,
		\]
		and therefore, after changing variables \(x=\Psi(\xi)\),
		\[
		\int_{\B^2}
		e^{\frac{2}{2\sigma-1}f(\xi)}(1-|\xi|^2)^{\frac{4(1-\sigma)}{2\sigma-1}}\,d\xi
		=
		\int_\Omega e^{\frac{2}{2\sigma-1}F(x)}r_\Omega(x)^{\frac{4(1-\sigma)}{2\sigma-1}}\,dx .
		\]
		This proves the desired inequality.

\section{Appendix}\label{Appendix}
		In this appendix, we aim to prove that every extremal function of \eqref{Intro Inequ} must be of class $C^1$. This regularity is a crucial ingredient in the moving sphere argument; see \eqref{Sec 4 MS equ-b}. The main idea is that, via the stereographic projection $\mathcal{S}$, it suffices to establish the regularity of extremal functions for the inequality \eqref{Sec 3 Equ Rn Lem euq}.
		
		Let $f\in L^{\infty}(\S^{n-1})$ be an extremal function of \eqref{Intro Inequ}, and set
		\[
		u(x')= f\circ \mathcal{S}(x',0)+(n-1)f_0.
		\]
		By Lemma \ref{Sec 3 Equ Rn Lem}, we know that $u$ satisfies the Euler--Lagrange equation associated with \eqref{Sec 3 Equ Rn Lem euq}, namely,
		\begin{align}\label{Sec 4 E-L equ}
			e^{u(x')}=c'_{\sigma}\int_{\R^{n}_{+}}e^{m_{\sigma}\mathcal{P}(u)(y)}y_n^{\frac{2n(1-\sigma)}{2\sigma-1}}\mathscr{P}(x'-y',y_n)\ud y.
		\end{align}
		Here
		\begin{align*}
			\mathscr{P}(x',x_n) =\left(\frac{x_n}{|x'|^2+x_n^2}\right)^{n-1}.
		\end{align*}
		Moreover, by \eqref{Sec 4 Pf0} and Remark \ref{Sec 2 rem}, we have
		\begin{align*}
			e^{m_{\sigma}\mathcal{P}(u)}=&e^{m_{\sigma}\mathcal{P}(f\circ \mathcal{S})}e^{m_{\sigma}(n-1)\left[\log |\xi +\mathbf{e}_{n}|^2 -2I\left(1\right)+2I(|\xi|^2)\right]\circ \mathcal{S}-m_{\sigma}(n-1)\log2}\\
			=&e^{m_{\sigma}\mathcal{P}(f\circ \mathcal{S})+2m_{\sigma}(n-1)( I(|\xi|^2)-I\left(1\right) )\circ \mathcal{S}} \left(\frac{2}{ (1+y_n)^2+|y'|^2}\right)^{\frac{n}{2\sigma-1}} .
		\end{align*}
		Since $f\in L^\infty(\mathbb S^{n-1})$ and \eqref{Sec 4 Def Pu} implies $\mathcal{P}(f\circ \mathcal{S})\in L^\infty(\R^{n}_{+})$, together with Remark \ref{Sec 2 rem} we obtain
		$$e^{m_{\sigma}\mathcal{P}(f\circ \mathcal{S})+2m_{\sigma}(n-1)( I(|\xi|^2)-I\left(1\right) )\circ \mathcal{S}}\in L^\infty(\mathbb R^{n}_{+}).$$
		Motivated by \eqref{Sec 4 E-L equ}, we therefore introduce the integral operator
		\begin{align}\label{Sec 4 T}
			T_{\sigma}(v)(x')=\int_{\R^{n}_{+}}\frac{v(y) y_n^{\frac{2n(1-\sigma)}{2\sigma-1}}}{((1+y_n)^2+|y'|^2)^{\frac{n}{2\sigma-1}}}\left(\frac{y_n}{|x'-y'|^2+y_n^2}\right)^{n-1}\ud y
		\end{align}
		for $v\in L^{\infty}(\R^{n}_{+})$ and $\sigma\in\left(\frac{1}{2},\frac{2n-1}{2(n-1)}\right)$. For convenience, set
		\[
		\alpha:=\frac{n}{2\sigma-1},
		\qquad
		\theta=\frac{2n(1-\sigma)}{2\sigma-1},
		\]
		and introduce the function
		\begin{align*}
			\mathscr{P}_{x_n}(x'):=\mathscr{P}(x',x_n) =\left(\frac{x_n}{|x'|^2+x_n^2}\right)^{n-1}.
		\end{align*}
		Then
		\[
		\mathscr{P}_{x_n}(x')=x_n^{-(n-1)}\mathscr{P}_{1}(x'/x_n).
		\]
		Since $\mathscr{P}_{1}\in W^{1,1}(\R^{n-1})$, by scaling we also have
		\[
		\|\nabla \mathscr{P}_{x_n}\|_{L^1(\R^{n-1})}
		=x_n^{-1}\|\nabla \mathscr{P}_{1}\|_{L^1(\R^{n-1})}
		\le \frac{C_1}{x_n}.
		\]
		Therefore, for every $h\in \R^{n-1}$,
		\begin{align}\label{Sec 3 P-ine}
			\|\mathscr{P}_{x_n}(\cdot+h)-\mathscr{P}_{x_n}\|_{L^1(\R^{n-1})}
			\le
			\min\Bigl\{2\|\mathscr{P}_{x_n}\|_{L^1},\,|h|\,\|\nabla \mathscr{P}_{x_n}\|_{L^1}\Bigr\}
			\le C\min\left\{1,\frac{|h|}{x_n}\right\}. 
		\end{align}
		We shall use this estimate repeatedly below. We first claim that $T_{\sigma}$ is well defined. Indeed, since
		\[
		\bigl((1+y_n)^2+|y'|^2\bigr)^\alpha\ge (1+y_n)^{2\alpha},
		\]
		we obtain
		\begin{align*}
			|T_\sigma(v)(x')|
			&\le
			\|v\|_{L^\infty(\R^n_+)}
			\int_0^\infty\int_{\R^{n-1}}
			\frac{y_n^{\theta}}{\bigl((1+y_n)^2+|y'|^2\bigr)^\alpha}
			\mathscr{P}_{y_n}(x'-y')\,\ud y'\ud y_n\\
			&\le
			\|v\|_{L^\infty(\R^n_+)}
			\int_0^\infty
			\frac{y_n^{\theta}}{(1+y_n)^{2\alpha}}
			\left(\int_{\R^{n-1}}\mathscr{P}_{y_n}(x'-y')\,\ud y'\right)\ud y_n\\
			&=
			C_0\|v\|_{L^\infty(\R^n_+)}
			\int_0^\infty \frac{y_n^{\theta}}{(1+y_n)^{2\alpha}}\,\ud y_n.
		\end{align*}
		Now,
		\[
		\theta=\frac{2n(1-\sigma)}{2\sigma-1}>-1
		\]
		because $\sigma<\frac{2n-1}{2(n-1)}$, while
		\[
		\theta-2\alpha
		=
		\frac{2n(1-\sigma)}{2\sigma-1}-\frac{2n}{2\sigma-1}
		=
		-\frac{2n\sigma}{2\sigma-1}
		<
		-1
		\]
		since $\sigma>\frac12$. Hence,
		\[
		\int_0^\infty \frac{y_n^{\theta}}{(1+y_n)^{2\alpha}}\,\ud y_n<\infty,
		\]
		and therefore $T_\sigma(v)$ is well defined and bounded on $\R^{n-1}$.

		\begin{pro}\label{prop:Tsigma_holder}
			Let $n\ge 2$ and $
			\theta:=\frac{2n(1-\sigma)}{2\sigma-1}>-1,
			$  where $\sigma\in \left(\frac12,\frac{2n-1}{2(n-1)}\right)$. Let $T_{\sigma}$ be defined by \eqref{Sec 4 T}, then 
			$T_\sigma(v)$ is well-defined on $\R^{n-1}$ and
			there exists a constant $C=C(n,\sigma)>0$ such that, for all $x',z'\in \R^{n-1}$,
			\[
			|T_\sigma(v)(x')-T_\sigma(v)(z')|
			\le C\|v\|_{L^\infty(\R^n_+)}
			\begin{cases}
				|x'-z'| & \theta>0,\\
				|x'-z'|\bigl(1+\bigl|\log |x'-z'|\bigr|\bigr), & \theta=0,\ |x'-z'|\le \frac12,\\[4pt]
				|x'-z'|^{\theta+1}, & -1<\theta<0.
			\end{cases}
			\]
			Moreover, if $\theta>0$, then
			\[
			|\nabla T_\sigma(v)(x')-\nabla T_\sigma(v)(z')|
			\le C\|v\|_{L^\infty(\R^n_+)}
			\begin{cases}
				|x'-z'| & \theta>1\\
				|x'-z'|\bigl(1+\bigl|\log |x'-z'|\bigr|\bigr), & \theta=1,\ |x'-z'|\le \frac12,\\[4pt]
				|x'-z'|^{\theta}, & 0<\theta<1.
			\end{cases}
			\]
		\end{pro}
		
		\begin{pf}
			
			Let $h\in \R^{n-1}$. Then
			\begin{align*}
				T_\sigma(v)(x'+h)-T_\sigma(v)(x')
				&=
				\int_0^\infty\int_{\R^{n-1}}
				\frac{v(y',y_n)y_n^{\theta}}{\bigl((1+y_n)^2+|y'|^2\bigr)^\alpha}
				\Bigl[\mathscr{P}_{y_n}(x'+h-y')-\mathscr{P}_{y_n}(x'-y')\Bigr]
				\,\ud y'\ud y_n.
			\end{align*}
			Taking absolute values and arguing as above,
			\begin{align*}
				&|T_\sigma(v)(x'+h)-T_\sigma(v)(x')|\\
				&\le
				\|v\|_{L^\infty(\R^n_+)}
				\int_0^\infty
				\frac{y_n^{\theta}}{(1+y_n)^{2\alpha}}
				\left(
				\int_{\R^{n-1}}
				|\mathscr{P}_{y_n}(x'+h-y')-\mathscr{P}_{y_n}(x'-y')|\,\ud y'
				\right)\ud y_n.
			\end{align*}
			By the change of variables,
			\[
			\int_{\R^{n-1}}
			|\mathscr{P}_{y_n}(x'+h-y')-\mathscr{P}_{y_n}(x'-y')|\,\ud y'
			=
			\|\mathscr{P}_{y_n}(\cdot+h)-\mathscr{P}_{y_n}\|_{L^1(\R^{n-1})}.
			\]
			Thus \eqref{Sec 3 P-ine} yields
			\begin{align*}
				|T_\sigma(v)(x'+h)-T_\sigma(v)(x')|
				&\le
				C\|v\|_{L^\infty(\R^n_+)}
				\int_0^\infty
				\frac{y_n^{\theta}}{(1+y_n)^{2\alpha}}
				\min\left\{1,\frac{|h|}{y_n}\right\}\,\ud y_n.
			\end{align*}
			Set
			\[
			I(h):=
			\int_0^\infty
			\frac{y_n^{\theta}}{(1+y_n)^{2\alpha}}
			\min\left\{1,\frac{|h|}{y_n}\right\}\,\ud y_n.
			\]
			We estimate $I(h)$ for $0<|h|\le 1$. Splitting the integral into three parts,
			\[
			I(h)
			=
			\int_0^{|h|}\frac{y_n^{\theta}}{(1+y_n)^{2\alpha}}\,\ud y_n
			+
			|h|\int_{|h|}^{1}\frac{y_n^{\theta-1}}{(1+y_n)^{2\alpha}}\,\ud y_n
			+
			|h|\int_{1}^{\infty}\frac{y_n^{\theta-1}}{(1+y_n)^{2\alpha}}\,\ud y_n.
			\]
			Since $(1+y_n)^{-2\alpha}\sim 1$ on $(0,1)$ and
			\[
			\theta-1-2\alpha
			=
			-\frac{2n\sigma}{2\sigma-1}-1<-1,
			\]
			the third integral is bounded by $C|h|$. Hence it remains to estimate the first two terms.
			
			\smallskip
			\noindent\emph{Case 1: $\theta>0$.}
			Then
			\[
			\int_0^{|h|} y_n^{\theta}\,\ud y_n\le C|h|^{\theta+1}\le C|h|,
			\qquad
			|h|\int_{|h|}^{1} y_n^{\theta-1}\,\ud y_n\le C|h|.
			\]
			Therefore
			\[
			I(h)\le C|h|,
			\]
			and so
			\[
			|T_\sigma(v)(x'+h)-T_\sigma(v)(x')|
			\le C\|v\|_{L^\infty(\R^n_+)}|h|.
			\]
			
			Moreover, in this case one can improve the above estimate to $C^{1,\alpha}$ regularity.
			Indeed, since
			\[
			\|\nabla_{x'}\mathscr{P}_{y_n}\|_{L^1(\R^{n-1})}
			=
			y_n^{-1}\|\nabla \mathscr{P}_{1}\|_{L^1(\R^{n-1})}
			\le C y_n^{-1},
			\]
			we have
			\[
			\int_0^\infty \frac{y_n^\theta}{(1+y_n)^{2\alpha}}
			\|\nabla_{x'}\mathscr{P}_{y_n}\|_{L^1(\R^{n-1})}\,\ud y_n
			\le
			C\int_0^\infty \frac{y_n^{\theta-1}}{(1+y_n)^{2\alpha}}\,\ud y_n<\infty,
			\]
			and therefore we may differentiate under the integral sign to obtain
			\[
			\nabla T_\sigma(v)(x')
			=
			\int_0^\infty\int_{\R^{n-1}}
			\frac{v(y',y_n)y_n^{\theta}}{\bigl((1+y_n)^2+|y'|^2\bigr)^\alpha}
			\nabla_{x'}\mathscr{P}_{y_n}(x'-y')
			\,\ud y'\ud y_n.
			\]
			Next, by scaling,
			\[
			\|D^2_{x'}\mathscr{P}_{y_n}\|_{L^1(\R^{n-1})}
			=
			y_n^{-2}\|D^2\mathscr{P}_{1}\|_{L^1(\R^{n-1})}
			\le C y_n^{-2}.
			\]
			Hence
			\begin{align*}
				\|\nabla \mathscr{P}_{y_n}(\cdot+h)-\nabla \mathscr{P}_{y_n}\|_{L^1(\R^{n-1})}
				&\le 2\|\nabla \mathscr{P}_{y_n}\|_{L^1(\R^{n-1})}
				\le C y_n^{-1},\\
				\|\nabla \mathscr{P}_{y_n}(\cdot+h)-\nabla \mathscr{P}_{y_n}\|_{L^1(\R^{n-1})}
				&\le |h|\,\|D^2 \mathscr{P}_{y_n}\|_{L^1(\R^{n-1})}
				\le C|h|\,y_n^{-2}.
			\end{align*}
			Combining the above two bounds, we get
			\[
			\|\nabla \mathscr{P}_{y_n}(\cdot+h)-\nabla \mathscr{P}_{y_n}\|_{L^1(\R^{n-1})}
			\le C y_n^{-1}\min\left\{1,\frac{|h|}{y_n}\right\}.
			\]
			Therefore,
			\begin{align*}
				&|\nabla T_\sigma(v)(x'+h)-\nabla T_\sigma(v)(x')|\\
				&\le
				\|v\|_{L^\infty(\R^n_+)}
				\int_0^\infty
				\frac{y_n^\theta}{(1+y_n)^{2\alpha}}
				\|\nabla \mathscr{P}_{y_n}(\cdot+h)-\nabla \mathscr{P}_{y_n}\|_{L^1(\R^{n-1})}
				\,\ud y_n\\
				&\le
				C\|v\|_{L^\infty(\R^n_+)}
				\int_0^\infty
				\frac{y_n^{\theta-1}}{(1+y_n)^{2\alpha}}
				\min\left\{1,\frac{|h|}{y_n}\right\}\,\ud y_n.
			\end{align*}
			Set
			\[
			J(h):=
			\int_0^\infty
			\frac{y_n^{\theta-1}}{(1+y_n)^{2\alpha}}
			\min\left\{1,\frac{|h|}{y_n}\right\}\,\ud y_n.
			\]
			As above, for $0<|h|\le 1$ we split
			\[
			J(h)
			=
			\int_0^{|h|}\frac{y_n^{\theta-1}}{(1+y_n)^{2\alpha}}\,\ud y_n
			+
			|h|\int_{|h|}^{1}\frac{y_n^{\theta-2}}{(1+y_n)^{2\alpha}}\,\ud y_n
			+
			|h|\int_{1}^{\infty}\frac{y_n^{\theta-2}}{(1+y_n)^{2\alpha}}\,\ud y_n.
			\]
			Since
			\[
			\theta-2-2\alpha
			=
			-\frac{2n\sigma}{2\sigma-1}-2<-1,
			\]
			the third integral is bounded by $C|h|$. For the first two terms, we distinguish three subcases.
			
			If $0<\theta<1$, then
			\[
			\int_0^{|h|} y_n^{\theta-1}\,\ud y_n\le C|h|^{\theta},
			\qquad
			|h|\int_{|h|}^{1} y_n^{\theta-2}\,\ud y_n\le C|h|^{\theta}.
			\]
			Hence
			\[
			J(h)\le C|h|^\theta.
			\]
			
			If $\theta=1$, then
			\[
			\int_0^{|h|}1\,\ud y_n=|h|,
			\qquad
			|h|\int_{|h|}^{1}\frac{1}{y_n}\,\ud y_n
			=
			|h|\log\frac1{|h|},
			\]
			and therefore
			\[
			J(h)\le C|h|\Bigl(1+\log\frac1{|h|}\Bigr).
			\]
			
			If $\theta>1$, then
			\[
			\int_0^{|h|} y_n^{\theta-1}\,\ud y_n\le C|h|^\theta\le C|h|,
			\qquad
			|h|\int_{|h|}^{1} y_n^{\theta-2}\,\ud y_n\le C|h|.
			\]
			Hence
			\[
			J(h)\le C|h|.
			\]
			
			Consequently,
			\[
			|\nabla T_\sigma(v)(x'+h)-\nabla T_\sigma(v)(x')|
			\le
			C\|v\|_{L^\infty(\R^n_+)}
			\begin{cases}
				|h|^\theta, & 0<\theta<1,\\[4pt]
				|h|\bigl(1+\bigl|\log|h|\bigr|\bigr), & \theta=1,\ |h|\le \frac12,\\[4pt]
				|h|, & \theta>1.
			\end{cases}
			\]
			Thus
			\[
			T_\sigma(v)\in C^{1,\alpha'}_{\loc}(\R^{n-1})
			\qquad \text{for every } \alpha'\in (0,\min\{1,\theta\}).
			\]
			
			\smallskip
			\noindent\emph{Case 2: $\theta=0$.}
			In this case,
			\[
			\int_0^{|h|}1\,\ud y_n=|h|,
			\qquad
			|h|\int_{|h|}^{1}\frac{1}{y_n}\,\ud y_n
			=
			|h|\log\frac1{|h|}.
			\]
			Hence
			\[
			I(h)\le C|h|\Bigl(1+\log \frac1{|h|}\Bigr),
			\]
			which implies
			\[
			|T_\sigma(v)(x'+h)-T_\sigma(v)(x')|
			\le
			C\|v\|_{L^\infty(\R^n_+)}
			|h|\Bigl(1+\log \frac1{|h|}\Bigr).
			\]
			Thus $T_\sigma(v)$ is log-Lipschitz, and in particular belongs to $C^{0,\alpha}_{\loc}$ for every $\alpha\in(0,1)$.
			
			\smallskip
			\noindent\emph{Case 3: $-1<\theta<0$.}
			Then
			\[
			\int_0^{|h|} y_n^{\theta}\,\ud y_n\le C|h|^{\theta+1},
			\qquad
			|h|\int_{|h|}^{1} y_n^{\theta-1}\,\ud y_n\le C|h|^{\theta+1}.
			\]
			Since $0<\theta+1<1$ and $0<|h|\le 1$, we also have $|h|\le |h|^{\theta+1}$. Therefore
			\[
			I(h)\le C|h|^{\theta+1},
			\]
			hence
			\[
			|T_\sigma(v)(x'+h)-T_\sigma(v)(x')|
			\le
			C\|v\|_{L^\infty(\R^n_+)}|h|^{\theta+1}.
			\]
			
			Combining the above three cases, we conclude that
			\[
			T_\sigma(v)\in C^{1,\alpha'}_{\loc}(\R^{n-1})
			\qquad \text{for every } \alpha'\in (0,\min\{1,\theta\})
			\quad \text{if } \theta>0,
			\]
			while
			\[
			T_\sigma(v)\in C^{0,\alpha'}_{\loc}(\R^{n-1})
			\quad\text{for every}\quad
			0<\alpha'<\min\{1,\theta+1\}
			\]
			if $-1<\theta\le 0$. Since
			\[
			\theta+1
			=
			\frac{2n-1-2(n-1)\sigma}{2\sigma-1}>0,
			\]
			the proof is complete.
		\end{pf}

		If $f\in L^{\infty}(\S^{n-1})$ is the extremal function, and recall 
		\begin{align}\label{Sec u-def}
			u(x')=f\circ \mathcal{S}(x',0)+(n-1)f_0, \qquad\mathrm{where}\qquad f_0(x')= \log\frac{2}{1+|x'|^2},
		\end{align}
		by \eqref{Sec 4 E-L equ} and Proposition \ref{prop:Tsigma_holder}, it follows that $u\in C^{\alpha'}_{\loc}(\mathbb R^{n-1})$ and so $f\in C_{\loc}^{\alpha'}(\S^{n-1}\backslash \{-\mathbf{e}_{n}\})$ for some $\alpha'\in(0,1)$. But choosing another projection point (for example $\mathbf{e}_{n} $), then it follows that $f\in C^{\alpha'}(\S^{n-1})$. From here and the property of the hyperbolic harmonic extension $P(f)$, we get $P(f)\in C^{\alpha'}(\overline{\B^n})$. Since  $P(f)\circ \mathcal{S}=\mathcal{P}(f\circ \mathcal{S})$, it follows that $\mathcal{P}(f\circ \mathcal{S})\in C^{\alpha'}(\overline{\R^{n}_{+}})$, therefore it follows that 
		\begin{align}\label{Sec 4 C alpha}
			e^{m_{\sigma}\mathcal{P}(f\circ \mathcal{S})+2m_{\sigma}(n-1)( I(|\xi|^2)-I\left(1\right) )\circ \mathcal{S}}\in C_{\loc}^{\alpha'}(\overline{\R_{+} ^{n}})
		\end{align}

		\begin{pro}\label{prop:local_gain}
			With the same assumption as Proposition \ref{prop:Tsigma_holder}, 
			if  \(v\in C_{\loc}^{\alpha'}(\overline{\R^n_+})\cap L^\infty(\R^n_+)\), for $\theta\in (-1,0]$, then the following statements hold.
			\begin{itemize}
				\item If \(0<\alpha'<1\) and \(\alpha'+\theta+1<1\), then
				\[
				T_\sigma(v)\in C^{\alpha'+\theta+1}_{\loc}(\R^{n-1}).
				\]
				
				\item If \(0<\alpha'<1\) and \(1<\alpha'+\theta+1<2\), then
				\[
				T_\sigma(v)\in C^{1,\alpha'+\theta}_{\loc}(\R^{n-1}).
				\]
			\end{itemize}
		\end{pro}
		
		\begin{pf}
			Fix \(R>0\), and let \(x'\in B'_R\subset \R^{n-1}\). Choose
			\[
			\chi(y',y_n)=\chi_1(y')\chi_2(y_n)\in C_c^\infty(\R^{n-1}\times[0,\infty))
			\]
			such that
			\[
			\chi_1\equiv 1 \ \text{on } B'_{3R/2}, \qquad \supp \chi_1\subset B'_{2R},
			\]
			and
			\[
			\chi_2\equiv 1 \ \text{on } [0,3R/2], \qquad \supp \chi_2\subset [0,2R].
			\]
			Write
			\[
			T_\sigma(v)=T_{\rm loc}(v)+T_{\rm far}(v),
			\]
			where
			\begin{align*}
				T_{\rm loc}(v)(x')
				&:=
				\int_0^\infty\int_{\R^{n-1}}
				\frac{\chi(y',y_n)\,v(y',y_n)\,  y_n^{\theta}}{\bigl((1+y_n)^2+|y'|^2\bigr)^\alpha}
				\left(\frac{y_n}{|x'-y'|^2+y_n^2}\right)^{n-1}
				\,\ud y'\ud y_n,
			\end{align*}
			and \(T_{\rm far}(v):=T_\sigma(v)-T_{\rm loc}(v)\).
			
			Since on the support of \(1-\chi\) one has either \(y_n\ge 3R/2\) or \(|y'|\ge 3R/2\), it follows that
			\[
			|x'-y'|+y_n\ge c_R>0
			\qquad \text{for all } x'\in B'_R .
			\]
			Hence the kernel of \(T_{\rm far}(v)\) is smooth in \(x'\), and differentiation under the integral sign gives
			\[
			T_{\rm far}(v)\in C^\infty(B'_R).
			\]
			
			It remains to study \(T_{\rm loc}(v)\). Set
			\[
			b(y',y_n):=\frac{1}{\bigl((1+y_n)^2+|y'|^2\bigr)^\alpha},
			\qquad
			\mathscr{P}_{y_n}(z):=\left(\frac{y_n}{|z|^2+y_n^2}\right)^{n-1}.
			\]
			Then
			\[
			T_{\rm loc}(v)(x')
			=
			\int_0^{2R}\int_{\R^{n-1}}
			\chi(y',y_n)\,v(y',y_n)\,b(y',y_n)\,y_n^{\theta}\mathscr{P}_{y_n}(x'-y')\,\ud y'\ud y_n.
			\]
			
			We decompose
			\[
			\chi(y',y_n)\,v(y',y_n)\,b(y',y_n)
			=
			R_1(x',y',y_n)+R_2(x',y',y_n),
			\]
			where
			\[
			R_1(y',y_n):=\chi(y',y_n)\bigl(v(y',y_n)b(y',y_n)-v(y',0)b(y',0)\bigr),
			\]
			and
			\[
			R_2(y',y_n):=\chi(y',y_n)v(y',0)b(y',0).
			\]
			Accordingly,
			\[
			T_{\rm loc}(v)=E_1+E_2,
			\]
			with
			\begin{align*}
				E_1(x')
				&:=
				\int_0^{2R}\int_{\R^{n-1}}
				R_1(y',y_n)\,y_n^{\theta}\mathscr{P}_{y_n}(x'-y')\,\ud y'\ud y_n,\\
				E_2(x')
				&:=
				\int_0^{2R}\int_{\R^{n-1}}
				R_2(y',y_n)\,y_n^{\theta}\mathscr{P}_{y_n}(x'-y')\,\ud y'\ud y_n.
			\end{align*}

			\medskip
			\noindent\emph{Step 1: the error term \(E_1\).}
			Because \(v\in C^{\alpha'}(\overline{\R^n_+})\), we have
			\[
			|v(y',y_n)b(y',y_n)-v(y',0)b(y',0)|\le C_R y_n^{\alpha'}
			\qquad \text{on } \supp \chi .
			\]
			Since \(b\) is bounded on \(\supp \chi\), it follows that
			\[
			|R_1(y',y_n)|\le C_R y_n^{\alpha'} .
			\]
			Hence \(E_1\) has the same form as \(T_\sigma\), but with \(y_n^{\theta}\) replaced by \(y_n^{\alpha'+\theta}\). Consequently its kernel is of order \(\alpha'+\theta+1\), and the same argument as Proposition \ref{prop:Tsigma_holder} gives
			\[
			E_1\in
			\begin{cases}
				C^{\alpha'+\theta+1}(B'_R), & \alpha'+\theta+1<1,\\[3pt]
				C^{1,\alpha'+\theta}(B'_R), & 1<\alpha'+\theta+1<2.
			\end{cases}
			\]

			\medskip
			\noindent\emph{Step 2: the principal term \(E_2\).}
			Define
			\[
			J_R(z):=\int_0^{2R} \chi_2(y_n)y_n^{\theta}\mathscr{P}_{y_n}(z)\,\ud y_n
			=\int_0^{2R}\frac{y_n^{n+\theta-1}\chi_2(y_n)}{(|z|^2+y_n^2)^{n-1}}\,\ud y_n.
			\]
			If \(z\neq 0\), the change of variables \(y_n=|z|\tau\) yields
			\[
			J_R(z)
			=
			|z|^{\theta+2-n}\int_0^{2R/|z|}
			\frac{\tau^{n+\theta-1}\chi_2(|z|\tau)}{(1+\tau^2)^{n-1}}\,\ud \tau.
			\]
			Hence
			\[
			J_R(z)=\begin{cases}
				c_{n,\theta}|z|^{\theta+2-n}+H_R(z)& n\not=2, \theta\in (-1,0], n=2, \theta\in (-1,0)\\
				c_{2,0}\log|z|+H_R(z)&n=2,\theta=0,
			\end{cases}
			\]
			where \(c_{n,\theta}>0\) and \(H_R\in C^\infty(B'_{4R})\). Therefore, set
			\[
			f(y'):=\chi_1(y')v(y',0)b(y',0)\in C^{\alpha'}_c(\R^{n-1}),
			\]
			If $n\not=2, \theta\in (-1,0]$ and $n=2, \theta\in (-1,0)$, we may write
			\[
			E_2(x')=\,c_{n,\theta}\int_{\R^{n-1}}\frac{f(y')}{|x'-y'|^{n-2-\theta}}\,\ud y'
			+\int_{\R^{n-1}}f(y')H_R(x'-y')\,\ud y'.
			\]
			The second term is smooth. For the first one, the classical Schauder estimate for the Riesz potential implies
			\[
			I_\theta(f)(x'):=\int_{\R^{n-1}}\frac{f(y')}{|x'-y'|^{n-2-\theta}}\,\ud y'
			\in
			\begin{cases}
				C^{\alpha'+\theta+1}(B'_R), & \alpha'+\theta+1<1,\\[3pt]
				C^{1,\alpha'+\theta}(B'_R), & 1<\alpha'+\theta+1<2.
			\end{cases}
			\]
			For $n=2$ and $\theta=0$, the same arguments are still valid.

			Combining the above decomposition,
			we conclude that
			\[
			T_\sigma(v)\in
			\begin{cases}
				C^{\alpha'+\theta+1}(B'_R), & \alpha'+\theta+1<1,\\[3pt]
				C^{1,\alpha'+\theta}(B'_R), & 1<\alpha'+\theta+1<2.
			\end{cases}
			\]
			Since \(R>0\) is arbitrary, the desired local statement follows.
		\end{pf}

		\begin{thm}\label{cor:bootstrap_C1}
			Assume that $f$ is an extremal function and let $u$ be defined by \eqref{Sec u-def}, then $u$ satisfies the integral equation
			\begin{align}\label{Sec 4 E-L equ-1}
				e^{u(x')}=c'_{\sigma}\int_{\R^{n}_{+}}e^{m_{\sigma}\mathcal{P}(u)(y)}y_n^{\frac{2n(1-\sigma)}{2\sigma-1}}\mathscr{P}(x'-y',y_n)\ud y    \end{align}
			and $u\in C^1_{\loc}(\R^{n-1})$.
		\end{thm}
		
		\begin{pf}
			If $\theta>0$, by Proposition \ref{prop:Tsigma_holder}, the conclusion is valid. Hence, we can assume $\theta\in (-1,0]$. Choosing $1-\alpha_0\not\in  (\theta+1)\mathbb{N}$ and 
			\[
			0<\alpha_0<\min\left\{1,\frac{2n-1-2(n-1)\sigma}{2\sigma-1}\right\},
			\]
			then by Proposition \ref{prop:Tsigma_holder}, it follows that $u\in C^{\alpha_0}_{\loc}(\R^{n-1})$, using \eqref{Sec 4 C alpha}, it follows that $$e^{m_{\sigma}\mathcal{P}(f\circ \mathcal{S})+2m_{\sigma}(n-1)( I(|\xi|^2)-I\left(1\right) )\circ \mathcal{S}}\in C_{\loc}^{\alpha_0}(\overline{\R^{n}_{+}}) .$$ By Proposition \ref{prop:local_gain}, if $\alpha_0+\theta+1<1$, it follows that $u\in C^{\alpha_0+\theta+1}_{\loc}(\R^{n-1})$. (If $\alpha_0+\theta+1>1$, then we are done.)
			Repeatedly applying Proposition~\ref{prop:local_gain}, we obtain
			\[
			u\in C^{\alpha_0+m(\theta+1)}_{\loc}(\R^{n-1})
			\qquad \text{for every integer } m\ge 1 \text{ with } \alpha_0+m(\theta+1)<1.
			\]
			Choose \(m_0\in \N\) such that
			\[
			\alpha_0+(m_0+1)(\theta+1)>1>\alpha_0+m_0(\theta+1).
			\]
			Then one more application of Proposition~\ref{prop:local_gain} yields
			
			\[
			u\in C^{1,\alpha_0+(m_0+1)(\theta+1)-1}_{\loc}(\R^{n-1}).
			\]
			This proves the claim.
		\end{pf}

\noindent{\textbf{Acknowledgments.} The research of Z. Chen is supported by National Key R\&D Program of China (Grant 2023YFA1010002) and NSFC (No. 12222109).
	The research of C. Gui is supported by University of Macau research grants
	CPG2024-00016-FST, CPG202500032-FST, CPSRG20263-0002711-FST,  MYRG-GRG2023-
	00139-FST-UMDF, UMDF Professorial Fellowship of Mathematics, Macao SAR
	FDCT0003/2023/RIA1 and Macao SAR FDCT0024/2023/RIB1, NSFC No. 12531010. The research of S. Zhang was supported by the
	Postdoctoral Fellowship Program and China Postdoctoral Science Foundation (Grant BX20250062), as well as the Shui Mu Tsinghua Scholar Program.}   \\
            
\noindent{\textbf{AI Assistance Statement:}
The authors used AI tools only for language editing and for checking the correctness of some identities. All arguments, identities, estimates, and the final manuscript were independently verified and remain the sole responsibility of the authors.}
                
		\bibliographystyle{unsrt}

		\bigskip
	
\noindent Department of Mathematical Sciences, Yau Mathematical Sciences Center, Tsinghua University, Beijing, China \\[1mm]
		Email: \textsf{zjchen2016@tsinghua.edu.cn}	\\[3mm]

		\noindent Department of Mathematics, University of Macau,Taipa, Macau \\[1mm]
		Email: \textsf{changfenggui@um.edu.mo}	\\[3mm]	

\noindent Yau Mathematical Sciences Center, Tsinghua University, Beijing, China \\[1mm]
		Email: \textsf{ shihong-zhang@tsinghua.edu.cn.}
        
		\medskip  	
}		
	\end{document}